\documentclass[11pt,reqno,oneside]{amsart}
\usepackage[asymmetric,top=3.5cm,bottom=4.0cm,left=3.1cm,right=3.1cm]{geometry}
\usepackage{amsmath,amsfonts,amsthm,mathrsfs,amssymb, cite,array}
\usepackage[usenames]{color}
\usepackage{mathtools} %for \mathclap, \mathrlap, \smashoperator etc
\usepackage{framed}
\usepackage{enumerate}
\usepackage{graphicx}      % 必备图片支持
\usepackage{caption}% 主标题格式控制（可选，若用 subcaption 则替换）
\usepackage{subcaption}
\usepackage{hyperref}
\usepackage{xcolor}
\hypersetup{
  colorlinks,
  linkcolor={blue!80!black},
  urlcolor={blue!80!black},
  citecolor={blue!80!black}
}
\usepackage[capitalize,noabbrev]{cleveref}
\usepackage{graphicx}
\usepackage{latexsym}
\usepackage{enumerate}
\usepackage{booktabs}
\usepackage{adjustbox}
\usepackage{multirow}

\newtheorem{definition}{Definition}[section]

\newtheorem{thm}{Theorem}[section]

\newtheorem{Lemma}[thm]{Lemma}

\numberwithin{equation}{section}
\newtheorem{remark}[thm]{Remark}

\newcommand{\Bx}{\mathbf{x}}
\newcommand{\Bi}{\mathbf{i}}

\newcommand{\BF}{\mathbf{F}}

\newcommand{\Bu}{\mathbf{u}}

\def\Oh{{\mathcal  O}}

\author{}
\title[NL balance laws in product spaces]{Determining nonlinear balance laws in product-type domains by a single local passive boundary observation}

\author{Chaohua Duan}
\address{Department of Mathematics, City University of Hong Kong, Hong Kong SAR, China}
\email{chduan3-c@my.cityu.edu.hk, chduan3@gmail.com}

\author{Hongyu Liu}
\address{Department of Mathematics, City University of Hong Kong, Hong Kong SAR, China}
\email{hongyu.liuip@gmail.com, hongyliu@cityu.edu.hk}

\author{Qingle Meng}
\address{Department of Mathematics, City University of Hong Kong, Hong Kong SAR, China}
\email{mengq12021@foxmail.com, qinmeng@cityu.edu.hk}

\author{Li Wang}
\address{Department of Mathematics, City University of Hong Kong, Hong Kong SAR, China}
\email{lwang637-c@my.cityu.edu.hk, liwangmath12@126.com}

\allowdisplaybreaks

\begin{document}

%\maketitle
\begin{abstract}  

This paper introduces an operator-theoretic paradigm for solving inverse problems in nonlinear balance laws, shifting the focus from identifying specific functional forms to recovering the input-output actions of the associated flux and source operators. It is established that a single local passive boundary observation suffices to uniquely determine realizations of these operators for systems posed on product-type domains. This framework, which encompasses dynamical regimes, reveals a holographic-type principle where macroscopic boundary data encodes microscopic dynamical information, with broad implications for fluid dynamics and reaction-diffusion systems.

       \medskip
    
    
    
\end{abstract}

\maketitle

\section{Introduction}\label{sec:intro}
\subsection{Mathematical setup}\label{sub:math}
Nonlinear balance laws provide a foundational framework in mathematical physics, describing the evolution and spatial distribution of conserved quantities in systems ranging from fluid dynamics and reaction-diffusion processes to quantum mechanics; see Section~\ref{subsec:RB-TI} in what follows for more related discussion. This class of problems is generally expressed as:
\begin{equation}\label{eq:bl}
\begin{cases}
    \partial_t H(\Bu) + \nabla_{\Bx} \cdot \BF(\Bx,t,\Bu) = f(\Bx,t,\Bu,\nabla \Bu,\Delta \Bu) &\  \text{in}  \ \ Q:=\Omega \times [0,T],\\
    \Bu(\Bx,t) = \boldsymbol{\psi}(\Bx,t) & \  \text{on} \ \ \widetilde{\Gamma}:=\partial \Omega \times [0,T],
\end{cases}
\end{equation}
where $\Omega$ is a bounded Lipschitz domain in $\mathbb{R}^n$, $T \in (0, \infty)$, the state variable $\Bu (\Bx, t): \mathbb{R}^n\times\mathbb{R}_+\to\mathbb{R}^m$ and the nonlinear functions are defined as $H: \mathbb{R}^m \to \mathbb{R}^m$, $\BF: \Omega \times [0, T] \times \mathbb{R}^m \to \mathbb{R}^{n \times m}$, and $f: \Omega \times [0, T] \times \mathbb{R}^m \times \mathbb{R}^{n \times m} \times \mathbb{R}^m \to \mathbb{R}^m$, with $m, n\geq 1$. The term $H(\Bu)$ represents the conserved quantity, a function of the state variables $\Bu$. Its partial derivative with respect to time, $\partial_t H(\Bu)$, describes this quantity's temporal rate of change. This temporal evolution is balanced by the term $\nabla_{\Bx} \cdot \BF(\Bx, t, \Bu)$, which represents the divergence of the flux. The flux $\BF$ governs the spatial flow of the conserved quantity within the system and depends on the spatial coordinate $x$, time $t$, and the system’s state $\Bu$. Additionally, $f(\Bx, t, \Bu, \nabla \Bu, \Delta \Bu)$ denotes a source or sink term, capturing external influences that may contribute to or diminish the conserved quantity. 

In this paper, we study the inverse problem for \eqref{eq:bl} of recovering the functions $\BF$ and $f$ from knowledge of the field $\Bu (\Bx, t)$ on the boundary segment $\Sigma:=\Gamma \times (T_1,T_2)\subset \widetilde{\Gamma}$ where $\Gamma\subset\partial\Omega$ and $0\leq T_1<T_2\leq  T$. Such a problem arises naturally when a physical process in $\Omega$ generates the state $\Bu$, and its boundary observation is recorded in the dataset:
\begin{equation*}%\label{eq:map}
    \mathcal{M}_{\BF,f}|_{\Sigma} = \left( \left. \Bu \right|_{\Sigma} , \left.\partial_\nu \Bu(\Bx, t)\right|_{\Sigma} + h_{\BF}(\Bx,t) \right).
\end{equation*}
The term $h_{\BF}(\Bx,t)$ embodies a boundary coupling effect, capturing the additional flux at the boundary due to non-diffusive processes (e.g., convection and nonlinear transport) and is directly tied to the flux function $\BF$. It will be explicitly defined in our subsequent analysis. For a fixed physical state $\Bu$, the dataset $\mathcal{M}_{\BF, f}$ thus represents a single, local, passive boundary observation, and our central inverse problem is to determine $(\BF, f)$ from this data:
\begin{equation}\label{eq:ip1}
\mathcal{M}_{\BF, f}|_{\Sigma} \longrightarrow (\BF, f).
\end{equation}
This problem presents significant challenges and has seen very little progress in the literature; we provide a detailed discussion in Section~\ref{subsec:RB-TI}. A key novelty of our work lies in reframing the objective: instead of recovering the function forms of $\BF$ and $f$ themselves, we aim to recover the operators that describe their overall impact on the system's behavior. This shift to an operator-theoretic standpoint, set in suitable Sobolev spaces, constitutes a more general and entirely new framework for tackling inverse problems for partial differential equations (PDEs).
 
Throughout this paper, we fix $m=1$ and $n=2, 3$, restricting our study to the case of a scalar state $u$ in \eqref{eq:bl}. Nevertheless, as detailed in Table~\ref{table1}, this includes a large class of important PDEs in mathematical physics and applied mathematics. Per our discussion above, we define the following operators associated with $\BF$ and $f$:
\begin{align}
\mathscr{F}: H^1(\Omega\times[0,T]) & \rightarrow L^2(\Omega\times[0,T])^n, \notag \\
u & \mapsto \mathscr{F}(u)(\Bx,t):=\BF(\Bx,t, u), \quad (\Bx,t) \in Q, \notag \\
\mathfrak{f}: H^1(\Omega\times[0,T]) & \rightarrow L^2(\Omega\times[0,T]), \notag\\
u & \mapsto \mathfrak{f}(u)(\Bx,t):=f(\Bx,t, u,\nabla u,\Delta u), \quad (\Bx,t) \in Q,\label{eq:oper}
\end{align}
where $u(\Bx)$ is regarded as a sample in $H^1(\Omega\times[0,T])$. The operator $\mathscr{F}$ is vector-valued, $\mathscr{F} = (\mathscr{F}_k)_{k=1}^n$, and $\mathscr{F}(u)(\Bx,t):=\BF(\Bx,t, u)$ and $\mathfrak{f}(u)(\Bx,t):=f(\Bx,t, u,\nabla u,\Delta u)$ correspond to realizations of the operators $\mathscr{F}$ and $\mathfrak{f}$ in $L^2(\Omega\times[0,T])$, respectively.

With the operator-theoretic perspective, the inverse problem \eqref{eq:ip1} can be recast as follows:
\begin{equation}\label{eq:main_IP}
\mathcal{M}_{\mathscr{F},\mathfrak{f}} \big|_\Sigma \longrightarrow (\mathscr{F}, \mathfrak{f}).
\end{equation}
 In the inverse problems \eqref{eq:main_IP}, our primary concern is to tackle the theoretical challenge of achieving unique identifiability, which can be expressed as:
 \begin{equation*}%\label{eq:main_Uniq}
\mathcal{M}_{\mathscr{F}^1, \mathfrak{f}^1} = \mathcal{M}_{\mathscr{F}^2, \mathfrak{f}^2}  \Longleftrightarrow  \mathscr{F}^1(u^j) = \mathscr{F}^2(u^j)\, \text{and} \,\,\, \mathfrak{f}^1(u^j) = \mathfrak{f}^2(u^j),\ \ j=1,2, 
\end{equation*}
where $(\mathscr{F}^j, \mathfrak{f}^j)$ are two configurations  with corresponding solutions $u^j$ ($j = 1, 2$). It should be noted that the flow balance condition dictates that the equality $\mathcal{M}_{\mathscr{F}^1, \mathfrak{f}^1} = \mathcal{M}_{\mathscr{F}^2, \mathfrak{f}^2}$ requires $h_{\BF^1} - h_{\BF^2}$ to satisfy an associated flux balance condition. The specific form of this condition, inherent to equation \eqref{eq:bl}, is derived in Section \ref{sec:result}.

We emphasize that $\mathscr{F}^1(u^j) = \mathscr{F}^2(u^j)$ and $\mathfrak{f}^1(u^j) = \mathfrak{f}^2(u^j)$ for $j = 1, 2$ do not imply that $\mathscr{F}^1 = \mathscr{F}^2$ and $\mathfrak{f}^1 = \mathfrak{f}^2$. In the sense of the effects produced by the associated operators $\mathscr{F}^j$ and $\mathfrak{f}^j$ for the passive measurements, the corresponding operators can be regarded as determined uniquely. We focus on the effects and impacts of the operators on various passive observations rather than the explicit forms of the operators $\mathscr{F}$ and $\mathfrak{f}$ themselves. To determine the explicit forms of the operators $\mathscr{F}$ and $\mathfrak{f}$, one typically requires either a priori information (e.g., that $\BF$ and $f$ are polynomials) or the system's response to multiple inputs through active measurements.
% In this work, we establish that, under passive observation of a single solution, the boundary measurement mapping $\mathcal{M}_{\mathscr{F},\mathfrak{f}}$ uniquely determines the action of these operators on that specific solution. Proving this uniqueness result is challenging, particularly when attempting partial recovery of either $\mathscr{F}$ or $\mathfrak{f}$ without prior information about the other.

% We emphasize that $\mathscr{F}^1(u^j) = \mathscr{F}^2(u^j)$ and $\mathfrak{f}^1(u^j) = \mathfrak{f}^2(u^j)$ for $j = 1, 2$ does not imply that the operators $\mathscr{F}^1$ and $\mathscr{F}^2$ or the functions $\mathfrak{f}^1$ and $\mathfrak{f}^2$ are identical. Instead, it indicates that the boundary measurements $\mathcal{M}_{\mathscr{F}^1, \mathfrak{f}^1}$ and $\mathcal{M}_{\mathscr{F}^2, \mathfrak{f}^2}$ coincide on $\Sigma$, the effective actions of $\mathscr{F}^1$ and $\mathscr{F}^2$, as well as $\mathfrak{f}^1$ and $\mathfrak{f}^2$, on their respective solutions $u^1$ and $u^2$ yield identical outcomes.

%Even so, the uniqueness result in \eqref{main:Uniq} is not easy to obtain for these two configurations. Even if only part of the configuration is recovered, such as the parameter $\mathscr{F}(u)$, it is quite difficult to obtain the uniqueness without any prior information about $\mathfrak{f}$. 

\subsection{Geometric setup} \label{sub:geometric}
In this subsection, we introduce the specific geometric setup for our study and outline the main results. We assume that the domain \(\Omega\) is possesses a certain product-type feature; refer to Fig. \ref{fig:TNe} for a schematic illustration of these two types of domains. Such geometric feature is critical in establishing our global recovery results.

\begin{definition}[Nozzle domain]\label{def:ND}
A domain \(\mathscr{N}^{N}_{\varepsilon} \subset \mathbb{R}^n\) $(n=2,3)$ is called a \textit{nozzle domain} if it is described as follows:  
Let \(\Omega^{N}_{\varepsilon}\) be a bounded, simply connected Lipschitz domain in \(\mathbb{R}^{n-1}\) with \(\mathrm{diam}(\Omega^{N}_{\varepsilon}) \leq \varepsilon \ll 1\). Let \(\gamma \in C^2(I; \mathbb{R}^n)\) be a simple (injective) curve parametrized over the interval \(I = (-L, L)\) for some \(L > 0\).  
The nozzle domain is constructed by parallel transporting \(\Omega^{N}_{\varepsilon}\) along the curve \(\gamma\), forming the domain:
\[
\mathscr{N}^{N}_{\varepsilon} := \bigcup_{a \in I} \Omega^{N}_{\varepsilon}(a) \times \{\gamma(a)\}, 
\]
where \(\Omega^{N}_{\varepsilon}(a)\) represents the translated cross-section at the point \(a\), and there exists a unique \(a_0 \in I\) such that \(\gamma(a_0) \in \Omega^{N}_{\varepsilon}(a_0)\).
\end{definition}
%Similarly, narrow end $\mathscr{N}^{S}_{\varepsilon}$ is defined as follows:
\begin{definition}[Slab domain]\label{def:SD}
The \textit{slab domain} $\mathscr{N}^{S}_{\varepsilon} \subset \mathbb{R}^3$ is constructed analogously to the nozzle domain but with quadrilateral cross-sections:
\begin{equation*}%\label{eq:NE}
    \mathscr{N}^{S}_{\varepsilon} := \bigcup_{a \in I} \Omega^{S}_{\varepsilon}(a) \times \{\gamma(a)\} \subset \mathbb{R}^3,
\end{equation*}
where $\gamma:I\to\mathbb{R}^3$ is the same $C^2$ simple curve as in Definition~\ref{def:ND}, and each $\Omega^{S}_{\varepsilon}(a)$ is a translation of a fixed quadrilateral $\Omega^{S}_{\varepsilon}\subset\mathbb{R}^2$ characterized by sufficiently small width $\varepsilon\ll 1$ and conventional length $\Oh(1)$.
\end{definition}
\begin{remark}\label{Re:TN}
It is important to note that the nozzle domain features two sufficiently small dimensions, while the slab domain has a thickness of $\varepsilon$, significantly smaller than its length and width. Here, $\varepsilon$ as the geometric parameter characterizes these product-type domains and plays a crucial role in our subsequent analysis,  particularly in understanding how boundary measurements encode and recover information about the interior dynamics. Moreover, these two geometric structures are physically relevant in various applications, such as nanoscale materials and microfluidic channels, see \cite{ISKW,MTT}.
\end{remark}
% \begin{remark}\label{RE:App}
% The geometric decomposition in~\eqref{eq:pd} naturally captures thin or narrow structures that are physically relevant in many applications, such as nanoscale materials, microfluidic channels, or biological membranes. These singular geometric features play a crucial role in the analysis of the system~\eqref{eq:bl}, particularly in understanding how boundary measurements encode and recover information about the interior dynamics.
% \end{remark}

%where:
%\begin{itemize}
%    \item $\Omega^{S}_{\varepsilon} \subset \mathbb{R}^2$ is a quadrilateral cross-section with:
 %   \begin{itemize}
 %       \item Width: $\varepsilon$ (with $\varepsilon \ll 1$)
 %       \item Length: $\Oh(1)$ (conventional dimension)
  %  \end{itemize}
 %   \item $\gamma(t)$ is the same $C^2$ simple curve as in Definition~\ref{def:ND}
    %\item The symbol "$\times$" denotes parallel transport of the cross-section along $\gamma(t)$
%\end{itemize}
%\end{definition}

To formalize the geometric framework of our study, we consider a domain $\Omega \subset \mathbb{R}^{n}$ which is defined as the union of multiple pairwise disjoint nozzle or slab domains:
\begin{equation}\label{eq:NNN}
\Omega = \bigcup_{\ell=1}^{M} \mathscr{N}_{\varepsilon}^{\ell},
\end{equation}
where each component $\mathscr{N}_{\varepsilon}^{\ell}$ represents either a nozzle domain or a slab domain, denoted as:
\begin{equation}
\label{eq:Ne}
\mathscr{N}_{\varepsilon}^{\ell} := \mathscr{N}^{N}_{\varepsilon} \text{ or } \mathscr{N}^{S}_{\varepsilon}.
\end{equation}

The corresponding lateral boundary is given by:
\begin{equation}\label{def:Gamma}
    \hat{\Gamma} := \partial\Omega_{\varepsilon} \times \gamma(I) \quad \text{for } \Omega_{\varepsilon} = \Omega^{N}_{\varepsilon} \text{ or } \Omega^{S}_{\varepsilon},
\end{equation}
where \(\gamma(I) = \{\gamma(a) : a \in I\}\) is a curve introduced in Definitions \ref{def:ND} and \ref{def:SD}. %Furthermore, the overall lateral boundary \(\Gamma_{\varepsilon}\) is the union of the individual boundaries of all ends, satisfying \(\Gamma_{\varepsilon} \subset \partial\Omega\) with \(C^{2}\) regularity. 
A schematic illustration of this configuration is provided in Figure~\ref{fig:TNe}.

\begin{figure}[htbp]
    \centering
     \begin{subfigure}[b]{0.3\textwidth}  % 增加宽度
        \centering
        \includegraphics[width=\linewidth]{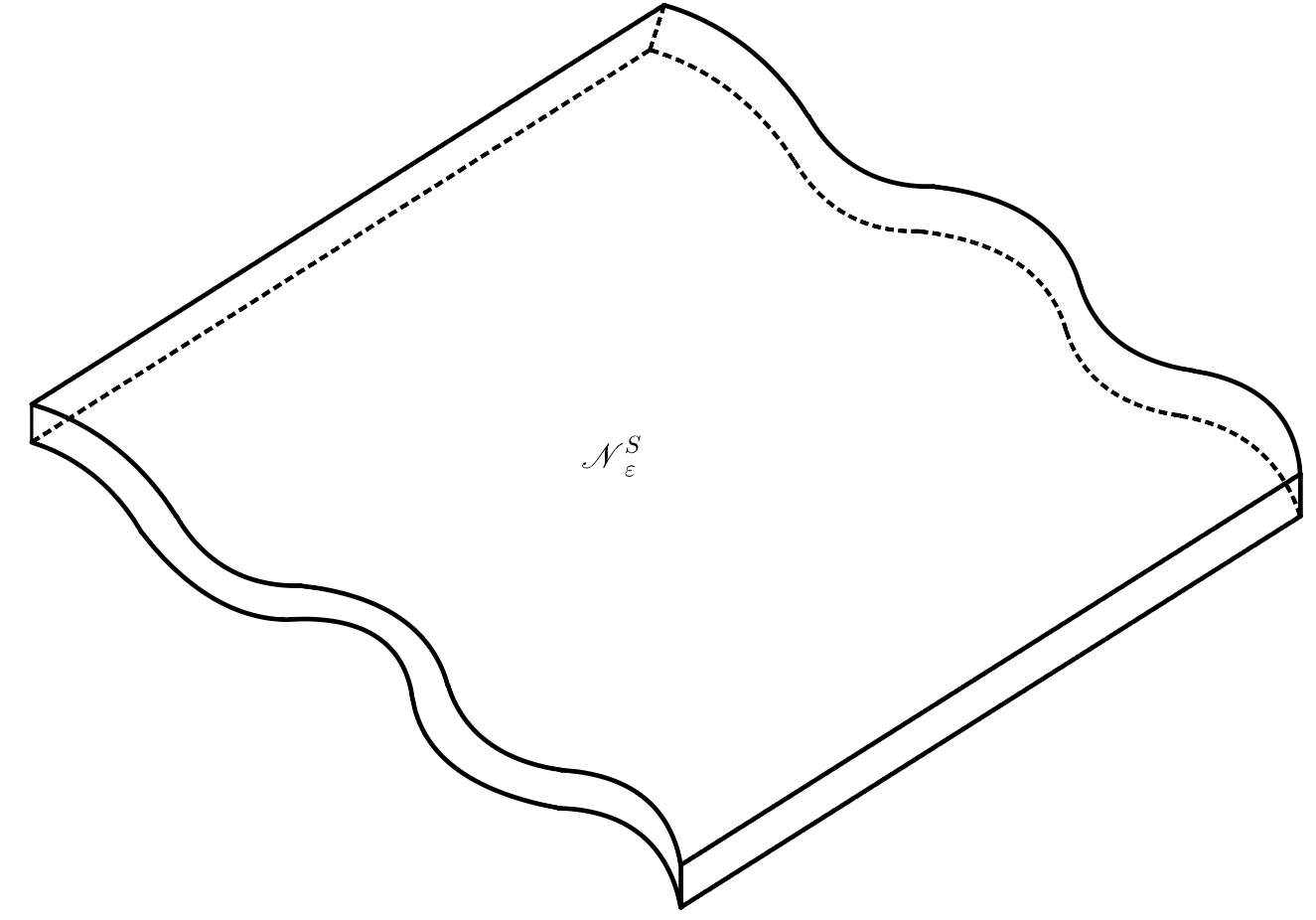}
        \caption{A single slab domain}
    \end{subfigure}
     \hspace{10mm}  
    \begin{subfigure}[b]{0.38\textwidth} 
        \centering
        \includegraphics[width=\linewidth]{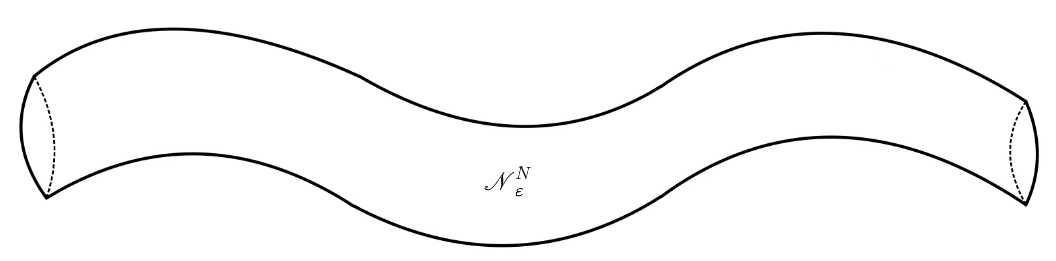}
        \caption{A single nozzle domain}
    \end{subfigure}
    \caption{Schematic illustration of the nozzle and slab domains in $\mathbb{R}^{3}$}
    \label{fig:TNe}
\end{figure}

% Throughout this paper, the parameter $\varepsilon >0$ will be referred to as the geometric parameter characterizing the asymptotic regime of either the cross-sectional diameter for thin ends $\mathscr{N}^{N}_{\varepsilon}$ or the thickness dimension for narrow ends $\mathscr{N}^{S}_{\varepsilon}$.

We now roughly state the main result of the inverse problem outlined in equation~\eqref{eq:main_IP}. This result establishes that, under generic conditions, the operators \(\mathscr{F}\) and \(\mathfrak{f}\), associated with the constitutive functions \(\BF\) and \(f\) in the system \eqref{eq:bl}, can be uniquely and approximately recovered from a single passive measurement \(\mathcal{M}_{\mathscr{F},\mathfrak{f}}\) on the lateral boundary $\hat{\Gamma}\times [0,T]$ defined in \eqref{def:Gamma}. %The more detailed description can be found in Section \ref{sec:result}. 

\begin{thm}\label{main:thm0}
Let \(\Omega\) be a product-type domain defined in \eqref{eq:NNN}. Assume that \(\BF^j\) and \(f^j\) belong to the admissible class \(\mathcal{A}\), with the corresponding solution \(u^j\) in the system \eqref{eq:bl} and the associated operators \(\mathscr{F}^j\) and \(\mathfrak{f}^j\) for \(j=1,2\). Let \(\Sigma \subset \widetilde{\Gamma}\), and let \(\mathcal{S}\) be an appropriately chosen function space defined on \(\Sigma\). If
\begin{equation*}
    \mathcal{M}_{\mathscr{F}^1,\mathfrak{f}^1} (u^1|_{\Sigma}) = \mathcal{M}_{\mathscr{F}^2,\mathfrak{f}^2} (u^2|_{\Sigma}) \quad \text{for } u^j|_{\Sigma} \in \mathcal{S},\,j=1,2,
\end{equation*}
then for any \((\Bx,t) \in Q \), we obtain
\begin{itemize}
    \item [(i)] \( |\tilde{\mathscr{F}^1}(u^j)(\Bx,t) - \tilde{\mathscr{F}^2}(u^j)(\Bx,t)| \leq C\varepsilon^{\tau} \), where \(\tilde{\mathscr{F}^1}(u^j)(\Bx,t)\) and \(\tilde{\mathscr{F}^2}(u^j)(\Bx,t)\) are related to \(\mathscr{F}^1(u^j)(\Bx,t)\) and \(\mathscr{F}^2(u^j)(\Bx,t)\) for \(j=1,2\), respectively.
    \item [(ii)] Furthermore, if \(\mathscr{F}^1(u^1)(\Bx,t) = \mathscr{F}^2(u^2)(\Bx,t)\), then 
    \[
    |\mathfrak{f}^1(u^j)(\Bx,t) - \mathfrak{f}^2(u^j)(\Bx,t)| \leq C' \varepsilon^{\tau_1}.
    \]
\end{itemize}
where \(\tau, \tau_1 \in (0,1)\), and $C$ and $C'$ are positive constants independent of \(\varepsilon\).
\end{thm}

The precise formulations and detailed analysis of Theorem \ref{main:thm0} are developed in Section \ref{sec:result}. The logical essence of the theorem \ref{main:thm0} is that when the boundary measurements coincide for two different configurations, the corresponding realizations of the operators acting on two solutions are quantitatively close. Specifically, the difference between $\mathscr{F}^1(u^j)$ and $\mathscr{F}^2(u^j)$ is bounded by \(O(\varepsilon^{\tau})\), where $\varepsilon$ is the geometric parameter characterizing the domain's small scale and $\tau$ is a regularity-dependent exponent. Moreover, under the additional condition that the flux operators produce identical effects on their respective solutions, the corresponding source term realizations \(\mathfrak{f}^1(u^j)\) and \(\mathfrak{f}^2(u^j)\) are guaranteed to be \(O(\varepsilon^{\tau_1})\)-close, where \(\tau_1\) is an additional regularity parameter.

\begin{remark}\label{re:hidden}
For the inverse problem in~\eqref{eq:main_IP}, we analyze a specialized geometric configuration in which the domain consists of a union of product manifolds \(\mathscr{N}_\varepsilon^{\ell} = \Omega_s \times \Omega_r\). Such as in $\mathbb{R}^3$, \(\Omega_s = \Omega_\varepsilon\) is a two-dimensional manifold which has the hidden dimension, and \(\Omega_r = \gamma(I)\) is a one-dimensional manifold, which has the regular dimension. Our main result establishes the unique identifiability of physical processes confined to the hidden dimension \(\Omega_s\) from a single passive observation on the boundary \(\partial \Omega_s \times \Omega_r\), which represents the projection onto the regular dimension \(\Omega_r\). We quantify the the difference between the operators \(\mathscr{F}^1\), \(\mathscr{F}^2\) (associated with the constitutive function \(\BF\)) and \(\mathfrak{f}^1\), \(\mathfrak{f}^2\) (associated with \(f\)) for the system ~\eqref{eq:bl}. Under suitable a priori assumptions on \(\BF\) and \(f\), this bound is controlled by the geometric parameter \(\varepsilon\) of \(\Omega_s\). This demonstrates that macroscopic variations in the operators are directly influenced by the micro-geometry of the hidden dimension.
\end{remark}

\subsection{Technical developments and discussion}\label{subsec:RB-TI}

Nonlinear balance laws provide a versatile framework for modeling diverse physical systems across quantum measurements, gravitational wave detection, fluid dynamics, biological processes, and quantum mechanics, as exemplified by the representative models in Table \ref{table1}. This paper introduces a unified mathematical framework, described by \eqref{eq:bl}, that encapsulates these diverse phenomena. Our primary objective is to develop a robust methodology for identifying nonlinear operators from boundary measurements in multiscale product spaces, offering a cohesive approach to inverse problems spanning multiple physical domains.

\begin{table}[!htbp]  % 扩展浮动位置优先级，避免表格位置不合理
    \centering
    \caption{Summary of PDE Models and Symbols}
    \label{table1}
    % 列宽设计：3列总宽适配文本区，左列右对齐（标题更整齐），中列左对齐（公式更紧凑），右列左对齐（符号说明易读）
    \begin{tabular}{>{\raggedleft\arraybackslash}p{2.0cm}  % 第1列：模型名（右对齐，宽度2.0cm）
                    >{\raggedright\arraybackslash}p{5.5cm} % 第2列：方程（左对齐，宽度5.5cm，适配公式）
                    >{\raggedright\arraybackslash}p{4.5cm}}% 第3列：符号说明（左对齐，宽度4.5cm，留足文字空间）
        \toprule  % 粗顶线，比\hline更美观
        \textbf{Model} & \textbf{Equation} & \textbf{Symbol Meanings} \\
        \midrule  % 细中线，分隔标题与内容
        % 1. Navier-Stokes方程：公式紧凑排列，符号说明分行避免拥挤
        Navier-Stokes & 
        $\partial_t \Bu + (\Bu \cdot \nabla)\Bu = \mu \Delta \Bu - \nabla p + \BF$ & 
        $\Bu$: velocity, $p$: pressure, $\mu$: viscosity \\
        \addlinespace[3pt]  % 小间距，分隔行与行，避免视觉拥挤
        
        % 2. 非牛顿流体方程：修正速度向量符号（u→u），保持一致性
        Non-Newtonian & 
        $\partial_t \Bu + \nabla \cdot (\mu(|\nabla \Bu|)\nabla \Bu) = \mathbf{G}$ & 
        $\mu$: variable viscosity, $\mathbf{G}$: body force \\
        \addlinespace[3pt]
        
        % 3. FitzHugh-Nagumo方程：aligned环境左对齐，缩小公式间距，适配列宽
        FitzHugh-Nagumo & 
        $\begin{aligned}
            \partial_t u &= \Delta u + u(1-u)(u-a) - v \\
            \partial_t v &= \epsilon \Delta v + \gamma(u - bv)
        \end{aligned}$ & 
        $u,v$: potentials; $a,b,\epsilon,\gamma$: constants \\  % 用;分隔，比逗号更清晰
        \addlinespace[3pt]
        
        % 4. 趋化性方程：修正扩散系数符号（D→D），说明文字紧凑
        Chemotaxis & 
        $\partial_t c = \nabla \cdot (D\nabla c - \chi\nabla \phi) + R(c)$ & 
        $c$: density; $\phi$: attractant; $\chi$: sensitivity \\
        \addlinespace[3pt]
        
        % 5. 磁薛定谔方程：补充向量符号（A→A），保持符号规范
        Magnetic Schrödinger & 
        $(i\partial_t + \nabla \cdot \mathbf{A})^2 \psi + V\psi = 0$ & 
        $\psi$: wave function; $\mathbf{A}$: magnetic potential \\
        \addlinespace[3pt]
        
        % 6. 金兹堡-朗道方程：修正向量符号（A→A），说明文字分行更易读
        Ginzburg-Landau & 
        $\partial_t \psi = \alpha \Delta \psi + \beta |\psi|^2 \psi + \gamma \mathbf{A} \cdot \nabla \psi$ & 
        $\psi$: order parameter; $\alpha,\beta,\gamma$: constants \\
        \bottomrule  % 粗底线，收尾整洁
    \end{tabular}
\end{table}

The study of forward problems for general systems of balance laws remains in its nascent stages, with the existing literature primarily focused on specific forms-notably hyperbolic conservation laws, where the flux $\BF$ and source $f$ are restricted to particular types. Historical overviews can be found in \cite{Da2005, EL2024, Lax1960, MA1984, Serre1999, Serre2000}. Well-posedness proofs for these systems can be broadly classified into two approaches. The first approach addresses the development of singularities, such as shock waves, which emerge in finite time even from smooth, small initial data. This is exemplified in \cite{HN2003} for conservation laws and in \cite{Y2004, Z2017} for balance laws. The second approach establishes the existence of classical solutions over short time intervals, as demonstrated in \cite{Da2005} for conservation laws and \cite{MMT2021} for the Schr\"{o}dinger equation. Despite these advances, significant open challenges persist, as noted in \cite{Serre1999, Serre2000}. Historically, the theory of balance laws has relied heavily on ad hoc methods. These approaches have also been extended to coupled systems, notably in the context of time-dependent mean-field games \cite{CGM2020, CT2019, A2021}.

Due to the reasons discussed above, inverse problems for the general form of nonlinear balance laws have received considerably less attention. The primary reason for this disparity is straightforward: the mathematical tools and theoretical understanding required to tackle inverse problems robustly are predicated on a well-established theory for the forward problem, which, as discussed, is still taking shape. Most existing results are confined to simpler models and one-dimensional scalar cases, such as the BFGS-based ``discretize-then-optimize” strategy \cite{BCS2009}, Lipschitz stability analysis via Carleman estimates \cite{FT2021}, reconstruction methods using front-tracking and Riemann solvers \cite{HPR2014}, and flux identification from shock asymptotics \cite{KT2005}. Extending these methodologies to higher spatial dimensions, however, introduces substantial theoretical and computational challenges. Key difficulties include the intricate geometry of characteristic networks, the formation of non-planar shocks, complex multi-dimensional wave interactions, and the curse of dimensionality affecting both numerical simulation and parameter space exploration. Moreover, foundational one-dimensional analytical tools, such as the method of characteristics, Riemann solvers, and shock-tracking algorithms do not generalize directly to higher dimensions, further complicating the mathematical analysis.

In this paper, we investigate the inverse problem for the general form of nonlinear balance laws, as presented in \eqref{eq:bl}, incorporating source terms in multiple dimensions. As previously highlighted, this problem remains a significant challenge with limited progress in the literature. Our primary focus is the reconstruction of the flux term $\BF$ and the source term $f$ in \eqref{eq:bl} using minimal observational data, specifically a single passive boundary measurement. Existing studies, such as \cite{CM2022, BCG2009, BFM2014, CGR2006, DLZ2023, IY1998, LL2024, LLL2024}, predominantly focus on inverse problems for specific forms of the nonlinear balance laws and on recovering explicit analytical forms of $\BF$ and $f$, often relying on strong a priori assumptions, such as specific polynomial structures. However, in practical applications, the input-output behavior of these terms is typically of greater interest than their exact functional forms. This insight drives our novel approach: we reformulate the inverse problem within an operator-theoretic framework, aiming to identify the nonlinear operators $\mathscr{F}$ and $\mathfrak{f}$, which map the state function $u$ to the flux and source responses $\mathscr{F}(u)$ and $\mathfrak{f}(u)$, respectively, without requiring explicit analytical representations. This method eliminates the need for restrictive a priori assumptions on functional forms and better aligns with practical scenarios where the system's behavioral response is the primary concern.

In what follows, we restrict the domain $\Omega$ to the class of product-type domains defined in \eqref{eq:NNN}–\eqref{eq:Ne}. As noted in Remark~\ref{Re:TN}, the topological structure of such domains has recently attracted limited but growing attention. Within this geometric setting, several inverse problems have been studied \cite{MW2006, KLU2012, DLZ2021,DLU2017}. It is worth emphasizing that our geometric framework for slab-like domains is quite general and does not require the domain to be infinite. Moreover, our framework establishes explicit quantitative relationships between the geometric parameter $\varepsilon$ and reconstruction accuracy, elucidating how microscale structures govern the recoverability of nonlinear dynamics from macroscale boundary data. As discussed in Remark \ref{re:hidden}, our work draws a conceptual parallel with the AdS/CFT correspondence in string theory \cite{MJ, Witten1998, Overduin1997}, wherein bulk physics is holographically encoded in boundary data.

To establish these results, we first introduce the operators \(\mathscr{F}^j\) and \(\mathfrak{f}^j\), which correspond to the flux function \(\mathbf{F}\) and the source term \(f\), respectively, thereby reformulating the inverse problem within an operator-theoretic framework. Next, we develop complex geometrical optics (CGO) solutions for the underlying balance laws and utilize microlocal analysis to characterize the behavior of \(\mathscr{F}^j\) and \(\mathfrak{f}^j\) in product-type domains. By leveraging the intrinsic geometric properties of these solutions, we achieve an approximate unique determination of both operators from a single local passive measurement.

The remainder of this paper is structured as follows. Section \ref{sec:result} introduces the geometric framework and states the main uniqueness theorems. Section \ref{sec:proof-2D} details the two-dimensional proof, including specialized geometric optics solutions and key energy estimates. Section \ref{sec:proof-3D} extends the analysis to three-dimensional configurations. Finally, Section \ref{sec:application} demonstrates the applicability of our framework to physical systems, with a focus on reaction-diffusion-convection models.

\section{Statement of main results and preliminaries} \label{sec:result}
\subsection{Mathematical setup}
We mainly restrict our study to the nonlinear system 
\begin{equation}\label{eq:main-mu}
    \begin{cases}
    \partial_t H(u) + \nabla_{\Bx} \cdot \BF(\Bx,t,u) = f(\Bx,t,u,\nabla u) + \mu \Delta u &\  \text{in}\ \ Q,\\
    u(\Bx,t) = \boldsymbol{\psi}(\Bx,t) &\  \text{on}\ \ \widetilde\Gamma,
\end{cases}
\end{equation}
where $\mu \in \mathbb{R}_+$ is a positive constant, typically representing the viscosity coefficient in fluid mechanics. The terms \(H\), \(\mathbf{F}\), and \(f\), as defined in \eqref{eq:bl}, satisfy essential a priori assumptions critical for tackling the challenges of limited passive boundary measurements, as elaborated in Section \ref{sub:math}.

\begin{definition}[Admissible class $\mathcal{A}$]\label{Def:admc}
We say a triplet $(H, \BF, f) \in \mathcal{A}$ if they satisfy the following properties:
\begin{itemize}
\item [$(i)$] The function $H: \mathbb{C} \to \mathbb{C}$ is H\"older continuous with exponent $\alpha_1$, where $0 < \alpha_1 < 1$, i.e., $H \in C^{0,\alpha_1}(\mathbb{C})$.
\item  [$(ii)$] The function $f:Q \times \mathbb{C} \times \mathbb{C} \to \mathbb{C}$ is H\"older continuous with exponent $\alpha_2$, where $0 < \alpha_2 < 1$, i.e., $f \in C^{0,\alpha_2}(Q \times \mathbb{C} \times \mathbb{C})$. Its norm is defined as:
\begin{align}
  \qquad \|f\|_{C^{0,\alpha_2}} = \max & \left\{ \sup\|f(\cdot,t_0,z_0,p_0)\|_{C^{0,\alpha_2}(\Omega)},\right.  \sup\|f(\Bx_0,\cdot,z_0,p_0)\|_{C^{0,\alpha_2/2}(\mathbb{C})}, \notag\\
    &\left. \hspace{0.2cm} \sup\|f(\Bx_0,t_0,\cdot,p_0)\|_{C^{0,\alpha_2}(\mathbb{C})},  \sup\|f(\Bx_0,t_0,z_0,\cdot)\|_{C^{0,\alpha_2}(\mathbb{C})} \right\},\notag
\end{align}
 where \(``sup"\) denotes the supremum taken over all possible values of the corresponding triplet across their entire respective domains. For a fixed $(t_0, z_0, p_0)$, the spatial H\"{o}lder norm is defined as
\[\qquad
   \|f(\cdot,t_0,z_0,p_0)\|_{C^{0,\alpha_2}(\Omega)} := \|f\|_{L^\infty(\Omega)} + \sup_{\Bx\neq\mathbf{y}} \frac{|f(\Bx,t_0,z_0,p_0) - f(\mathbf{y},t_0,z_0,p_0)|}{|\Bx-\mathbf{y}|^{\alpha_2}}
   \]
  and the other norms are defined analogously for their respective domains.
   \item [$(iii)$] Each component $f_k$ of the vector-valued function $\BF = (f_k)_{k=1}^n: Q \times \mathbb{C} \to \mathbb{C}$ is H\"older continuous with exponent $\alpha_3$, where $0 < \alpha_3 < 1$, i.e., $f_k \in C^{0,\alpha_3}(Q \times \mathbb{C})$ for $k = 1, \dots, n$.
   \item [$(iv)$] For the triplet $(H, \BF, f)$ satisfying conditions (i)-(iii), there exists a solution to the boundary value problem \eqref{eq:main-mu} for some boundary data $\psi \in H^{1/2}(\widetilde{ \Gamma})$.
\end{itemize}
\end{definition}

\begin{remark}\label{rem:n1}

    As detailed in Subsection~\ref{subsec:RB-TI}, establishing the existence of solutions to the forward problem for system \eqref{eq:main-mu} is highly challenging. However, local existence on short-time intervals is a well-established result; see \cite{Da2005, MA1984, MMT2021, Z2017} and references therein. Since the inverse problem considered in this paper relies on boundary observations from an arbitrary subinterval $[T_1, T_2] \subset [0, T]$, we may, without loss of generality, restrict our analysis to such a short-time interval. Consequently, the existence requirement in condition (iv) of Definition~\ref{Def:admc} is satisfied under the stated regularity assumptions.

% for suitable \(\boldsymbol{\psi}(\mathbf{x}, t) \in H^{1/2}(\widetilde{\Gamma})\), we assume the existence of solutions \(u \) to \eqref{eq:main-mu}. Crucially, we do not presume the uniqueness of these solutions, thereby allowing for the possibility of multiple solutions.

%     \textcolor{red}{
% The existence of solutions to the forward problem \eqref{eq:main-mu} is a non-trivial but well-established result for nonlinear balance laws, particularly in the context of short-time existence. The well-posedness theory for such systems, specifically, the existence of classical or sufficiently regular solutions on a possibly short time interval, has been studied, see \cite{Da2005,MA1984, MMT2021,Z2017} and the references therein. Since our inverse problem \eqref{eq:main_IP} only requires observations on an arbitrary sub-interval $[T_1, T_2] \subset [0, T]$, we can always restrict our analysis to such a short-time existence regime without loss of generality. Therefore, the existence requirement in condition (iv) of Definition \ref{Def:admc} is indeed satisfiable under the stated regularity assumptions.}
\end{remark}

In the mathematical framework outlined above, the functions $(H, \BF, f)$ are defined through the admissible class $\mathcal{A}$ in Definition \ref{Def:admc}. Building upon this functional foundation, we now introduce the corresponding operator structures. As discussed in Section \ref{sub:math}, we establish the operators $\mathscr{F} = (\mathscr{F}_k)_{k=1}^n$ and $\mathfrak{f}$ associated with the flux $\BF$ and source term $f$, respectively. Within this correspondence, the functions $\BF$ and $f$ can be viewed as realizations of the operators $\mathscr{F}$ and $\mathfrak{f}$. Specifically, according to the operator definition in equation \eqref{eq:oper}, the components of the vector-valued function $\BF$ constitute realizations of the operator $\mathscr{F} = (\mathscr{F}_k)_{k=1}^{n}$, while the scalar nonlinearity $f$ represents a realization of the operator $\mathfrak{f}$. Within this operator framework, we can establish the approximate unique identifiability of the inverse problem \eqref{eq:main_IP} within an a priori admissible class of these functions using a single passive measurement on $\Sigma$.

Our analysis focuses on recovering the realizations of the operators $\mathscr{F}$ and $\mathfrak{f}$ from a single passive boundary observation. This result shows that we can obtain information in small-scale space $\Omega_{\varepsilon}$ by observing information in regular-scale space $\Omega_r$. The following theorem provides a precise formulation of Theorem \ref{main:thm0}.

\begin{thm}\label{main:thm}
%Let \(\Omega\) be a Lipschitz domain with product-type ends \(\mathscr{N}_\varepsilon\), as defined in  \eqref{eq:NNN}. Let \(\mathcal{M}_{\mathscr{F}^j, \mathfrak{f}^j}\) denote the passive boundary measurement on \( \widetilde{\Gamma}\) associated with the triplet \((H, \BF^j, f^j)\in \mathcal{A}\)  and the solution \(u^j\), as specified in~\eqref{eq:main-mu}, and \(\mathscr{F}^j\) and \(\mathfrak{f}^j\) are operators corresponding to \(\BF^j\) and \(f^j\), denoted in \eqref{eq:oper}. 

Let \(\Omega\) be a product-type domain as in \eqref{eq:NNN}, and let \(\mathcal{M}_{\mathscr{F}^j, \mathfrak{f}^j},\,(j=1,2)\) be the passive boundary measurement on \(\widetilde{\Gamma}\) corresponding to the triplet \((H, \BF^j, f^j)\in \mathcal{A}\) and its solution \(u^j\) in \eqref{eq:main-mu}, with \(\mathscr{F}^j\) and \(\mathfrak{f}^j\) being the operators associated to \(\BF^j\) and \(f^j\) as defined in \eqref{eq:oper}. Assume there exist \(\varepsilon\)-independent positive constants \(C_1\), \(C_1^{\prime}\), \(C_2\), \(C_3\), and \(C_4\), and H\"older exponents \(\alpha_1, \alpha_2, \alpha_3, \alpha_4 \in (0,1)\)
 satisfying the following conditions:
\begin{itemize}
% \item [$(i)$]The nonlinear state transfer function $H$ satisfies 
% \begin{equation}\label{Cond:Hb}
%     \|H\|_{C^{1,\alpha_1}(\mathbb{C})}\leq C_{1}\quad \mbox{for}\,\,\alpha_1\in (0,1);
% \end{equation}
\item[$(i)$] The source term $f^j$ satisfies
\begin{equation}\label{Cond:fg0}
\|f^j\|_{C^{0,\alpha_2}(\overline{Q} \times \mathbb{C}\times  \mathbb{C}) }\leq C_{2};
\end{equation}
\item [$(ii)$] The nonlinear terms $f_k^j$ satisfy for $k=1, \ldots, n$:
 \begin{equation}\label{Cond:FG}
\|f_k^j\|_{C^{0,\alpha_3}(\overline{Q} \times \mathbb{C})}\leq C_{3};
\end{equation}
\item [$(iii)$] The solution $u^j$ satisfies
 \begin{equation}\label{Cond:uv}
\begin{split}
\qquad u^j\in H^{1}_{\mathrm{loc}}(Q)\cap C^{1,\alpha_4}(\overline{Q}), \,\, \mbox{with} \,\,
 \|u^j\|_{C^{1,\alpha_4}}\leq C_{4}.
\end{split}
 \end{equation}
\end{itemize}
 If for any subset $\Sigma\subset \widetilde{\Gamma}$ and sufficiently small $\varepsilon$, one has
\begin{equation*}
    \mathcal{M}_{\mathscr{F}^1,\mathfrak{f}^1} (u^1|_{\Sigma}) = \mathcal{M}_{\mathscr{F}^2,\mathfrak{f}^2} (u^2|_{\Sigma}),
\end{equation*}
which yields the flux balance condition
\begin{equation}\label{eq:fb}
   \int_{\Sigma} \nu \cdot (\BF^1 - \BF^2) \, \mathrm{d}\sigma = \int_{\Sigma} \mu (h_{\BF^1} -h_{\BF^2}) \, \mathrm{d}\sigma,
\end{equation}
where $\nu$ is the unit outward normal and  $h_{\BF^j} \in L^2(Q) \cap C^{0,\alpha_4}(\overline{Q})$, then we have
\begin{itemize}
\item[$(a)$] If $\mathscr{F}^1(u^1)(\Bx,t)-\mathscr{F}^2(u^2)(\Bx,t) \neq \mathbf{0}$ in $ Q$, suppose the curve $\gamma$ defined in Definition \ref{def:ND} satisfies the following condition: for some $l >0$, there exist parameters $a_1, a_2 \in I$ such that as long as $|a_1 - a_2| \leq \varepsilon^l$, then there must exist $b_1,b_2 \in (a_1,a_2)$ such that $\gamma'(b_1)$ is parallel to $\gamma'(b_2)$. Furthermore, assume the nonlinear state transfer function $H$ satisfies
\begin{equation}\label{Cond:H}    
\|H\|_{C^{0,\alpha_1}(\mathbb{C})}\leq C_{1}.
\end{equation}
 Then, for any $(\Bx,t)\in Q$, there exists a positive constant $\varepsilon_0$ such that when $\varepsilon < \varepsilon_0$, the following estimate holds
    \begin{equation}\label{Re:FG}
    \left|\tilde{\mathscr{F}^1}(u^j)(\Bx,t)-\tilde{\mathscr{F}^2}(u^j)(\Bx,t)\right|\leq C(\varepsilon_{0},\mu,\|h\|_{L^{\infty}},C_{1},C_{2},C_{3},C_{4})\varepsilon^{\tau},\,\,\, j=1,2,
     \end{equation}
where $h:=h_{\BF^1} -h_{\BF^2}$, $C(\varepsilon_{0},\mu,\|h\|_{L^{\infty}},C_{1},C_{2},C_{3},C_{4})$ is  a generic positive constant depending only on the indicated parameters, and $\tau$ is given as follows 
\begin{equation}\label{eq:ta}
\tau = \begin{cases}\alpha_2 \alpha_3l & \text { if } \,\, 0<l\leq \frac{1}{1+\alpha_2 \alpha_3}, \\ \frac{1}{2}\left[1-l\left(1-\alpha_2 \alpha_3\right)\right] & \text { if }\, \,\frac{1}{1+\alpha_2 \alpha_3}<l<1, \\ \frac{1}{2}\left(1-l+\alpha_2 \alpha_3 \right) & \text { if }\, \, 1 \leq l<1+\alpha_2 \alpha_3.
\end{cases}
\end{equation} 
\item[$(b)$] If $\mathscr{F}^1(u^1)(\Bx,t)-\mathscr{F}^2(u^2)(\Bx,t) = \mathbf{0}$ and $h_{\BF^1} -h_{\BF^2}=0$ in $Q$, assume that for any point $\Bx\in\Omega$, there exists at least one point $\Bx_0\in \hat{\Gamma}$ such that the vector $\Bx-\Bx_0$ is parallel to $\nu_{\Bx_{0}}$, where $\nu_{\Bx_{0}}$ is the unit outward normal at $\Bx_0$.  Additionally, the nonlinear state transfer function $H \in C^{1,\alpha_1}$ satisfies
\begin{equation}\label{Cond:Hb}
    \|H\|_{C^{1,\alpha_1}(\mathbb{C})}\leq C_{1}^{\prime}.
\end{equation} 
Then, for any $(\Bx,t)\in Q$, the following estimate holds
    \begin{equation}\label{Re:fg}
    \hspace{1.5cm}    |\mathfrak{f}^1(u^j)(\Bx,t) -\mathfrak{f}^2(u^j)(\Bx,t)| \leq C'(\varepsilon_{0},\mu,C_{1}^{\prime},C_{2},C_{4})\varepsilon^{\tau_1},\quad j=1,2,
    \end{equation}
 where $C'(\varepsilon_{0},\mu,C_{1},C_{2},C_{3},C'_{4})$ is a generic positive constant depending only on the indicated parameters and  $\tau_1 = \frac{\alpha_1 \alpha_3^2}{1+\alpha_1 \alpha_3}$.

\end{itemize}

\end{thm}

\begin{remark}
   The studied system, described in Subsection \ref{sub:math}, models the evolution of a scalar state $u$ within $\Omega$ generated by a physical process, with measurements taken on the boundary $\Sigma$. The regularity imposed by assumptions \eqref{Cond:fg0}--\eqref{Cond:uv} and \eqref{Cond:H} (or \eqref{Cond:Hb}) is physically justifiable, as it requires only pointwise bounds—a condition met by many physical processes. This is consistent with the framework of local-time existence outlined in Remark~\ref{rem:n1}.

It is important to highlight that the result of the main theorem is non-trivial, even in the case of a linear source term $f^j$, namely, $f^j(\mathbf{x}, t, u^j, \nabla u^j) = q^j(\mathbf{x}, t) u^j$ with $q^j \in C^{0,\alpha_2}(\overline{Q})$. The same boundary observation data $(u^1,\partial_{\nu}u^1+h_{\BF^1})\big|_\Sigma=(u^2,\partial_{\nu}u^2+h_{\BF^2})\big|_\Sigma$ cannot guarantee that the second- or higher-order derivatives of $u^1$ and $u^2$ are pointwise close internally in the product-type domain.  

% Even when boundary measurements coincide ($u^1=u^2$, $\partial_{\nu}u^1+h_{\BF^1}=\partial_{\nu}u^2+h_{\BF^2}$) on product-type domains, these conditions do not ensure that the second- or higher-order derivatives of $u^1$ and $u^2$ are pointwise close internally.

% Importantly, the conclusion of the theorem is non-trivial, even when boundary measurements coincide ($u^1=u^2$, $\partial_{\nu}u^1+h_{\BF^1}=\partial_{\nu}u^2+h_{\BF^2}$) on product-type domains, these conditions do not ensure that second- or higher-order derivatives of $u^1$ and $u^2$ are pointwise close internally. 

%  \textcolor{red}{This can be elucidated by examining the passive boundary observation. The condition $\mathcal{M}_{\mathscr{F}^1,\mathfrak{f}^1}(u^1|_{\Sigma}) = \mathcal{M}_{\mathscr{F}^2,\mathfrak{f}^2}(u^2|_{\Sigma})$ implies that the observed boundary data, such as the Dirichlet data $u^1=u^2$ and the Neumann data $\partial_\nu u^1 + h_{\BF^1} = \partial_\nu u^2 + h_{\BF^2}$, coincide. Here, $h = h_{\BF^1} - h_{\BF^2} \in C^{0,\alpha_4}(\overline{Q})$ represents the inherent uncertainty or difference in the passive measurement setup. } \textcolor{blue}{Crucially, even for a linear source term $f_0^j(\mathbf{x}, t, u, \nabla u) = q^j(\mathbf{x}, t) u$ with $q^j \in C^{0,\alpha_2}(\overline{Q})$, the $\mathcal{O}(1)$ difference $h$ in the Neumann condition can induce $\mathcal{O}(1)$ discrepancies in the second-order derivatives (e.g., $\Delta u^1 - \Delta u^2$) within the domain, due to the inherent instability of differentiating interior fields with respect to boundary fluxes. }
\end{remark}

\begin{remark}\label{rmk:tilde-fg-forms}
The realizations $\tilde{\mathscr{F}^1}(u^j)$ and $\tilde{\mathscr{F}^2}(u^j)$ in Theorem \ref{main:thm}$(a)$ are linear combinations of the components of $\mathscr{F}^1(u^j)$ and $\mathscr{F}^2(u^j)$, respectively, $j=1,2$. Their specific forms depend on the dimension and geometric configuration of  \(\Omega\), as follows:
\begin{itemize}
\item [(i)]In $\mathbb{R}^2$, the realization take the rotational form:
\begin{equation*} %\label{eq:2fg}
\tilde{\mathscr{F}}^j(u) :={\mathscr{F}}^{j\prime}_2(u),
\end{equation*}
where $\mathscr{F}^{j\prime}_2(u)$ is the second component of $\mathscr{F}^{j\prime}(u)$, obtained through coordinate transformation of $\mathscr{F}^j(u)$. The specific form of $\mathscr{F}^{\prime}(u)$  is given by
\begin{align}
 \qquad   \mathscr{F}^{j\prime}(u)(\Bx,t):= \BF^{j\prime}(\Bx,t,u^j) &= (f_1^{j\prime}, f_2^{j\prime})^{\top} = (f^j_1\cos\theta + f^j_2\sin\theta, -f^j_1\sin\theta + f^j_2\cos\theta)^{\top}.\label{eq:2FG}
\end{align}
\item [(ii)] In $\mathbb{R}^3$, the expressions vary based on the geometric configuration 
\begin{equation*}%\label{eq:3fg}
\tilde{\mathscr{F}}^j(u) :={\mathscr{F}}^{j\prime}_k(u),
\end{equation*}
where $k =1, 3$ for nozzle domain $\mathscr{N}_{\varepsilon}^{N}$, and $k = 3$ for slab domain $\mathscr{N}_{\varepsilon}^{S}$. The transformed components $\mathscr{F}^{j\prime}(u):=\BF^{j\prime}(\Bx,t,u)=(f^{j\prime}_1,f^{j\prime}_2,f^{j\prime}_3)^\top$ are given by
\begin{align}\label{eq:3FG}
{\mathscr{F}}^{j\prime}_1(u)=f_1^{j\prime} &= f^j_1 \cos \beta \cos \theta + f^j_2 \sin \beta - f^j_3 \cos \beta \sin \theta, \notag\\
{\mathscr{F}}^{j\prime}_2(u)=f_2^{j\prime} &= -f^j_1 \sin \beta \cos \theta + f^j_2 \cos \beta + f^j_3 \sin \beta \sin \theta, \notag\\
{\mathscr{F}}^{j\prime}_3(u)=f_3^{j\prime} &= f^j_1 \sin \theta + f^j_3 \cos \theta.
\end{align}
\end{itemize}
\end{remark}
It is important to note that the rotation angles \(\theta\) and \(\beta\) in Remark \ref{rmk:tilde-fg-forms} are determined by the curve \(\gamma\) as specified in Theorem \ref{main:thm}$(a)$. The corresponding coordinate transformations are defined as follows:
\begin{itemize}
    \item[(i)] In \(\mathbb{R}^2\), the coordinate system is translated to a point on \(\gamma\) (corresponding to parameter \(b_1\)) and rotated clockwise about the origin by an angle \(\theta\), aligning the tangent vector \(\gamma'(b_1)\) with the \(x_1\)-axis of the new coordinate system. See Figure \ref{fg:01} for an illustration.
    \item[(ii)] In \(\mathbb{R}^3\), the coordinate system is translated to a point on \(\gamma\) (corresponding to \(b_1\)) and rotated as follows: first, a clockwise rotation by an angle \(\theta\) about the \(x_2\)-axis, followed by a clockwise rotation by an angle \(\beta\) about the \(x_3\)-axis. This ensures that \(\gamma\) lies in the \(x_2 x_3\)-plane and the tangent vector \(\gamma'(b_1)\) aligns with the \(x_2\)-axis of the new coordinate system. See Figures \ref{fg:02} and \ref{fg:03} for illustrations.
\end{itemize}

The proofs of Theorem \ref{main:thm} are presented in Sections \ref{sec:proof-2D} and \ref{sec:proof-3D}. To this end, we first introduce the auxiliary tools employed in the proof, focusing on complex geometrical optics solutions and key analytical constructions that play a fundamental role in our approach.

\begin{Lemma}[CGO Solutions]\label{lem:cgo}
The exponential function
\begin{equation}\label{u0}
    u_0(\Bx,t) = e^{\frac{\rho}{\sqrt{\mu}} \cdot \Bx + \lambda t}, \quad \rho = s\mathbf{d} + \Bi\sqrt{s^2 + \lambda}\mathbf{d}^\perp,
\end{equation}
where $\mathbf{d},\mathbf{d}^\perp \in \mathbb{S}^{n-1}$ satisfy $\mathbf{d} \cdot \mathbf{d}^\perp = 0$ and $\mu,s,\lambda \in \mathbb{R}_+$, has the properties:
\begin{equation}\label{eq:cgo}
    (\partial_t + \mu \Delta)u_0 = 0 \quad \text{and} \quad \partial_t u_0 = \lambda u_0.
\end{equation}
\end{Lemma}

Another essential tool is the following Green's formula, which will be used extensively in our energy estimates:

\begin{Lemma}[Space-Time Green's Formula]\label{lem:green}
Let $\Omega \subset \mathbb{R}^n$ be a bounded Lipschitz domain and $[T_1,T_2] \subset [0,T]$ a time interval. For any functions $f,g$ in the space $H^1(\Omega \times [T_1,T_2])$ with $\Delta f,\,\Delta g \in L^2(\Omega \times [T_1,T_2])$ and $\partial_t f,\,\partial_t g \in L^2(\Omega \times [T_1,T_2])$, we have:
\begin{equation*}
    \int_{T_1}^{T_2} \int_\Omega (g\Delta f - f\Delta g) d\Bx dt = \int_{T_1}^{T_2} \int_{\partial\Omega} (g\partial_\nu f - f\partial_\nu g) d\sigma dt
\end{equation*}
with $\nu$ denoting the outward unit normal to $\partial\Omega$.
\end{Lemma}

\section{Proof of Theorem \ref{main:thm} in two dimensions}\label{sec:proof-2D}

\subsection{Geometric Preparation}
To establish the foundation for proving Theorem \ref{main:thm} in the two-dimensional case, we first introduce the key geometric configuration. By the assumption of $\gamma$ in Theorem \ref{main:thm}$(a)$, there are many points on the curve with the same tangent direction, and the radial distance between adjacent points does not exceed $\varepsilon^l$. Choose any point, move the original coordinate system to this point and use this point as the new origin, then rotate the coordinate system as described in Remark \ref{rmk:tilde-fg-forms}. In the new coordinate system, a region of horizontal length $\varepsilon^l$ is intercepted in $\mathscr{N}_{\varepsilon}^{\ell}\Subset \Omega$ starting from the origin and defined as $D_{\varepsilon}$, see Figure \ref{fg:01} for illustration. 

For the remaining part $\mathscr{N}_{\varepsilon}^{\ell} \backslash D_{\varepsilon}$, we translate the origin to the next point without rotating the coordinate system and select a new subregion. Since coordinate translation does not alter the form of $\nabla_{\Bx} \cdot \BF^{\prime}$, we obtain the same result in the new subregion. Repeating this process a finite number of times covers the entire $\mathscr{N}_{\varepsilon}^{\ell}$. As for the remaining part $\Omega \backslash \overline{\mathscr{N}_{\varepsilon}^{\ell}}$, a similar method can be applied.

%For the remaining part $\mathscr{N}_{\varepsilon}^{\ell}\backslash D_{\varepsilon}$, we can translate the origin to the next point without rotating the coordinate system, and then select a new subregion. Since the coordinate translation does not change the form of $\nabla \cdot \BF^{\prime}$, we can obtain the same result in the new subregion. Repeating this step a finite number of times will cover the entire $\mathscr{N}_{\varepsilon}^{\ell}$, which means that Theorem \ref{main:thm} has been proved on region $\mathscr{N}_{\varepsilon}$. 

In the transformed coordinates, the flux terms become $\BF'$ and $\mathbf{G}'$, which are defined in \eqref{eq:2FG}. The boundary $\partial D_{\varepsilon}$ decomposes into four distinct components:
\begin{align}
    \Gamma_1 &= \{(x_1,x_2) \mid 0 < x_1 < \varepsilon^l,\ x_2 = \gamma(x_1)\}, \notag\\
    \Gamma_2 &= \{(x_1,x_2) \mid x_1 = 0,\ 0 < x_2 < \varepsilon \},\notag \\
    \Gamma_3 &= \{(x_1,x_2) \mid 0 < x_1 < \varepsilon^l,\ x_2 = \gamma(x_1)+\varepsilon\},\notag\\
    \Gamma_4 &= \{(x_1,x_2) \mid x_1 = \varepsilon^l,\ \gamma(\varepsilon^l) < x_2 < \gamma(\varepsilon^l) +\varepsilon \},\label{eq:bd2}
\end{align}
where $\gamma$ satisfy that
\begin{equation}\label{eq:gamm2}
    |\gamma|= \Oh(\varepsilon),
\end{equation}
and for simplicity, we denote
\begin{equation}\label{eq:l0}
    l_0 := \min\{l, 1\},
\end{equation} 
where $l$ is defined in Theorem \ref{main:thm}, part $(a)$.
\begin{figure}[htbp]
    \centering
    \begin{subfigure}[b]{0.45\textwidth}
        \centering
        \includegraphics[width=\linewidth]{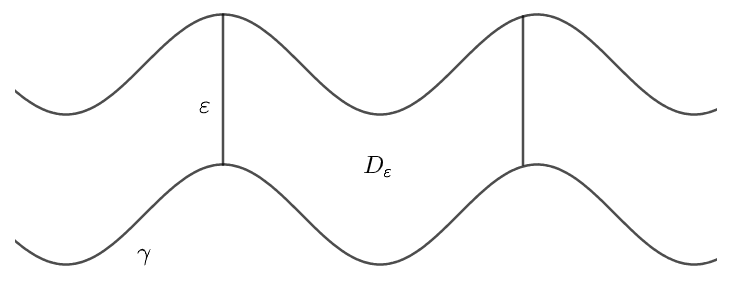}
        \caption{General nozzle configuration in 2D}
    \end{subfigure}
    \hfill
    \begin{subfigure}[b]{0.45\textwidth}
        \centering
        \includegraphics[width=\linewidth]{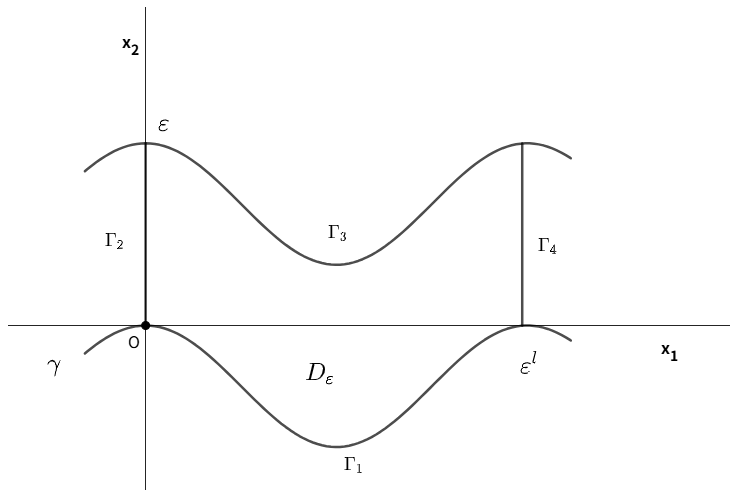}
        \caption{Transformed subdomain $D_\varepsilon$}
    \end{subfigure}
    \caption{ configuration for the 2D analysis showing (a) the original nozzle domain and (b) the transformed coordinate system with characteristic subdomain.}
    \label{fg:01}
\end{figure}

%This geometric setup provides the framework for our subsequent analysis, where the coordinate transformation simplifies calculations while preserving the essential scaling behavior as $\varepsilon \to 0$. The uniform thickness assumption allows us to focus on the fundamental analytical aspects, with the understanding that the results extend to more general geometries through the conditions previously established.

%We begin by establishing fundamental results that will underpin our main analysis. First, considering the difference $w = u - v$ between solutions of the coupled system \eqref{main:eq1}, we obtain:

To prove the uniqueness of our main results, we derive key estimates for the solution pair of the general coupled PDE system \eqref{eq:uv}.

\begin{Lemma}\label{lem:w}
Let $\mathscr{N}_{\varepsilon}^{\ell}$ be a product-type domain, as specified in \eqref{eq:Ne}. Suppose that \( u, v \in H^{1}_{\mathrm{loc}}(\mathscr{N}_\varepsilon^{\ell}\times[0,T]) \cap C^{1,\alpha_4}(\overline{\mathscr{N}_\varepsilon^{\ell}}\times[0,T])\) with $\alpha_4 \in (0,1)$ satisfy \eqref{Cond:uv} and the coupled system
\begin{equation}\label{eq:uv}   
\begin{cases}        
\partial_t H(u) - \mu \Delta u = f - \nabla_{\Bx} \cdot \BF^{\prime}  & \text{in } \mathscr{N}_\varepsilon^{\ell} \times [0,T], \\   
\partial_t H(v) - \mu \Delta v = g - \nabla_{\Bx} \cdot \mathbf{G}^{\prime}  & \text{in } \mathscr{N}_\varepsilon^{\ell} \times [0,T], \\               
u = v, \quad \partial_{\nu}u = \partial_{\nu}v + h(\Bx,t) & \text{on } \hat{\Gamma}\times [0,T],
\end{cases}
\end{equation}
where $\hat{\Gamma}$ is defined in \eqref{def:Gamma}. The functions \( f \) and \( g \) satisfy  \eqref{Cond:fg0}, while the components of functions \( \mathbf{F}' \) and \( \mathbf{G}' \) satisfy  \eqref{Cond:FG}. \( H \) satisfies \eqref{Cond:H} or \eqref{Cond:Hb}, and $h \in L^2(\mathscr{N}_\varepsilon^{\ell}\times[0,T]) \cap C^{0,\alpha_4}(\overline{\mathscr{N}_\varepsilon^{\ell}}\times[0,T])$. Let $w = u - v$. Then $w$ satisfies
\begin{equation}\label{main:eqw}    
\begin{cases}        
\partial_t (H(u) - H(v)) - \mu \Delta w = (f - g) - \nabla_{\Bx} \cdot (\BF^{\prime} - \mathbf{G}^{\prime}) & \text{in } \mathscr{N}_\varepsilon^{\ell} \times [0,T], \\               
w = 0, \quad \partial_{\nu}w = h(\Bx,t) & \text{on } \hat{\Gamma} \times [0,T].
\end{cases}
\end{equation}
\end{Lemma}

\begin{Lemma}\label{lemma:exp}
Under the same assumptions on $f, g$, $\BF^{\prime}$, $\mathbf{G}^{\prime}$ as in Lemma \ref{lem:w}, let $\BF^{\prime} = (f_1^{\prime}, f_2^{\prime})^{\top}$ and $\mathbf{G}^{\prime} = (g_1^{\prime}, g_2^{\prime})^{\top}$ be defined by \eqref{eq:2FG}. For any point $(\Bx,t) \in \overline{D}_{\varepsilon} \times [T_1, T_2] \subset \mathscr{N}_{\varepsilon}^{\ell} \times [0,T]$, with $T_2 - T_1 = \varepsilon^2$ and $\varepsilon$ being the geometric parameter, the following expansions hold
\begin{align*}
f(\Bx, t, u, \nabla u) - g(\Bx, t, v, \nabla v)
&= (f - g)_{(\Bx,t)=(\Bx_0,t_0)} + \delta_{f}f - \delta_{g}g \\
&\coloneqq (f - g)_{(\Bx,t)=(\Bx_0,t_0)} + \delta_0(f - g), \\
f_k^{\prime}(\Bx, t, u) - g_k^{\prime}(\Bx, t, v)
&= (f_k^{\prime} - g_k^{\prime})_{(\Bx,t)=(\Bx_0,t_0)} + \delta_{f_k^{\prime}} f_k^{\prime} - \delta_{g_k^{\prime}} g_k^{\prime} \\
&\coloneqq (f_k^{\prime} - g_k^{\prime})_{(\Bx,t)=(\Bx_0,t_0)} + \delta_{k}(f_k^{\prime} - g_k^{\prime}).
\end{align*}
Furthermore, the remainder terms satisfy the estimates:
\begin{equation*}
\begin{aligned}
|\delta_0(f - g)| &\leq C(C_2, C_4) \big(|\Bx - \Bx_0| + |t - t_0|^{1/2}\big)^{\alpha_2 \alpha_4}, \\
|\delta_k(f_k^{\prime} - g_k^{\prime})| &\leq C(C_3, C_4) \big(|\Bx - \Bx_0| + |t - t_0|^{1/2}\big)^{\alpha_3 \alpha_4},
\end{aligned}
\end{equation*}
where $(\Bx_0, t_0) \in \overline{D}_{\varepsilon} \times [T_1, T_2]$, $|\cdot|$ denotes the Euclidean distance, $\alpha_i \in (0,1)$ for $i = 2,3,4$, and the constants $C_i$ ($i = 2,3,4$) are as defined in \eqref{Cond:fg0}, \eqref{Cond:FG}, and \eqref{Cond:uv}, respectively.
\end{Lemma}

These expansions will enable us to carefully track the dependence on geometric parameter $\varepsilon$ in our estimates. With these preparatory results established, we now derive a key integral identity that connects boundary behavior to interior quantities. This identity, which follows from careful application of Green's theorem, will be central to our analysis.
\begin{Lemma}\label{lemma:I1-6}
Consider the nonlinear system \eqref{main:eqw} and under the same setup as in Lemma~\ref{lem:w}, using the CGO solution $u_0$ from \eqref{u0} with direction $\mathbf{d} = (d_1,d_2)$ where $d_1,d_2<0$, and for any time interval $[T_1,T_2] \subset [0,T]$ with $T_2-T_1=\varepsilon^2$, we have the identity:
\begin{equation}\label{I1-6}
    I_1 + I_2 = I_3 +I_4 -I_5 -I_6,
\end{equation}
where
\begin{align*}
        I_1 =& \mu \int_{T_1}^{T_2} \int_{\Gamma_2 \cup \Gamma_4} (w \partial_{\nu} u_0 - u_0 \partial_{\nu} w) \mathrm{d} \sigma \mathrm{d} t, 
        &&I_2 = \int_{T_1}^{T_2} \int_{\Gamma_2 \cup \Gamma_4}\nu \cdot (\BF^{\prime} -\mathbf{G}^{\prime}) u_0\mathrm{d} \sigma \mathrm{d} t,\nonumber\\
        I_3 =&  \int_{T_1}^{T_2} \int _{D_{\varepsilon}} (f - g) u_0 \mathrm{d} \Bx \mathrm{d} t,    &&I_4 = \int_{T_1}^{T_2} \int _{D_{\varepsilon}} (\BF^{\prime} -\mathbf{G}^{\prime}) \cdot \nabla u_0 \mathrm{d} \Bx \mathrm{d} t,\\
        I_5 =&  \int_{D_{\varepsilon}}(H(u)-H(v))u_0|_{T_1}^{T_2} \mathrm{d} \Bx,
       && I_6 = \int_{T_1}^{T_2} \int _{D_{\varepsilon}}\lambda (w -(H(u)-H(v))) u_0\mathrm{d} \Bx \mathrm{d} t.\nonumber
\end{align*}
\end{Lemma}

\begin{proof}
By applying Lemma \ref{lem:green}, the CGO solution in \eqref{eq:cgo}, the geometry in \eqref{eq:bd2}, and the boundary conditions from \eqref{main:eqw}, we obtain
\begin{align}\label{GreenF}
    &\int_{T_1}^{T_2} \int _{D_{\varepsilon}} (\partial_t (H(u)-H(v)) - \mu \Delta w)u_0 + (\partial_t u_0 + \mu \Delta u_0)w  \mathrm{d} \Bx \mathrm{d} t\nonumber\\
    = & \int_{T_1}^{T_2} \int_{\partial D_{\varepsilon}}\mu (w \partial_{\nu} u_0 - u_0 \partial_{\nu} w) \mathrm{d} \sigma \mathrm{d} t + \int_{D_{\varepsilon}}(H(u)-H(v))u_0|_{T_1}^{T_2} \mathrm{d} \Bx\nonumber\\
    &+ \int_{T_1}^{T_2} \int _{D_{\varepsilon}}(w -(H(u)-H(v)))\partial_t u_0\mathrm{d} \Bx \mathrm{d} t\nonumber\\
    = & \mu \int_{T_1}^{T_2} \big[\int_{\Gamma_2 \cup \Gamma_4} (w \partial_{\nu} u_0 - u_0 \partial_{\nu} w )\mathrm{d} \sigma- \int_{\Gamma_1 \cup \Gamma_3} u_0 h(u,v) \mathrm{d} \sigma \big]\mathrm{d} t  \notag\\ &+ \int_{D_{\varepsilon}}(H(u)-H(v))u_0|_{T_1}^{T_2} \mathrm{d} \Bx
      + \int_{T_1}^{T_2} \int _{D_{\varepsilon}}\lambda (w -(H(u)-H(v))) u_0\mathrm{d} \Bx \mathrm{d} t.
\end{align}

On the other hand, using \eqref{eq:fb}, \eqref{eq:cgo}, and \eqref{main:eqw}, we can drive that 
\begin{align}\label{DivF}
        &\int_{T_1}^{T_2} \int _{D_{\varepsilon}} (\partial_t (H(u)-H(v)) - \mu \Delta w)u_0 + (\partial_t u_0 + \mu \Delta u_0)w  \mathrm{d} \Bx \mathrm{d} t\nonumber\\
        = & \int_{T_1}^{T_2} \int _{D_{\varepsilon}} (\partial_t (H(u)-H(v)) - \mu \Delta w)u_0 \mathrm{d} \Bx \mathrm{d} t\nonumber\\
        =& \int_{T_1}^{T_2} \int _{D_{\varepsilon}} (f - g) u_0 \mathrm{d} \Bx \mathrm{d} t -  \int_{T_1}^{T_2} \int _{D_{\varepsilon}} \nabla_{\Bx} \cdot (\BF^{\prime} -\mathbf{G}^{\prime}) u_0 \mathrm{d} \Bx \mathrm{d} t\nonumber\\
        =& \int_{T_1}^{T_2} \int _{D_{\varepsilon}} (f - g) u_0 \mathrm{d} \Bx \mathrm{d} t - \int_{T_1}^{T_2} \int _{\partial D_{\varepsilon}} \nu \cdot (\BF^{\prime} -\mathbf{G}^{\prime}) u_0 \mathrm{d} \sigma \mathrm{d} t \nonumber\\
        &+ \int_{T_1}^{T_2} \int _{D_{\varepsilon}} (\BF^{\prime} -\mathbf{G}^{\prime}) \cdot \nabla u_0 \mathrm{d} \Bx \mathrm{d} t\nonumber\\
        =& \int_{T_1}^{T_2} \int _{D_{\varepsilon}} (f - g) u_0 \mathrm{d} \Bx \mathrm{d} t + \int_{T_1}^{T_2} \int _{D_{\varepsilon}} (\BF^{\prime} -\mathbf{G}^{\prime}) \cdot \nabla u_0 \mathrm{d} \Bx \mathrm{d} t\nonumber\\
        &- \int_{T_1}^{T_2} \big[\int_{\Gamma_2 \cup \Gamma_4}\nu \cdot (\BF^{\prime} -\mathbf{G}^{\prime}) u_0\mathrm{d} \sigma + \mu \int_{\Gamma_1 \cup \Gamma_3} u_0 h(u,v) \mathrm{d} \sigma  \big]\mathrm{d} t.
\end{align}
Equating expressions \eqref{GreenF} and \eqref{DivF} and analyzing the boundary terms yields the desired identity \eqref{I1-6}.
\end{proof}

These terms will be carefully estimated in the following process to establish our main results. For the CGO solution $u_0$ defined in \eqref{u0} in two dimensions, we take $\mathbf{d} = (d_1,d_2)$ with $d_1,d_2 <0$ and $\mathbf{d}^{\bot} = (d_2,-d_1)$, then we get 
    \begin{align}
        u_0(\Bx,t) &= e^{\lambda t} e^{\mu^{-\frac{1}{2}}(sd_1 + \Bi \sqrt{s^2 + \lambda}d_2)x_1  + \mu^{-\frac{1}{2}}(sd_2 - \Bi \sqrt{s^2 + \lambda}d_1)x_2 },\notag\\
        \nabla u_0(\Bx,t) &= \mu^{-\frac{1}{2}} (sd_1 + \Bi \sqrt{s^2 + \lambda}d_2,sd_2 - \Bi \sqrt{s^2 + \lambda}d_1)^{\top} u_0(\Bx,t).\label{eq:cgo1}
    \end{align}

In the subsequent analysis, we consider $s = \varepsilon^{-\beta}$ for some $\beta>0$ satisfying the conditions
\begin{equation}\label{eq:beta}
\beta \leq l \, \text{~and~} \, \beta <1.
\end{equation}

\begin{Lemma}\label{lemma:I1}
   The boundary integral term $I_1$ defined in \eqref{I1-6} admits the following asymptotic estimate:
    \begin{equation}\label{I1}
      % |I_1| \leq C(\|h\|_{L^\infty},\mu)s^2\varepsilon^{3+l} + C(C_4,\mu)s^2\varepsilon^{2+l+(1+\alpha_4)l_0} + C(C_4,\mu)\varepsilon^{2+\alpha_4 l_0}+ \mathcal{O}(s^3\varepsilon^{3+2l})+ \mathcal{O}(s \varepsilon^{3+\alpha_4 l_0}),
       |I_1| \leq C(\mu,\|h\|_{L^\infty},C_4)(s^2\varepsilon^{3+l} + s^2\varepsilon^{2+l+(1+\alpha_4)l_0} +\varepsilon^{2+\alpha_4 l_0})+ \mathcal{O}(s^3\varepsilon^{3+2l})+ \mathcal{O}(s \varepsilon^{3+\alpha_4 l_0}),
    \end{equation}
    where %$s\varepsilon^{l_0},s\varepsilon,\varepsilon \ll 1$, 
    $l_0$ is defined in \eqref{eq:l0}, and positive constant $ C(\mu,\|h\|_{L^\infty},C_4)$ depends only on $\mu$, $\|h\|_{L^\infty}$, and $C_4$.
\end{Lemma}
\begin{proof}
Using the boundary condition $w=0$ on $\hat{\Gamma}\times [T_1,T_2]$, it follows that the tangential derivative of \(w\) along the boundary curve \(\gamma\) and any time derivative on the boundary vanish. Namely, $\partial_t w(\Bx,t) =0$ for $\Bx\in \hat{\Gamma}$. In the transformed coordinate system, this implies $\partial_1 w(\mathbf{0},t) =0$.  Moreover, based on Lemma \ref{lem:w}, we have $w\in C^{1,\alpha_4}(\overline{D}_{\varepsilon}\times [T_1,T_2])$. From the geometric framework in \eqref{eq:bd2}, we derive the expansion for any point $(\Bx,t) \in \Gamma_2 \times [T_1,T_2]$ as $\varepsilon \ll 1$:   
    \begin{align}
        w(\Bx,t) &= w(\mathbf{0},t_0) + (0,x_2) \cdot (\partial_1 w(\mathbf{0},t_0),\partial_2 w(\mathbf{0},t_0))+ \partial_t w(\mathbf{0},t_0)(t-t_0)+ \mathcal{R}_{w|\Gamma_2}\notag \\
        &= - h(\mathbf{0},u,v)_{(\Bx,t)=(\mathbf{0},t_0)}x_2 +\mathcal{R}_{w|\Gamma_2},\label{eq:Gamma_2 w}\\
        \partial_{\nu} w(\Bx,t) &=(-1,0)\cdot (\partial_1 w(\Bx,t),\partial_2 w(\Bx,t)) = - \partial_1 w(\Bx,t) \notag \\
        &=- \partial_1 w(\mathbf{0},t_0) + \mathcal{R}_{\partial w | \Gamma_2} =\mathcal{R}_{\partial w | \Gamma_2}, \notag\\
        \partial_{\nu} u_0 &=- \mu^{-\frac{1}{2}}(sd_1 + \Bi \sqrt{s^2 + \lambda}d_2) e^{\lambda t} e^{\mu^{-\frac{1}{2}}(sd_2 - \Bi \sqrt{s^2 + \lambda}d_1)x_2 },\label{eq:Gamma_2 par u0}
    \end{align}
where
\begin{align}
|\mathcal{R}_{w|\Gamma_2}|&\leq C_4 (|\Bx| + |t-t_0|^{\frac{1}{2}}) ^{1+\alpha_4}\leq C_4 \varepsilon ^{1+\alpha_4},\notag\\
 |\mathcal{R}_{\partial w | \Gamma_2}| &\leq C_4  (|\Bx| + |t-t_0|^{\frac{1}{2}})^{\alpha_4}\leq C_4\varepsilon^{\alpha_4 },\label{eq:r2}
\end{align}
for some $t_0\in [T_1,T_2]$ and $T_2-T_1=\varepsilon^2$.

Similarly, for any point $(\Bx,t) \in \Gamma_4 \times [T_1,T_2]$ for $\varepsilon \ll 1$, we obtain:
    \begin{align}
        w(\Bx,t) &= w(\mathbf{0},t_0) + (\varepsilon^l,x_2) \cdot (\partial_1 w(\mathbf{0},t_0),\partial_2 w(\mathbf{0},t_0)) + \partial_t w(\mathbf{0},t_0)(t-t_0)+\mathcal{R}_{w|\Gamma_4}\notag\\
        &= - h(\mathbf{0},u,v)_{(\Bx,t)=(\mathbf{0},t_0)}x_2 + \mathcal{R}_{w|\Gamma_4},\label{eq:Gamma_4 w}\\
        \partial_{\nu} w(\Bx,t) &=(1,0)\cdot (\partial_1 w(\Bx,t),\partial_2 w(\Bx,t)) =  \partial_1 w(\Bx,t) \notag\\
        &= \partial_1 w(\mathbf{0},t_0) +  \mathcal{R}_{\partial w | \Gamma_4} =\mathcal{R}_{\partial w | \Gamma_4}, \notag\\
        \partial_{\nu} u_0 &= \mu^{-\frac{1}{2}}(sd_1 + \Bi \sqrt{s^2 + \lambda}d_2) e^{\lambda t} e^{\mu^{-\frac{1}{2}}(sd_2 - \Bi \sqrt{s^2 + \lambda}d_1)x_2 }  e^{\mu^{-\frac{1}{2}}(sd_1 + \Bi \sqrt{s^2 + \lambda}d_2)\varepsilon^l },\label{eq:Gamma_4 par u0}
    \end{align}

where 
\begin{align}
    |\mathcal{R}_{w|\Gamma_4}|&\leq C_4 \big(|(\varepsilon^l,x_2)| + |t-t_0|^{\frac{1}{2}}\big) ^{1+\alpha_4} \leq C_4 \varepsilon ^{(1+\alpha_4) l_0},\notag\\
    |\mathcal{R}_{\partial w | \Gamma_4}|& \leq C_4 \big(|(\varepsilon^l,x_2)| + |t-t_0|^{\frac{1}{2}}\big)^{\alpha_4} \leq C_4 \varepsilon^{\alpha_4 l_0},\label{eq:r4}
\end{align}
with $l_0$ defined in \eqref{eq:l0}.
Therefore, combining \eqref{eq:Gamma_2 w}, \eqref{eq:Gamma_2 par u0}, \eqref{eq:Gamma_4 w}, and \eqref{eq:Gamma_4 par u0}, we have the following estimate:
    \begin{align}
    & \left|\int_{T_1}^{T_2}\int_{\Gamma_2 \cup \Gamma_4} w \partial_{\nu} u_0  \mathrm{d} \sigma  \mathrm{d} t\right| \notag\\
    &\leq \bigg| \mu^{-\frac{1}{2}} (sd_1 + \Bi \sqrt{s^2 + \lambda}d_2) (e^{\mu^{-\frac{1}{2}}(sd_1 + \Bi \sqrt{s^2 + \lambda}d_2)\varepsilon^l}-1) \frac{e^{\lambda T_1}}{\lambda} (e^{\lambda \varepsilon}-1)\notag \\
    & \quad \times \int_0^{\varepsilon} \left(- h(\mathbf{0},u,v)_{(\Bx,t)=(\mathbf{0},t_0)} x_2 + C_4 (\varepsilon ^{(1+\alpha_4)l_0})\right) e^{\mu^{-\frac{1}{2}}(sd_2 - \Bi \sqrt{s^2 + \lambda}d_1)x_2} \mathrm{d} x_2 \bigg|.\label{eq:b24}
    \end{align}
By incorporating the CGO parameter $s$ from \eqref{eq:beta}, we obtain    
%Substituting the CGO parameter $s$ from \eqref{eq:beta} into the equation yields:
    \begin{align}
     &|s d_1 + i \sqrt{s^2 + \lambda} d_2| \leq  \sqrt{s^2+\lambda},\notag\\
     & e^{\mu^{-\frac{1}{2}}sd_2 \varepsilon}=1+\mu^{-\frac{1}{2}}sd_2 \varepsilon + \frac{(\mu^{-\frac{1}{2}}sd_2 \varepsilon)^2}{2} + \Oh(s^3\varepsilon^{3}),\notag\\
        &|1 - e^{ \mu^{-\frac{1}{2}}(sd_1 + \Bi \sqrt{s^2 + \lambda}d_2) \varepsilon^l}| \leq (e^{ \mu^{-\frac{1}{2}} \sqrt{1+\lambda}}-1) s  \varepsilon^l. \label{eq:esti}
    \end{align}
Next we estimate part of integral formula \eqref{eq:b24} further,
\begin{align}\label{eq:part}
    &\left|\int_0^\varepsilon x_2 e^{\mu^{-\frac{1}{2}} (sd_2 - \Bi \sqrt{s^2 + \lambda}d_1)x_2 } \mathrm{d} x_2\right|\leq \int_0^\varepsilon x_2 e^{\mu^{-\frac{1}{2}} s d_2 x_2 }\mathrm{d} x_2 \notag\\
    &\leq \bigg|\frac{e^{\mu^{-\frac{1}{2}} s d_2 \varepsilon} (\mu^{-\frac{1}{2}} s d_2 \varepsilon - 1) + 1}{(\mu^{-\frac{1}{2}} s d_2)^2}\bigg|= \frac{\varepsilon^2}{2} + \mathcal{O}(s\varepsilon^3).
\end{align}
Similarly, we have
\begin{equation}\label{eq:part1}
  \left  |\int_0^\varepsilon e^{(sd_2 - \Bi \sqrt{s^2 + \lambda}d_1)x_2 \mu^{-\frac{1}{2}}} \mathrm{d} x_2\right|\leq  \varepsilon + \mathcal{O}(s\varepsilon^2).
\end{equation}
Combining \eqref{eq:b24}, \eqref{eq:esti}, \eqref{eq:part}, and \eqref{eq:part1}, we obtain
\begin{align}\label{eq:part3}
         \left|\int_{T_1}^{T_2}\int_{\Gamma_2 \cup \Gamma_4} w \partial_{\nu} u_0  \mathrm{d} x_2  \mathrm{d} t\right|\leq & \frac{\mu^{-\frac{1}{2}} |h(\mathbf{0},u,v)_{(\Bx,t)=(\mathbf{0},t_0)}| e^{\lambda T_1}  }{2} s^2 \varepsilon^{l + 3}\notag\\
         &+C_4 \mu^{-\frac{1}{2}} e^{\lambda T_1}s^2 \varepsilon^{2+l+(1+\alpha_4)l_0} + \mathcal{O}(s^3\varepsilon^{3+2l}).
\end{align}
%Since $s \varepsilon^l \ll 1$, the higher-order term $\mathcal{O}(s^3 \varepsilon^{3+2l})$ is negligible.
Likewise, we have the following estimate:
    \begin{align}\label{eq:part4}
        \left |\int_{T_1}^{T_2}\int_{\Gamma_2 \cup \Gamma_4}  u_0 \partial_{\nu}w  \mathrm{d} x_2 \mathrm{d} t\right| 
        \leq  & C_4\varepsilon^{\alpha_4 l_0} \bigg|\int_{T_1}^{T_2}\int_0^{\varepsilon} e^{\lambda t} e^{\mu^{-\frac{1}{2}}sd_2 x_2} (1+e^{\mu^{-\frac{1}{2}}sd_1 \varepsilon^l})  \mathrm{d} x_2 \mathrm{d} t\bigg|\notag\\
        &\leq  C_4 e^{\lambda T_1} \varepsilon^{2+\alpha_4 l_0} + \mathcal{O}(s \varepsilon^{3+\alpha_4 l_0}).     
    \end{align}
Combining \eqref{eq:part3} and \eqref{eq:part4}, it yields the desired bound \eqref{I1}.
\end{proof}

\begin{Lemma}\label{lemma:I2}
   The term $I_2$, defined in \eqref{I1-6}, can be decomposed as
    \begin{equation}\label{eq:I2-2}
        I_2 = I_{21} + I_{22},
    \end{equation}
    with
    \begin{align}
        I_{21}=& \frac{(f_1^{\prime}-g_1^{\prime})_{(\Bx,t)=(\mathbf{0},t_0)}}{\lambda (sd_2 - \Bi \sqrt{s^2 + \lambda}d_1)} e^{\lambda T_1} (e^{\lambda \varepsilon} - 1) (e^{(sd_1 + \Bi \sqrt{s^2 + \lambda}d_2)\varepsilon^l -1})(e^{(sd_2 - \Bi \sqrt{s^2 + \lambda}d_1)\varepsilon} - 1),\notag\\
        |I_{22}| \leq & C(\mu,C_3,C_4)\varepsilon^{2+\alpha_3 \alpha_4 l_0} +\mathcal{O}(s\varepsilon^{3+\alpha_3 \alpha_4 l_0}),\label{eq:I_2}
    \end{align}
where $l_0$ is defined in \eqref{eq:l0}, and $C(\mu,C_3,C_4)$ is a positive constant depending only on $\mu$, $C_3$, and $C_4$.
\end{Lemma}

\begin{proof}
    Using the H\"older expansion from Lemma \ref{lemma:exp}, we have
        \begin{align}
         I_2 &=\int_{T_1}^{T_2} \int_{\Gamma_2 \cup \Gamma_4}\nu \cdot (\BF^{\prime} -\mathbf{G}^{\prime}) u_0\mathrm{d} x_2 \mathrm{d} t\notag\\
         &= \int_{T_1}^{T_2} \int_{\Gamma_2} (-1,0) \cdot (\BF^{\prime} -\mathbf{G}^{\prime}) u_0\mathrm{d} x_2 \mathrm{d} t + \int_{T_1}^{T_2} \int_{\Gamma_4} (1,0) \cdot (\BF^{\prime} -\mathbf{G}^{\prime}) u_0\mathrm{d} x_2 \mathrm{d} t\notag\\
         &= \int_{T_1}^{T_2} \int_{\Gamma_2} -(f_1^{\prime}-g_1^{\prime}) u_0\mathrm{d} x_2 \mathrm{d} t + \int_{T_1}^{T_2} \int_{\Gamma_4} (f_1^{\prime}-g_1^{\prime}) u_0\mathrm{d} x_2 \mathrm{d} t \notag \\
         %=& \int_{T_1}^{T_2} \int_{\Gamma_2} -(f_1^{\prime}-g_1^{\prime})_{(\Bx,t)=(\mathbf{0},t_0)} u_0\mathrm{d} \Bx \mathrm{d} t - \int_{T_1}^{T_2} \int_{\Gamma_2} \delta_1(f_1^{\prime}-g_1^{\prime}) u_0\mathrm{d} \Bx \mathrm{d} t\\
         %&+ \int_{T_1}^{T_2} \int_{\Gamma_4} (f_1^{\prime}-g_1^{\prime})_{(\Bx,t)=(\mathbf{0},t_0)} u_0\mathrm{d} \Bx \mathrm{d} t + \int_{T_1}^{T_2} \int_{\Gamma_2} \delta_1(f_1^{\prime}-g_1^{\prime}) u_0\mathrm{d} \Bx \mathrm{d} t\\
         &=(f_1^{\prime}-g_1^{\prime})_{(\Bx,t)=(\mathbf{0},t_0)}(\int_{T_1}^{T_2}\int_{\Gamma_4} u_0\mathrm{d} x_2 \mathrm{d} t - \int_{T_1}^{T_2}\int_{\Gamma_2} u_0\mathrm{d} x_2 \mathrm{d} t)\notag\\
         &\quad +\big( \int_{T_1}^{T_2} \int_{\Gamma_4} \delta_1(f_1^{\prime}-g_1^{\prime}) u_0\mathrm{d} x_2 \mathrm{d} t - \int_{T_1}^{T_2} \int_{\Gamma_2} \delta_1(f_1^{\prime}-g_1^{\prime}) u_0\mathrm{d} x_2 \mathrm{d} t \big) \notag\\
        & :=I_{21}+I_{22}.\label{eq:I2}
        \end{align}
Furthermore, we get
    \begin{align}
       &\int_{T_1}^{T_2}\int_{\Gamma_4} u_0\mathrm{d} x_2 \mathrm{d} t - \int_{T_1}^{T_2}\int_{\Gamma_2} u_0\mathrm{d} x_2 \mathrm{d} t\notag\\
      & = \int_{T_1}^{T_2}\int_{0}^{\varepsilon} e^{\lambda t}  (e^{\mu^{-\frac{1}{2}}(sd_1 + \Bi \sqrt{s^2 + \lambda}d_2)\varepsilon^l }-1) e^{\mu^{-\frac{1}{2}}(sd_2 - \Bi \sqrt{s^2 + \lambda}d_1)x_2}\mathrm{d} x_2 \mathrm{d} t\notag\\
       &= \frac{e^{\lambda T_1}  (e^{\lambda \varepsilon} - 1) }{\lambda \mu^{-\frac{1}{2}}  (sd_2 - \Bi \sqrt{s^2 + \lambda }d_1)} (e^{\mu^{-\frac{1}{2}}(sd_1 + \Bi \sqrt{s^2 + \lambda}d_2)\varepsilon^l }-1) (e^{\mu^{-\frac{1}{2}}(sd_2 - \Bi \sqrt{s^2 + \lambda}d_1)\varepsilon} - 1),\label{eq:I21}
    \end{align}
and
    \begin{align}
        &\left|\int_{T_1}^{T_2} \int_{\Gamma_2 \cup \Gamma_4} \delta_1(f_1^{\prime}-g_1^{\prime}) u_0\mathrm{d} x_2 \mathrm{d} t\right|\notag\\
        %\leq &2 C_3 C_{4}^{\alpha_3} \left|\int_{T_1}^{T_2} \int_{0}^{\varepsilon} e^{\lambda t}e^{(sd_1 + \Bi \sqrt{s^2 + \lambda}d_2)\varepsilon^l } e^{(sd_2 - \Bi \sqrt{s^2 + \lambda}d_1)x_2} (\varepsilon^{2l} + x_2^{2} + (t-t_0)^{2})^{\frac{\alpha_3 \alpha_4}{2}} \mathrm{d} \Bx \mathrm{d} t\right|\\
         \leq &2 C_3 C_{4}^{\alpha_3} \varepsilon^{\alpha_3 \alpha_4 l_0} \left|(e^{\mu^{-\frac{1}{2}} sd_1 \varepsilon^l }+1) \int_{T_1}^{T_2} \int_{0}^{\varepsilon} e^{\lambda t} e^{\mu^{-\frac{1}{2}} sd_2x_2}   \mathrm{d} x_2 \mathrm{d} t\right|\notag\\
        \leq & \frac{2 C_3 C_{4}^{\alpha_3}e^{\lambda T_1}}{|\lambda \mu^{-\frac{1}{2}} sd_2|}  |(e^{\mu^{-\frac{1}{2}} sd_1\varepsilon^l}+1) (e^{\lambda \varepsilon}-1)(e^{\mu^{-\frac{1}{2}} sd_2\varepsilon} -1)| \varepsilon^{\alpha_3 \alpha_4 l_0}\notag\\
        \leq &C(\mu,C_3,C_4)\varepsilon^{2+\alpha_3 \alpha_4 l_0} +\mathcal{O}(s\varepsilon^{2+l_0+\alpha_3 \alpha_4 l_0}).\label{eq:I22}
    \end{align}
    Utilizing \eqref{eq:I2}, \eqref{eq:I21}, and \eqref{eq:I22}, we deduce \eqref{eq:I_2}.
\end{proof}

\begin{Lemma}\label{lemma:I3}
    For the $I_3$ defined in \eqref{I1-6}, we have the following estimate:
\begin{equation*}
   % |I_3| \leq  |(f - g)_{(\Bx,t)=(\mathbf{0},t_0)}|e^{\lambda T_1} \varepsilon^{2+l} + \mathcal{O}(s\varepsilon^{2+l+l_0}),
   |I_3| \leq  C(\mu,C_2,C_4)(\varepsilon^{2+l}+\varepsilon^{2+l+\alpha_2 \alpha_4 l_0}) + \mathcal{O}(s\varepsilon^{2+l+l_0}),
\end{equation*}
where $l_0$ is defined in \eqref{eq:l0}, and $C(\mu,C_2,C_4)$ is a positive constant depending only on $\mu$, $C_2$ and $C_4$.
\end{Lemma}

\begin{proof} 
Noting that $\mu\in \mathbb{R}_+$ in Lemma \ref{lem:cgo} and that $d_1 < 0$, and combining \eqref{eq:beta}, we obtain
\begin{equation} \label{eq:sel}
  (1-e^{\mu^{-\frac{1}{2}} d_1})s \varepsilon^l \leq \big|e^{\mu^{-\frac{1}{2}} s d_1 \varepsilon^l}-1\big| \leq -\mu^{-\frac{1}{2}} d_1 s \varepsilon^l.
\end{equation}
Then, applying Lemma \ref{lemma:exp} and the mean value theorem, together with \eqref{eq:gamm2}, \eqref{eq:esti}, and \eqref{eq:sel}, we obtain
    \begin{equation*}
        \begin{aligned}
            |I_3| &=\left |\int_{T_1}^{T_2} \int _{D_{\varepsilon}} (f - g) u_0 \mathrm{d} \Bx \mathrm{d} t\right|\\
            & \leq \left|(f - g)_{(\Bx,t)=(\Bx_0,t_0)} \int_{T_1}^{T_2} \int _{D_{\varepsilon}}  u_0 \mathrm{d} \Bx \mathrm{d} t\right| \\
            &\quad + C(C_2 C_{4}) \left|\int_{T_1}^{T_2} \int _{D_{\varepsilon}}  (|\Bx-\Bx_0| + |t-t_0|^{\frac{1}{2}})^{\alpha_2 \alpha_4}  u_0 \mathrm{d} \Bx \mathrm{d} t\right|\\
            &\leq \big(|(f - g)_{(\Bx,t)=(\Bx_0,t_0)}| +  C(C_2 C_{4}) \varepsilon^{\alpha_2 \alpha_4 l_0}\big) \bigg|\int_{T_1}^{T_2} \int_{D_{\varepsilon}} e^{\lambda t}e^{\mu^{-\frac{1}{2}}(sd_1x_1 + sd_2x_2)} \mathrm{d} \Bx \mathrm{d} t\bigg| \\
           % &\quad + 2C_2 C_{4}^{\alpha_2} \varepsilon^{\alpha_2 \alpha_4 l_0}\big|\int_{T_1}^{T_2} \int_{D_{\varepsilon}} e^{\lambda t}e^{\mu^{-\frac{1}{2}}(sd_1x_1 + sd_2x_2)} \mathrm{d} \Bx \mathrm{d} t\big| \\
            &\leq  \frac{|(f - g)_{(\Bx,t)=(\Bx_0,t_0)}|+ C(C_2 C_{4}) \varepsilon^{\alpha_2 \alpha_4 l_0}}{\mu^{-1}\lambda s^2d_1d_2}e^{\lambda T_1}e^{\mu^{-\frac{1}{2}} sd_1 \gamma(\xi)}\\
            & \quad \times (e^{\lambda \varepsilon}-1)(e^{\mu^{-\frac{1}{2}}sd_2\varepsilon}-1)(e^{\mu^{-\frac{1}{2}}sd_1\varepsilon^l}-1)\\
            %&\quad +\frac{2C_2 C_{4}^{\alpha_2}}{\lambda s^2d_1d_2}e^{\lambda T_1 +sd_2 \gamma(\xi)} \varepsilon^{\alpha_2 \alpha_4 l}(e^{\lambda \varepsilon}-1)(e^{sd_2\varepsilon}-1)(e^{sd_1\varepsilon^l}-1)\\
            &\leq C(\mu,C_2) \varepsilon^{2+l}+ C(\mu,C_2,C_4)\varepsilon^{2+l+\alpha_2 \alpha_4 l_0} + \mathcal{O}(s\varepsilon^{2+l+l_0}),
        \end{aligned}
    \end{equation*}
    where $\xi \in (0,\varepsilon^l)$.
\end{proof}

\begin{Lemma}%\label{lemma:I4}
    For the term $I_4$ defined in \eqref{I1-6}, we have the decomposition 
    \begin{equation}\label{eq:I4}
            I_4 = I_{41} + I_{42} + I_{43} + I_{44},
    \end{equation}
with
\begin{equation*}
    \begin{aligned}
          |I_{41} - I_{21}|  \leq &   C(\mu,C_3) s^2 \varepsilon^{3+l} + \mathcal{O}(s^3 \varepsilon^{3+l+l_0}),\\
          |I_{42}|  \leq &  C(\mu,C_3,C_4) s \varepsilon^{2+l+\alpha_3 \alpha_4 l_0} + \mathcal{O}(s^2 \varepsilon^{2+l+l_0+\alpha_3 \alpha_4 l_0}), \\ 
          |I_{43}|  \geq & |(f_2^{\prime} -g_2^{\prime})|_{(\Bx,t) = (\Bx_0, t_0)}|  (1-e^{\mu^{-1 / 2}d_1}) e^{\lambda T_1}   s \varepsilon^{2+l} + \Oh(s^2 \varepsilon^{2+l+l_0}),\\
          |I_{44}|  \leq & C(\mu,C_3,C_4) s \varepsilon^{2+l+\alpha_3 \alpha_4 l_0} + \mathcal{O}(s^2 \varepsilon^{2+l+l_0+\alpha_3 \alpha_4 l_0}).
        \end{aligned}
\end{equation*}
Here, $I_{21}$ and $l_0$ are defined by \eqref{eq:I2} and \eqref{eq:l0}, respectively, $\Bx_0$ is an arbitrary point in $\overline{D}_{\varepsilon}$, $t_0 \in [T_1,T_2]$, and $C(\mu, C_3, C_4)$ is the same positive constant as in \eqref{eq:I_2}.
\end{Lemma}

\begin{proof}
Note that $\BF^{\prime},\mathbf{G}^{\prime} \in C^{\alpha_3}(\overline{D}_{\varepsilon})$ by virtue of \eqref{Cond:FG} and \eqref{eq:2FG}. Then, from Lemma \ref{lemma:exp} and \eqref{eq:cgo1}, we obtain the following H\"{o}lder expansion:
 
        \begin{align}\label{I-4}
            I_4 = & \int_{T_1}^{T_2} \int _{D_{\varepsilon}} (\BF^{\prime} -\mathbf{G}^{\prime}) \cdot \nabla u_0 \mathrm{d}\Bx \mathrm{d} t \notag \\
            =& \mu^{-\frac{1}{2}} \int_{T_1}^{T_2} \int _{D_{\varepsilon}} (sd_1 + \Bi \sqrt{s^2 + \lambda}d_2) (f_1^{\prime} - g_1^{\prime}) u_0 \mathrm{d}\Bx \mathrm{d} t \notag\\
            &+  \mu^{-\frac{1}{2}} \int_{T_1}^{T_2} \int _{D_{\varepsilon}} (sd_2 - \Bi \sqrt{s^2 + \lambda}d_1) (f_2^{\prime} - g_2^{\prime}) u_0 \mathrm{d} \Bx \mathrm{d} t  \notag \\
            =& \mu^{-\frac{1}{2}} (f_1^{\prime} -g_1^{\prime})|_{(\Bx,t) = (\mathbf{0}, t_0)} \int_{T_1}^{T_2} \int _{D_{\varepsilon}} (sd_1 + \Bi \sqrt{s^2 + \lambda}d_2)  u_0 \mathrm{d} \Bx \mathrm{d} t  \notag\\
            &+ \mu^{-\frac{1}{2}} \int_{T_1}^{T_2} \int _{D_{\varepsilon}} \delta_1 (f_1^{\prime} -g_1^{\prime}) (sd_1 + \Bi \sqrt{s^2 + \lambda}d_2)  u_0 \mathrm{d} \Bx \mathrm{d} t  \notag\\
            & + \mu^{-\frac{1}{2}}(f_2^{\prime} -g_2^{\prime})|_{(\Bx,t) = (\Bx_0, t_0)} \int_{T_1}^{T_2} \int _{D_{\varepsilon}} (sd_2 - \Bi \sqrt{s^2 + \lambda}d_1)  u_0 \mathrm{d} \Bx \mathrm{d} t  \notag\\
            &+ \mu^{-\frac{1}{2}} \int_{T_1}^{T_2} \int _{D_{\varepsilon}} \delta_2 (f_2^{\prime} -g_2^{\prime}) (sd_2 - \Bi \sqrt{s^2 + \lambda}d_1)  u_0 \mathrm{d} \Bx \mathrm{d} t  \notag\\
            :=& I_{41} + I_{42} + I_{43} + I_{44},
        \end{align}
%where $\Bx_0 \in \overline{D}_{\varepsilon}$. 
Using \eqref{eq:I2} and through direct calculation, we obtain
    \begin{align}
        |I_{41} -I_{21}|
      & =\left|(f_1^{\prime} -g_1^{\prime})|_{(\Bx,t)= (\mathbf{0}, t_0)}\right|\Big|  \mu^{-\frac{1}{2}}(sd_1 + \Bi \sqrt{s^2 + \lambda}d_2) \int_{T_1}^{T_2} e^{\lambda t} \mathrm{d} t \notag \\
      &\quad\,\, \times \int_{0}^{\varepsilon^{l}} \int_{\gamma(x_1)}^{\gamma(x_1)+\varepsilon}  \, e^{\mu^{-\frac{1}{2}}(sd_1 + \Bi \sqrt{s^2 + \lambda}d_2)x_1} e^{\mu^{-\frac{1}{2}}(sd_2 - \Bi \sqrt{s^2 + \lambda}d_1)x_2}\mathrm{d} x_2 \mathrm{d} x_1 \notag \\
      &\quad \,\,- \mu^{-\frac{1}{2}}(sd_1 + \Bi \sqrt{s^2 + \lambda}d_2)\int_{T_1}^{T_2} e^{\lambda t} \mathrm{d} t \int_{0}^{\varepsilon^l} \notag \\
      &\quad \,\,\times  \int_{0}^{\varepsilon}   e^{\mu^{-\frac{1}{2}}(sd_1 + \Bi \sqrt{s^2 + \lambda}d_2 )x_1 }  e^{\mu^{-\frac{1}{2}}(sd_2 - \Bi \sqrt{s^2 + \lambda}d_1)x_2}\mathrm{d} x_2 \mathrm{d} x_1 \Big| \notag \\
       &= \left|(f_1^{\prime} -g_1^{\prime})|_{(\Bx,t)= (\mathbf{0}, t_0)}\right| \Big|  \frac{sd_1 + \Bi \sqrt{s^2 + \lambda}d_2}{sd_2 - \Bi \sqrt{s^2 + \lambda}d_1}  (e^{\mu^{-\frac{1}{2}}(sd_2 - \Bi \sqrt{s^2 + \lambda}d_1)\varepsilon} -1) \notag\\
       &\quad\,\,\times \int_{T_1}^{T_2} \int_{0}^{\varepsilon^{l}} e^{\lambda t} e^{\mu^{-\frac{1}{2}}(sd_1 + \Bi \sqrt{s^2 + \lambda}d_2)x_1}
         (e^{\mu^{-\frac{1}{2}}(sd_2 - \Bi \sqrt{s^2 + \lambda}d_1)\gamma(x_1)} -1) \mathrm{d} x_1 \mathrm{d} t \Big|.\notag
    \end{align}
Note that 
\begin{equation*}
    \left|\frac{sd_1 + \Bi \sqrt{s^2 + \lambda}d_2}{sd_2 - \Bi \sqrt{s^2 + \lambda}d_1} \right|\leq 1+ \Oh(s^{-2}).
\end{equation*}
Applying \eqref{eq:gamm2} along with  $s \varepsilon \ll  1$, we further obtain
\begin{equation*}
    |e^{\mu^{-\frac{1}{2}} (sd_2 - \Bi \sqrt{s^2 + \lambda}d_1)\gamma(x_1)} -1|\leq \mu^{-\frac{1}{2}} s \varepsilon + \mathcal{O}(s^2\varepsilon^2).
\end{equation*}
Then using \eqref{eq:sel} and through direct calculation, we obtain the following result:
%\allowdisplaybreaks
    \begin{align}
        &\Big|\int_{T_1}^{T_2} \int_{0}^{\varepsilon^{l}} e^{\lambda t} e^{ \mu^{-\frac{1}{2}} (sd_1 + \Bi \sqrt{s^2 + \lambda}d_2)x_1} (e^{ \mu^{-\frac{1}{2}} (sd_2 - \Bi \sqrt{s^2 + \lambda}d_1)\gamma(x_1)} -1) \mathrm{d} x_1 \mathrm{d} t \Big| \notag \\
       & \leq \Big| \frac{e^{\lambda T_1}}{\lambda}(e^{\lambda \varepsilon}-1)\int_{0}^{\varepsilon^{l}} e^{\mu^{-\frac{1}{2}} sd_1x_1}|e^{\mu^{-\frac{1}{2}} (sd_2 - \Bi \sqrt{s^2 + \lambda}d_1)\gamma(x_1)} -1| \mathrm{d} x_1 \Big| \notag\\
      &\leq  \Big|\frac{1}{\mu^{-\frac{1}{2}} \lambda s d_1} e^{\lambda T_1}(e^{\lambda \varepsilon}-1)(e^{\mu^{-\frac{1}{2}}sd_1\varepsilon^l}-1)\Big| (\mu^{-\frac{1}{2}} s \varepsilon + \mathcal{O}(s^2\varepsilon^2))\notag\\
       &\leq C(\lambda,\mu) s \varepsilon^{2+l} + \mathcal{O}(s^2 \varepsilon^{2+l+l_0}).\notag
   \end{align}
Therefore, we obtain the estimate $|I_{41}-I_{21}| \leq C(\lambda,\mu,C_3) s^2 \varepsilon^{3+l} + \mathcal{O}(s^3 \varepsilon^{3+l+l_0})$, where positive constant $C(\mu,C_3)$ depends only on $\lambda$, $\mu$ and $C_3$, with $C_3$ is defined in \eqref{Cond:FG}.

For $I_{42}$ defined in \eqref{I-4}, applying Lemma \ref{lemma:exp} and \eqref{eq:sel}, and using a method similar to the one in \eqref{eq:esti}, we have the following estimate:
    \begin{align}
        I_{42}&=\Big| \mu^{-\frac{1}{2}} \int_{T_1}^{T_2} \int _{D_{\varepsilon}} \delta_1 (f_1^{\prime} -g_1^{\prime}) (sd_1 + \Bi \sqrt{s^2 + \lambda}d_2)  u_0 \mathrm{d} \Bx \mathrm{d} t\Big| \notag\\
        &\leq  \mu^{-\frac{1}{2}} C(C_3 C_{4}) \Big| (sd_1 + \Bi \sqrt{s^2 + \lambda}d_2) \int_{T_1}^{T_2} \int _{D_{\varepsilon}}  (|\Bx -\Bx_0|+ |t-t_0|^{\frac{1}{2}})^{\alpha_3 \alpha_4} u_0  \mathrm{d} \Bx \mathrm{d} t\Big| \notag\\
        &\leq  \mu^{-\frac{1}{2}} C(C_3 C_{4}) \varepsilon^{\alpha_3 \alpha_4 l_0}  \Big| (sd_1 + \Bi \sqrt{s^2 + \lambda}d_2) \int_{T_1}^{T_2} \int _{D_{\varepsilon}} e^{\lambda t} e^{ \mu^{-\frac{1}{2}} s d_1 x_1}  e^{ \mu^{-\frac{1}{2}} s d_2 x_2} \mathrm{d} \Bx \mathrm{d} t \Big| \notag\\
        & =  \frac{C(C_3 C_{4})}{\mu^{-\frac{1}{2}} \lambda s^2 |d_1 d_2|} \varepsilon^{\alpha_3 \alpha_4 l_0} e^{\lambda T_1}\Big|(e^{\lambda \varepsilon} - 1)  (e^{\mu^{-\frac{1}{2}} sd_2\varepsilon} - 1)(e^{\mu^{-\frac{1}{2}}sd_1 \varepsilon^l} - 1)\Big|\notag\\
        & \leq C(\mu,C_3,C_4) s \varepsilon^{2+l+\alpha_3 \alpha_4 l_0} + \mathcal{O}(s^2 \varepsilon^{2+l+l_0+\alpha_3 \alpha_4 l_0}).\label{eq:I42-2}
    \end{align}

The estimate for $I_{44}$ can be deduced analogously.
For $I_{43}$ defined in \eqref{I-4}, using \eqref{eq:sel}, and the mean value theorem, we obtain the following estimate:
        \begin{align}
            &\mu^{-\frac{1}{2}} \Big| (sd_2 - \Bi \sqrt{s^2 + \lambda}d_1) \int_{T_1}^{T_2} \int _{D_{\varepsilon}}  u_0 \mathrm{d} \Bx \mathrm{d} t\Big|\notag\\
            =&\mu^{-\frac{1}{2}} \Big| (sd_2 - \Bi \sqrt{s^2 + \lambda}d_1) \int_{T_1}^{T_2} e^{\lambda t} \mathrm{d} t \notag\\
            &\times\int_{0}^{\varepsilon^{l}} \int_{\gamma(x_1)}^{\gamma(x_1)+\varepsilon}  e^{\mu^{-\frac{1}{2}}(sd_1 + \Bi \sqrt{s^2 + \lambda}d_2)x_1 } e^{\mu^{-\frac{1}{2}} (sd_2 - \Bi \sqrt{s^2 + \lambda}d_1)x_2} \mathrm{d} x_2 \mathrm{d} x_1 \Big|\notag\\
            =&\Big|(e^{\mu^{-\frac{1}{2}} (sd_2 - \Bi \sqrt{s^2 + \lambda}d_1)\varepsilon} - 1) (e^{\lambda \varepsilon} - 1) \frac{e^{\lambda T_1}}{\lambda} \notag\\
            &\times \int_{0}^{\varepsilon^{l}}  e^{\mu^{-\frac{1}{2}} (sd_2 - \Bi \sqrt{s^2 + \lambda}d_1)\gamma(x_1)} e^{\mu^{-\frac{1}{2}}(sd_1 + \Bi \sqrt{s^2 + \lambda}d_2)x_1}\mathrm{d} x_1 \Big| \notag\\
            \geq &\Big|(e^{\mu^{-\frac{1}{2}} (sd_2 - \Bi \sqrt{s^2 + \lambda}d_1)\varepsilon} - 1) (e^{\lambda \varepsilon} - 1) \frac{e^{\lambda T_1}}{\lambda}\Big| \notag\\
             &\times \Big| \Re{\int_{0}^{\varepsilon^{l}}  e^{\mu^{-\frac{1}{2}}(sd_2 - \Bi \sqrt{s^2 + \lambda}d_1)\gamma(x_1)}e^{\mu^{-\frac{1}{2}}(sd_1 + \Bi \sqrt{s^2 + \lambda}d_2)x_1}\mathrm{d} x_1 }\Big| \notag\\
            \geq & \Big|\frac{\cos{(\mu^{-\frac{1}{2}} \sqrt{s^2+\lambda}(d_1\gamma(\xi_1)-d_2\xi_1))}}{\mu^{-\frac{1}{2}} \lambda  {sd_1}} e^{\lambda T_1}e^{\mu^{-\frac{1}{2}} sd_2\gamma(\xi_1)} \notag\\
            &\times (e^{\mu^{-\frac{1}{2}} (sd_2 - \Bi \sqrt{s^2 + \lambda}d_1)\varepsilon} - 1) (e^{\lambda \varepsilon} - 1)  (e^{ \mu^{-\frac{1}{2}} sd_1\varepsilon^l} - 1)  \Big| \notag\\
            \geq &   (1-e^{ \mu^{-\frac{1}{2}} d_1}) e^{\lambda T_1}  e^{\mu^{-\frac{1}{2}} sd_2\gamma(\xi_1)} s \varepsilon^{2+l} + \Oh(s^2 \varepsilon^{2+l+l_0}), \label{eq:ge1}
        \end{align}
       where $\xi_1\in(0,\varepsilon^l)$ and $l_0$ is defined in \eqref{eq:l0}.
\end{proof}
 
\begin{Lemma}\label{lemma:I5}
    For the $I_5$  defined in \eqref{I1-6}, we have the following estimation
    \begin{equation*}
       | I_5 |\leq C(C_1,C_4)\varepsilon^{2+l+\alpha_1 \alpha_4 l_0} +   \mathcal{O}(s\varepsilon^{2+l+l_0+\alpha_1 \alpha_4 l_0}),
       %2C_1 C_4^{\alpha_1} \lambda e^{\lambda T_1}\varepsilon^{2+l+\alpha_1 \alpha_4 l_0} +   \mathcal{O}(s\varepsilon^{2+l+l_0+\alpha_1 \alpha_4 l_0}),
    \end{equation*}
 where $l_0$ is defined in \eqref{eq:l0}, and $C(C_1, C_4)$ is a positive constant depending only on $C_1$ and $C_4$.
\end{Lemma}

\begin{proof}
For any point $(\Bx,t) \in D_{\varepsilon} \times [0,T]$, we have the following expansion
    \begin{align}
        H(u)=& H(u(\mathbf{0},t_0)) + \delta H(u(\Bx,t)),\notag\\
        H(v)=& H(v(\mathbf{0},t_0)) + \delta H(v(\Bx,t)),\label{eq:H}
    \end{align}
where
\begin{align*}
|\delta H(u(\Bx,t))| &\leq C_1 C_4^{\alpha_1}  (|\Bx|+|t-t_0|^{\frac{1}{2}})^{\alpha_1 \alpha_4},\\
|\delta H(v(\Bx,t))| &\leq C_1 C_4^{\alpha_1}  (|\Bx|+|t-t_0|^{\frac{1}{2}})^{\alpha_1 \alpha_4}.
\end{align*}
Therefore, combining the mean value theorem, and using  \eqref{Cond:H}, \eqref{eq:gamm2}, \eqref{eq:l0}, and \eqref{eq:sel}, we obtain
\begin{equation*}
    \begin{aligned}
        |I_5| =&\left | \int_{D_{\varepsilon}}(H(u)-H(v))u_0|_{T_1}^{T_2} \mathrm{d} \Bx\right|\\
        \leq &C_1 C_4^{\alpha_1} \varepsilon^{\alpha_1 \alpha_4 l_0} e^{\lambda T_1} \big|(e^{\lambda \varepsilon}-1) \int_{D_{\varepsilon}} e^{\mu^{-\frac{1}{2}} sd_1x_1+\mu^{-\frac{1}{2}} sd_2x_2} \mathrm{d} \Bx\big| \\
        = & \frac{C_1 C_4^{\alpha_1} \varepsilon^{\alpha_1 \alpha_4 l_0}}{\mu^{-1} \lambda s^2 |d_1d_2|}  e^{\lambda T_1}e^{\mu^{-\frac{1}{2}}  sd_2\gamma(\xi_2)} \big| (e^{\lambda \varepsilon}-1)(e^{\mu^{-\frac{1}{2}}  sd_1\varepsilon}-1)(e^{\mu^{-\frac{1}{2}} sd_2\varepsilon^{l}}-1)\big|\\
        \leq & C(C_1,C_4)\varepsilon^{2+l+\alpha_1 \alpha_4 l_0} +   \mathcal{O}(s\varepsilon^{2+l+l_0+\alpha_1 \alpha_4 l_0}),
    \end{aligned}
\end{equation*}
where $\xi_2\in(0,\varepsilon^l)$, and $l_0$ is defined in \eqref{eq:l0}.
\end{proof}

\begin{Lemma}\label{lemma:I6}
    For the $I_6$ defined in \eqref{I1-6}, we have the following estimation
    \begin{equation*}
        |I_6| \leq C(C_1,C_4)  \varepsilon^{2 + l +\alpha_1 \alpha_4 l_0} + \mathcal{O}(s\varepsilon^{2+l+l_0+\alpha_1 \alpha_4 l_0}),
        %C_6 \lambda e^{\lambda T_1 } \varepsilon^{2 + l +\alpha_1 \alpha_4 l_0} + \mathcal{O}(s\varepsilon^{2+l+l_0+\alpha_1 \alpha_4 l_0}),
    \end{equation*}
    where $l_0$ is defined in \eqref{eq:l0} and positive constant $ C(C_1,C_4)$ depends on $C_1$ and $C_4$.
\end{Lemma}

\begin{proof}
Using \eqref{eq:H}, the boundary condition $u = v$, and the H\"older expansion of $w$, we obtain 
\begin{equation*}
   \begin{aligned}
        &|w -(H(u)-H(v))| \\
        \leq & \Big(\|w\|_{C^{\alpha_4}(D_{\varepsilon} \times [0,T])} +\|H\|_{C^{\alpha_1}} ( \|u\|_{C^{\alpha_4}(D_{\varepsilon} \times [0,T])}^{\alpha_1}+ \|v\|_{C^{\alpha_4}(D_{\varepsilon} \times [0,T])}^{\alpha_1})\Big) (|\Bx|+|t-t_0|^{\frac{1}{2}})^{\alpha_1 \alpha_4}.
   \end{aligned}
\end{equation*}
Therefore, combining \eqref{Cond:uv}, \eqref{Cond:H}, \eqref{eq:l0}, and \eqref{eq:sel}, we have 
    \begin{equation*}
        \begin{aligned}
            |I_6| =& \left|\int_{T_1}^{T_2} \int _{D_{\varepsilon}}\lambda (w -(H(u)-H(v))) u_0\mathrm{d} \Bx \mathrm{d} t\right|\\
            \leq & \lambda (C_4 +2C_1 C_4^{\alpha_1}) \varepsilon^{\alpha_1 \alpha_4 l_0} \int_{T_1}^{T_2} \int _{D_{\varepsilon}} |u_0| \mathrm{d} \Bx \mathrm{d} t\\
            \leq & \frac{(C_4 +2C_1 C_4^{\alpha_1}) e^{\lambda T_1 + \mu^{-\frac{1}{2}}sd_2 \gamma(\xi_3)}}{\mu^{-1} s^2|d_1d_2|}\varepsilon^{\alpha_1 \alpha_4 l_0} \Big|(e^{\lambda \varepsilon}-1)(e^{\mu^{-\frac{1}{2}}sd_2\varepsilon}-1)(e^{\mu^{-\frac{1}{2}}sd_1\varepsilon^l}-1)\Big|\\
            \leq & C(C_1,C_4)  \varepsilon^{2 + l +\alpha_1 \alpha_4 l_0} + \mathcal{O}(s\varepsilon^{2+l+l_0+\alpha_1 \alpha_4 l_0}),
        \end{aligned}
    \end{equation*}
  where $\xi_3\in(0,\varepsilon^l)$, and $l_0$ is defined in \eqref{eq:l0}.
\end{proof}

We now prove Theorem \ref{main:thm} in two-dimensional case. The proof is divided into two independent parts: we first establish statement (a), and then prove statement (b). It is important to note that (a) and (b) are separate conclusions, neither depends on the other. %To ensure notational consistency, we set $u^1 := u$, $u^2 := v$, $f^1 := f$, $f^2 := g$, $f_k^1 := f_k$, and $f_k^2 := g_k$ for $k = 1, 2$ in Theorem \ref{main:thm}. Note that $\BF^{\prime} = (f_1^{\prime}, f_2^{\prime})^{\top}$ is obtained by coordinate transformation from $\BF = (f_1, f_2)^{\top}$ in \eqref{eq:2FG}. For convenience, we will use the original variables $u,\,v,\,f,\,g,\,f_k,\,g_k$ from Lemma \ref{lem:w} in proving Theorem \ref{main:thm}.

\begin{proof}
We first prove part (a) of Theorem \ref{main:thm}. Note that $u^j$ is the solution to \eqref{eq:main-mu} with the configuration $(H, \mathbf{F}^j, f^j)$ for $j=1,2$. If the boundary measurements coincide, i.e., $\mathcal{M}_{\mathscr{F}^1, \mathfrak{f}^1} = \mathcal{M}_{\mathscr{F}^2, \mathfrak{f}^2}$, then $(u^1, u^2)$ constitutes a solution to the coupled PDE system \eqref{eq:uv} with $\mathbf{F}^{\prime} = \mathbf{F}^{1\prime}$, $\mathbf{G}^{\prime} = \mathbf{F}^{2\prime}$, $f = f^1$, $g = f^2$, and $h(\mathbf{x}, t) = h_{\mathbf{F}^1} - h_{\mathbf{F}^2}$, where $\mathbf{F}^{1\prime}$ and $\mathbf{F}^{2\prime}$ are as defined in \eqref{eq:2FG}.

Applying \eqref{I1-6}, \eqref{eq:I2-2}, and \eqref{eq:I4}, we obtain
    \begin{equation}\label{main:I1-6}
        I_1 + I_{22} -I_3 -I_{42} - I_{44} - I_{5} -I_{6} = I_{41}+I_{43} -I_{21},
    \end{equation}
which means that
\begin{equation}\notag
    |I_{43}| \leq |I_1| + |I_{22}| + |I_3| + |I_{42}| + |I_{44}| + |I_5|+ |I_6|+|I_{41}-I_{21}|.
\end{equation}
The combination of Lemmas \ref{lemma:I1}--\ref{lemma:I6} with equation \eqref{main:I1-6} yields the following result:
    \begin{align*}
       & |(f_2^{1\prime} -f_2^{2\prime})|_{(\Bx,t) = (\Bx_0, t_0)}| (1-e^{ \mu^{-\frac{1}{2}} d_1})  e^{\lambda T_1} e^{\mu^{-\frac{1}{2}} sd_2\gamma(\xi_1)} s \varepsilon^{2+l}\\
       \leq & |I_1| + |I_{22}|  + |I_3| + |I_{42}| + |I_{44}| + |I_5| +|I_6| +|I_{41}-I_{21}|+ \mathcal{O}(s^2 \varepsilon^{2+l+l_0})\\
       \leq & C(\mu,\|h\|_{L^{\infty}},C_1,C_2,C_3,C_4)\big(s^2\varepsilon^{3+l}+s^2\varepsilon^{2+l+(1+\alpha_4)l_0}+\varepsilon^{2+\alpha_4 l_0} +\varepsilon^{2+\alpha_3\alpha_4 l_0}+ \varepsilon^{2+l}\\
       &+s\varepsilon^{2+l+\alpha_3\alpha_4 l_0} +\varepsilon^{2+l+\alpha_1\alpha_4 l_0}\big)+\Oh(s^2\varepsilon^{2+l+l_0})+ \Oh(s^3 \varepsilon^{3+2l})+\Oh(s \varepsilon^{3+\alpha_3\alpha_4 l_0})\\
       &+\Oh( \varepsilon^{2+l+\alpha_2 \alpha_4 l_0}),
    \end{align*}
where positive constant $C( \mu,\|h\|_{L^{\infty}},C_1,C_2,C_3,C_4)$ depends on $\mu$, $\|h\|_{L^\infty}$, $C_1$, $C_2$, $C_3$, and $C_4$. Taking $s=\varepsilon^{-\beta}$, with
\begin{equation*}
    \beta =\begin{cases}l & \text { if } 0<l\leq \frac{1}{1+\alpha_3 \alpha_4}, \\ \frac{1}{2}\left[1+l\left(1-\alpha_3 \alpha_4\right)\right], & \text { if } \frac{1}{1+\alpha_3 \alpha_4}<l<1, \\ \frac{1}{2}\left(1+l-\alpha_3 \alpha_4\right), & \text { if } 1 \leq l<1+\alpha_3 \alpha_4. \end{cases}
\end{equation*}
 We conclude that
\begin{equation}\label{eq:result-f}
    |f_2^{1\prime}(\Bx_0,t_0,u(\Bx_0,t_0)) -f_2^{2\prime}(\Bx_0,t_0,v(\Bx_0,t_0))| \leq C(\mu,\|h\|_{L^{\infty}},C_1,C_2,C_3,C_4) \varepsilon^{\tau},
\end{equation}
where  \begin{equation}\label{eq:tau}
    \tau = \begin{cases}\alpha_3 \alpha_4l & \text { if } \,\,0<l\leq \frac{1}{1+\alpha_3 \alpha_4}, \\ \frac{1}{2}\left[1-l\left(1-\alpha_3 \alpha_4\right)\right] & \text { if } \,\, \frac{1}{1+\alpha_3 \alpha_4}<l<1, \\ \frac{1}{2}\left(1-l+\alpha_3 \alpha_4 \right) & \text { if } \,\, 1 \leq l<1+\alpha_3 \alpha_4. \end{cases}
\end{equation}
By the triangle inequality,
\begin{align}%\label{eq:tri}
&|f_2^{1\prime}(\Bx_0,t_0,u^1(\Bx_0,t_0)) -f_2^{2\prime}(\Bx_0,t_0,u^1(\Bx_0,t_0))|\\ \notag
\leq &|f_2^{1\prime}(\Bx_0,t_0,u^1(\Bx_0,t_0)) -f_2^{2\prime}(\Bx_0,t_0,u^2(\Bx_0,t_0))| + |f_2^{2\prime}(\Bx_0,t_0,u^1(\Bx_0,t_0)) -f_2^{2\prime}(\Bx_0,t_0,u^2(\Bx_0,t_0))|.
\end{align}
The boundary condition $u^1=u^2$ in Lemma \ref{lem:w} and the regularity of $u^1, u^2$ in \eqref{Cond:uv} imply
\begin{equation}\label{eq:minus}
|u^1(\Bx_0,t_0) -u^2(\Bx_0,t_0)| \leq C(C_4) \varepsilon^{l_0},
\end{equation}
where $l_0$ is defined in \eqref{eq:l0}. Consequently, utilizing \eqref{eq:minus} and the regularity of $g_2'$ in \eqref{Cond:FG} and \eqref{eq:2FG} yields that
\begin{equation}\label{eq:gminus}
|f_2^{2\prime}(\Bx_0,t_0,u^1(\Bx_0,t_0)) -f_2^{2\prime}(\Bx_0,t_0,u^2(\Bx_0,t_0))| \leq C(C_2,C_4) \varepsilon^{\alpha_3 l_0}.
\end{equation}
Note that $\tau \leq \alpha_3 l_0$ follows from \eqref{eq:l0} and \eqref{eq:tau}. Moreover, estimates \eqref{eq:result-f}--\eqref{eq:gminus} together imply
\begin{equation*}
|f_2^{1\prime}(\Bx_0,t_0,u^1(\Bx_0,t_0)) -f_2^{2\prime}(\Bx_0,t_0,u^1(\Bx_0,t_0))| \leq C(\mu,\|h\|_{L^{\infty}},C_1,C_2,C_3,C_4) \varepsilon^{\tau}.
\end{equation*}
%We begin by recalling the operator definition in \eqref{eq:oper}. Let $\mathscr{F}^{\prime}: H^1(\Omega;[0,T]) \rightarrow L^2(\Omega;[0,T])^2$ be defined by $u \mapsto \mathscr{F}^{\prime}(u)(\Bx,t) := \BF^{\prime}(\Bx,t, u)$, where $\mathscr{F}^{\prime} = (\mathscr{F}_1^{\prime},\mathscr{F}_2^{\prime})^{\top}$ and $\BF^{\prime} = (f_1^{\prime},f_2^{\prime})^{\top}$. Similarly, define $\mathscr{G}^{\prime}: H^1(\Omega;[0,T]) \rightarrow L^2(\Omega;[0,T])^2$ by $u \mapsto \mathscr{G}^{\prime}(u)(\Bx,t) := \mathbf{G}^{\prime}(\Bx,t, u)$, with $\mathscr{G}^{\prime} = (\mathscr{G}_1^{\prime},\mathscr{G}_2^{\prime})^{\top}$ and $\mathbf{G}^{\prime} = (g_1^{\prime},g_2^{\prime})^{\top}$. 
Based on the definition in \eqref{eq:2FG}, we conclude that
\begin{equation}\label{eq:op FG}
    |\mathscr{F}_2^{1\prime}(u^1)(\Bx_0,t_0) -\mathscr{F}_2^{2\prime}(u^1)(\Bx_0,t_0)|\leq  C(\mu,\|h\|_{L^{\infty}},C_1,C_2,C_3,C_4) \varepsilon^{\tau}.
\end{equation}
The same estimate holds at $v(\Bx_0, t_0)$ by symmetry:
\begin{equation*}
|\mathscr{F}_2^{1\prime}(u^2)(\Bx_0,t_0) -\mathscr{F}_2^{2\prime}(u^2)(\Bx_0,t_0)|\leq  C(\mu,\|h\|_{L^{\infty}},C_1,C_2,C_3,C_4) \varepsilon^{\tau}.
\end{equation*}
Therefore, we obtain the estimate \eqref{Re:FG}.

We next prove statement (b) in Theorem \ref{main:thm}. From the condition that $\BF - \mathbf{G} = \mathbf{0}$ and $h = 0$ hold simultaneously, together with Lemma \ref{lemma:I1-6}, we obtain
\begin{equation}\label{I135}
    I_{1} = I_3 -I_5-I_6,
\end{equation}
where $I_1,I_3,I_5,I_6$ are defined in \eqref{I1-6}. 

Similar to the proof of the case (a), we analyze $I_1,\,I_3,\,I_5,\,I_6$ item by item. Firstly, for $I_1$, we slightly modify the analysis process of Lemma \ref{lemma:I1}; specifically, we modify the estimate of $\partial_{\nu}w$ on $\Gamma_2$ and $\Gamma_4$. Note that for any point $\Bx\in\mathscr{N}_{\varepsilon}^{\ell}$, there exists at least one point $\Bx_0\in \hat{\Gamma}$ such that the vector $\Bx-\Bx_0$ is parallel to $\nu_{\Bx_{0}}$, here $\nu_{\Bx_{0}}$ is the unit outward normal at $\Bx_0$, based on the geometric assumption in Theorem \ref{main:thm}(b). This implies that 
\begin{equation*}%\label{eq:x-x0}
    |\Bx-\Bx_0|\leq \varepsilon.
\end{equation*}
Then, by applying the boundary condition in \eqref{main:eqw}, we derive the expansion for any point $(\Bx,t) \in \overline{\mathscr{N}_{\varepsilon}^{\ell}} \times [T_1,T_2]$ for $\varepsilon \ll 1$:
\begin{align}\label{eq:Rwb}
        w(\Bx,t) &=  w(\Bx_0,t_0) + \nabla_{\Bx} w(\Bx_0,t_0) \cdot (\Bx-\Bx_0)+ \partial_t w(\Bx_0,t_0) (t-t_0)+\mathcal{R}_{w},\notag\\
        \partial_{\nu} w(\Bx,t) &= \partial_{\nu} w(\Bx_0,t_0) + \mathcal{R}_{\partial_{\nu}w},
\end{align}
with
\begin{align}
|\mathcal{R}_{w|\Gamma_2}|&\leq C_4 (|\Bx-\Bx_0| + |t-t_0|^{\frac{1}{2}}) ^{1+\alpha_4}\leq C_4 \varepsilon ^{1+\alpha_4},\notag\\ 
|\mathcal{R}_{\partial w | \Gamma_2}| &\leq C_4  (|\Bx-\Bx_0| + |t-t_0|^{\frac{1}{2}})^{\alpha_4}\leq C_4\varepsilon^{\alpha_4 },\label{eq:r2b}
\end{align}
for some $(\Bx_0,t_0)\in \hat{\Gamma}\times[T_1,T_2]$ and $T_2-T_1=\varepsilon^2$. Hence, utilizing \eqref{eq:Rwb} and \eqref{eq:r2b}, for any point $(\Bx,t) \in \overline{\mathscr{N}_{\varepsilon}^{\ell}} \times [T_1,T_2]$ for $\varepsilon \ll 1$, we obtain
\begin{equation}\label{eq:wb}
     w(\Bx,t) \leq C_4\varepsilon ^{1+\alpha_4},\,\, \partial_{\nu} w(\Bx,t) \leq C_4 \varepsilon ^{\alpha_4 }.
\end{equation}
Therefore, utilizing \eqref{eq:Gamma_2 par u0}, \eqref{eq:Gamma_4 par u0}, and \eqref{eq:wb}, we obtain
\begin{equation}\label{I1-b}
    |I_{1}| \leq C(\mu,C_4) (s\varepsilon^{3+\alpha_4} + \varepsilon^{2+\alpha_4 })+ \Oh(s^2\varepsilon^{4+\alpha_4}) +\Oh(s \varepsilon^{3+\alpha_4}),
\end{equation}
where $C(\mu,C_4)$ is a positive constant depending only on $\mu$ and $C_4$.

For the $I_3$  defined in \eqref{I1-6}, we have 
\begin{equation*}    
\begin{aligned}        
I_3 =& \int_{T_1}^{T_2} \int _{D_{\varepsilon}} (f^1 - f^2) u_0 \mathrm{d} \Bx \mathrm{d} t\\        
=& (f^1 - f^2)_{(\Bx,t)=(\Bx_0,t_0)}\int_{T_1}^{T_2} \int _{D_{\varepsilon}}  u_0 \mathrm{d} \Bx \mathrm{d} t + \int_{T_1}^{T_2} \int _{D_{\varepsilon}}   \delta_0(f^1 - f^2)u_0 \mathrm{d} \Bx \mathrm{d} t\\        
:=& I_{31} + I_{32}.    
\end{aligned}
\end{equation*}
Furthermore, using a method similar to that in \eqref{eq:ge1}, we deduce that
\begin{equation*}    
\begin{aligned}        
&\Big|\int_{T_1}^{T_2} \int _{D_{\varepsilon}}  u_0 \mathrm{d} \Bx \mathrm{d} t\Big|  \\        
&\geq \left|\frac{\cos{(\sqrt{s^2+\lambda}(d_1\gamma(\xi_4)-d_2\xi_4))}}{ \mu^{-\frac{1}{2}} \lambda d_1 s^2\sqrt{1+\frac{\lambda d_1^2}{s^2}}} e^{\lambda T_1} e^{\mu^{-\frac{1}{2}}sd_2\gamma(\xi_4)}(e^{\lambda \varepsilon} - 1)\right. \\
&\left. \quad \times  (e^{\mu^{-\frac{1}{2}} (sd_2 - \Bi \sqrt{s^2 + \lambda}d_1)\varepsilon} - 1) (e^{\mu^{-\frac{1}{2}} sd_1\varepsilon^l} - 1)\right |\\  
&\ge \mu^{\frac{1}{2}} (1-e^{ \mu^{-\frac{1}{2}} d_1}) e^{\lambda T_1}  e^{\mu^{-\frac{1}{2}}sd_2\gamma(\xi_4)}  \varepsilon^{2+l} + \Oh(s \varepsilon^{2+2l}),
%&= e^{\lambda T_1} \varepsilon^{2+l} + \mathcal{O}(s\varepsilon^{2+l+l_0})    
\end{aligned}
\end{equation*}
where $\xi_4 \in (0,\varepsilon^l)$. And using the similar analysis of lemma \ref{lemma:I3}, we obtain 
\begin{equation*}    
\begin{aligned}        
&\left|\int_{T_1}^{T_2} \int _{D_{\varepsilon}}   \delta_0(f^1 - f^2)u_0 \mathrm{d} \Bx \mathrm{d} t\right|\\        
&\leq \frac{C(C_2,C_4)}{\mu^{-1}\lambda s^2|d_1d_2|}e^{\lambda T_1 } e^{sd_2 \gamma(\xi_5)}\varepsilon^{\alpha_2 \alpha_4 l_0}(e^{\lambda \varepsilon}-1)(e^{\mu^{-\frac{1}{2}}sd_2\varepsilon}-1)(e^{\mu^{-\frac{1}{2}}sd_1\varepsilon^l}-1)\\        
&\leq C(\mu,C_2,C_4) \varepsilon^{2+l+\alpha_2 \alpha_4 l} + \mathcal{O}(s\varepsilon^{2+2l+\alpha_2 \alpha_4 l}),    
\end{aligned}
\end{equation*}
where $\xi_5 \in (0,\varepsilon^l)$. Therefore, we have the following estimate
\begin{equation}\label{I3-b}    
I_3 = I_{31} + I_{32},
\end{equation}
with
\begin{equation*}    
\begin{aligned}        
|I_{31}| \geq & |(f^1 - f^2)_{(\Bx,t)=(\Bx_0,t_0)}|\mu^{\frac{1}{2}} (1-e^{ \mu^{-\frac{1}{2}} d_1}) e^{\lambda T_1}  e^{\mu^{-\frac{1}{2}}sd_2\gamma(\xi_4)}  \varepsilon^{2+l} +  \mathcal{O}(s\varepsilon^{2+2l}),  \\        
|I_{32}| \leq & C(\mu,C_2,C_4) \varepsilon^{2+l+\alpha_2\alpha_4 l} + \mathcal{O}(s\varepsilon^{2+2l+\alpha_2\alpha_4 l}),    
\end{aligned}
\end{equation*}
where $C(\mu,C_2,C_4)$ is a constant depending on $\mu,C_2$, and $C_4$.

As for $I_5$ defined in \eqref{I1-6}, by slightly modifying the estimate of $H(u)-H(v)$, we can obtain a new estimate of $I_5$. Since $H\in C^{1,\alpha}(\mathbb{C})$ in \eqref{Cond:Hb}, and in view of the regularity properties of $u$ in \eqref{Cond:uv}, we have the following estimate:
\begin{equation*}
    H(u^1) = H(u^1(\Bx_0,t_0)) + H_u(u^1(\mathbf{0},t_0))(u^1(\Bx,t)-u^1(\Bx_0,t_0)) + \mathcal{R}(\Bx,t),
\end{equation*}
where the remainder $\mathcal{R}(\Bx,t)$ satisfies
\begin{equation*}
    |\mathcal{R}(\Bx,t)| \leq \|H\|_{C^{1, \alpha_1}} \left|u^1(\Bx,t)-u^1\left(\Bx_0,t_0\right)\right|^{1+\alpha_1} \leq C_1^{\prime}C_{4}^{1+\alpha_1} \big(|\Bx-\Bx_0|+|t-t_0|^{\frac{1}{2}}\big)^{1+\alpha_1}.
\end{equation*}
A similar expansion holds for $H(u^2)$. Applying the boundary condition $u^1(\Bx_0,t_0) =u^2(\Bx_0,t_0)$, we get
\begin{equation}\label{eq:Hu-Hv}
    |H(u^1)-H(u^2)| \leq |H_u(u^1(\mathbf{0},t_0))(u^1(\Bx,t)-u^2(\Bx,t))| + 2C_1^{\prime}C_{4}^{1+\alpha_1} \big(|\Bx-\Bx_0|+|t-t_0|^{\frac{1}{2}}\big)^{1+\alpha_1},
\end{equation}
where $C_1^{\prime}$ is specified in \eqref{Cond:Hb}.
%Since the tangential derivative and normal derivative of $w$ at the origin are both 0, we can get
%\begin{equation*}
%    w(\Bx,t)=w(\mathbf{0},t_0) + (x_1,x_2,t-t_0)\cdot (\partial_1 w(\mathbf{0},t_0),\partial_2 w(\mathbf{0},t_0),\partial_t w(\mathbf{0},t_0)) + \delta w(\Bx,t) = \delta w(\Bx,t),
%\end{equation*}
%where $|\delta w(\Bx,t)| \leq \|w\|_{C^{1,\alpha}(D_{\varepsilon} \times [0,T])} |(\Bx,t)-(\mathbf{0},t_0)|^{1+\alpha_4}$. Therefore, we have
 Noting the regularity of $u,v $ in \eqref{Cond:uv} and applying \eqref{eq:wb} and \eqref{eq:Hu-Hv}, we deduce 
\begin{equation*}
    |H(u^1)-H(u^2)| \leq C(C_1^{\prime},C_4) \big(|\Bx-\Bx_0|+|t-t_0|^{\frac{1}{2}}\big)^{1+\min\{\alpha_1,\alpha_4\}} \leq  C(C_1^{\prime},C_4) \varepsilon^{1+\min\{\alpha_1,\alpha_4\}}.
\end{equation*}
Then, we obtain 
\begin{equation}\label{I5-b}
    |I_5| \leq C(\mu,C_1^{\prime},C_4)\varepsilon^{3+l+\min\{\alpha_1,\alpha_4\}}  + \mathcal{O}(s\varepsilon^{3+2l+\min\{\alpha_1,\alpha_4\}}),
\end{equation}
where $ C(\mu,C_1^{\prime},C_4)$ is a positive constant depending only on $\mu$, $C_1^{\prime}$, and $C_4$.

%Using the above analysis process of $I_5$, we can get
%\begin{equation*}
%    |w-(H(u)-H(v))| \leq (2\|H\|_{C^{1,\alpha_1}} +1) (\|u\|_{C^{1, \alpha_4}}^{1+\alpha_1} + \|v\|_{C^{1, \alpha_4}}^{1+\alpha_1})\left|(\Bx,t)-(\mathbf{0},t_0)\right|^{1+\min\{\alpha_1,\alpha_4\}}.
%\end{equation*}
Similarly, applying \eqref{eq:wb} and \eqref{eq:Hu-Hv}, we obtain the following estimate for $I_6$:
\begin{equation}\label{I6-b}
    |I_6| \leq C(\mu,C_1^{\prime},C_2)\varepsilon^{3+l+\min\{\alpha_1,\alpha_4\}}  + \mathcal{O}(s\varepsilon^{3+2l+\min\{\alpha_1,\alpha_4\}}),
\end{equation}
where $ C(\mu,C_1^{\prime},C_4)$ is a positive constant depending only on $\lambda$, $\mu$, $C_1^{\prime}$, and $C_4$.

Combining \eqref{eq:gamm2}, \eqref{eq:beta}, \eqref{I135}, \eqref{I1-b}, \eqref{I3-b}, \eqref{I5-b}, and \eqref{I6-b}, we arrive at
\begin{align}
&|(f^1 - f^2)_{(\Bx,t)=(\Bx_0,t_0)}|\mu^{\frac{1}{2}} (1-e^{ \mu^{-\frac{1}{2}} d_1}) e^{\lambda T_1}  e^{\mu^{-\frac{1}{2}}sd_2\gamma(\xi_4)}  \varepsilon^{2+l}\notag\\
\leq & C(\mu,C_1^{\prime},C_2,C_4) \big(s\varepsilon^{3+\alpha_4} + \varepsilon^{2+\alpha_4} +\varepsilon^{2+l+\alpha_2\alpha_4 l}+ \varepsilon^{3+l+\min\{\alpha_1,\alpha_4\}}\big) \notag\\
&+\Oh(s^2\varepsilon^{4+\alpha_4}) +\Oh(s \varepsilon^{3+\alpha_4})  + \mathcal{O}(s\varepsilon^{3+l+\alpha_2\alpha_4}) + \mathcal{O}(s\varepsilon^{2+2l}),\notag
\end{align}
where $C$ depends on $C_1^{\prime},C_2,C_4$. Taking $s = \varepsilon^{-\beta}$ with $\beta = \frac{(1-\alpha_2\alpha_4)\alpha_4}{1+\alpha_2 \alpha_4}$, we obtain
\begin{equation*}
\begin{split}
    &|f^1(\Bx_0,t_0,u(\Bx_0,t_0),\nabla u(\Bx_0,t_0),\Delta u(\Bx_0,t_0))-f^2(\Bx_0,t_0,v(\Bx_0,t_0),\nabla v(\Bx_0,t_0),\Delta v(\Bx_0,t_0))|\\
    &\leq  C(\mu,C_1^{\prime},C_2,C_4) \varepsilon^{\tau_1},
\end{split}
\end{equation*}
where $\tau_1=\frac{\alpha_2\alpha_4^2}{1+\alpha_2 \alpha_4}$. Similar to the analysis of case $(a)$, we can get 
\begin{equation}\label{eq:fg 2D}
    |\mathfrak{f}^1(\tilde{u})(\Bx_0,t_0) -\mathfrak{f}^2(\tilde{u})(\Bx_0,t_0)| \leq  C(\mu,C_1^{\prime},C_2,C_4) \varepsilon^{\tau_1},
\end{equation}
where $\tilde{u} =u$ or $v$. Therefore, the estimate \eqref{Re:fg} is proved.
\end{proof}

\section{Proof of Theorem \ref{main:thm} in three dimensions}\label{sec:proof-3D}
Following an approach analogous to the two-dimensional case, we provide a detailed proof of Theorem \ref{main:thm} for the three-dimensional setting. For any subregion of $\mathscr{N}_{\varepsilon}^{\ell}$, define
\begin{equation*}
    {D}_{\varepsilon}:=\Omega_{\varepsilon} \times \gamma(x_2) \subset \mathscr{N}_{\varepsilon}^{\ell}, \, x_2 \in(0, \varepsilon^{l}) \subset I,
\end{equation*}
and lateral boundary of ${D}_{\varepsilon}$
\begin{equation*}
    {\Gamma}_{\varepsilon}:=\partial \Omega_{\varepsilon} \times \gamma(x_2),
\end{equation*}
where $I$ is defined in Definition \ref{def:ND}.
 
For the CGO solution defined in \eqref{u0} in three dimensions, we choose $\mathbf{d}$ and $\mathbf{d}^\bot$ as follows:
\begin{equation}\label{cond:d-3D}
    \mathbf{d}:=(d_1, d_2, d_3)^{\top} \in \mathbb{S}^2 \text { and } \mathbf{d}^{\top}:=(d_2, -d_1, 0)^{\top}
\end{equation}
such that $d_1,d_2,d_3<0$.

Throughout this section, we consistently assume that the unit vector $\mathbf{d}$ in the form of the CGO solution, as given by \eqref{u0}, satisfies \eqref{cond:d-3D}. Following the approach for 2D case, we first analyze the coupled system \eqref{eq:uv} to derive key estimates for its solutions. For notational consistency, we define 
\begin{align}\label{I1-6-3D}
        I_1 =&  \int_{T_1}^{T_2} \int_{{\Omega}_{\varepsilon} \cup {\Omega}_{\varepsilon}^{\prime}}\mu (w \partial_{\nu} u_0 - u_0 \partial_{\nu} w) \mathrm{d} \sigma \mathrm{d} t, I_2 = \int_{T_1}^{T_2} \int_{{\Omega}_{\varepsilon} \cup {\Omega}_{\varepsilon}^{\prime}}\nu \cdot (\BF^{\prime} -\mathbf{G}^{\prime}) u_0\mathrm{d} \sigma \mathrm{d} t,\nonumber\\
        I_3 =&  \int_{T_1}^{T_2} \int _{D_{\varepsilon}} (f - g) u_0 \mathrm{d} \Bx \mathrm{d} t,\qquad\qquad \ \
        I_4 = \int_{T_1}^{T_2} \int _{D_{\varepsilon}} (\BF^{\prime} -\mathbf{G}^{\prime}) \cdot \nabla u_0 \mathrm{d} \Bx \mathrm{d} t,\\
        I_5 =&  \int_{D_{\varepsilon}}(H(u)-H(v))u_0|_{T_1}^{T_2} \mathrm{d} \Bx,\qquad\quad\ \
        I_6 = \int_{T_1}^{T_2} \int _{D_{\varepsilon}}\lambda (w -(H(u)-H(v))) u_0\mathrm{d} \Bx \mathrm{d} t,\nonumber
\end{align}
where $\BF^{\prime}=(f_1^{\prime},f_2^{\prime},f_3^{\prime})^{\top}$ and $\mathbf{G}^{\prime}=(g_1^{\prime},g_2^{\prime},g_3^{\prime})^{\top}$ are defined similarly by \eqref{eq:3FG}.

In the three-dimensional case, we consider two geometries: nozzle domains and slab domains.
\subsection{ Nozzle domains} \label{sub:nozzle}
Recall that the diameter of $\Omega_{\varepsilon}$ is $\varepsilon$, as mentioned in Subsection \ref{sub:geometric}. Thus, the boundary of ${D}_{\varepsilon}$ is given by
\begin{equation*}
    \partial {D}_{\varepsilon}={\Gamma}_{\varepsilon} \cup {\Omega}_{\varepsilon} \cup {\Omega}_{\varepsilon}^{\prime},
\end{equation*}
where
\begin{equation*}
    \begin{aligned}
        & {\Omega}_{\varepsilon}=\left\{(x_1, x_2, x_3) \mid x_2=0, x_1 \in E, x_3 \in(0, \varepsilon)\right\}, \\
        & {\Omega}_{\varepsilon}^{\prime}=\{(x_1, x_2, x_3) \mid x_2=\varepsilon^{l}, x_1 \in E, x_3 \in(\gamma(\varepsilon^{l}), \gamma(\varepsilon^{l})+\varepsilon)\}.
        \end{aligned}
\end{equation*}
Here, the length (or diameter) of the interval $E$ is $\varepsilon$, which means $\mathrm{diam}(E) = \sup\{ |\Bx - \mathbf{y}| : \Bx, \mathbf{y} \in E \} = \varepsilon$. See Figure \ref{fg:02} for an illustration.

\begin{figure}[htbp]
    \centering
    \begin{subfigure}[c]{0.45\textwidth}
        \centering
        \includegraphics[width=\linewidth]{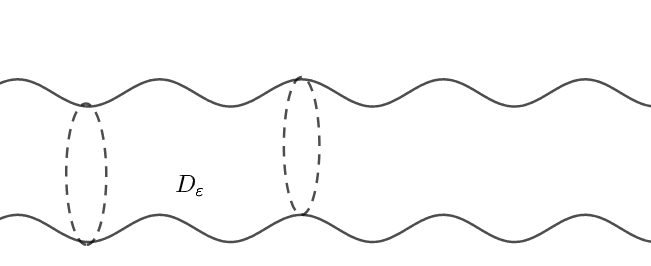}
       % \caption{thin end}
        %\label{fig:image1}
    \end{subfigure}
    \hfill
    \begin{subfigure}[c]{0.45\textwidth}
        \centering
        \includegraphics[width=\linewidth]{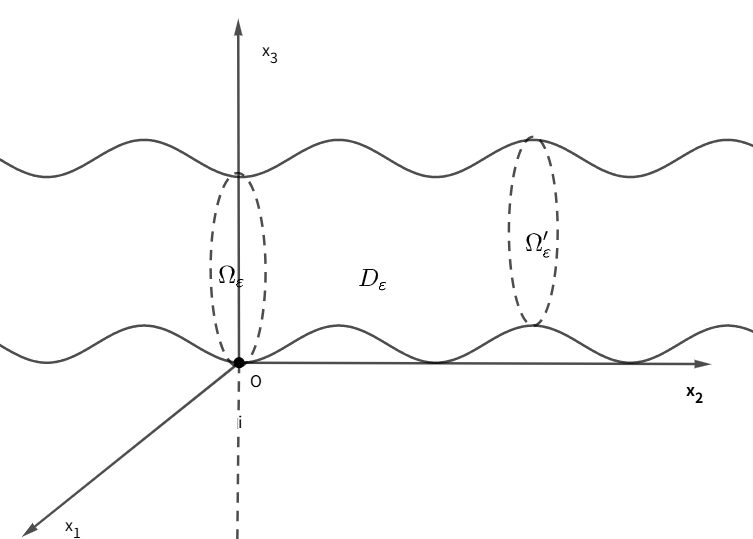}
        %\caption{$D_\varepsilon $}
        %\label{fig:image2}
    \end{subfigure}
    \caption{illustration of the nozzle domain in 3D}
    \label{fg:02}
\end{figure}

\begin{Lemma}\label{lemma:I1-3D}
   For $I_1$ defined in \eqref{I1-6-3D}, we have the following estimate:
   \begin{equation}\label{I1-te}
    |I_1|\leq C(\mu,\|h\|_{L^\infty},C_4) (s^2\varepsilon^{4+l}+ s^2\varepsilon^{3+l+(1+\alpha_4)l_0} + \varepsilon^{3+\alpha_4 l_0}) + \mathcal{O}(s^3\varepsilon^{4+2l}) + \mathcal{O}(s\varepsilon^{4+\alpha_4 l_0}),
   \end{equation}
where $l_0$ is defined in \eqref{eq:l0} and $C(\mu,\|h\|_{L^\infty},C_4)$ is a positive constant depending only on  $\mu$, $\|h\|_{L^\infty}$, and $C_4$.
\end{Lemma}

\begin{proof}
    For the CGO solution definded in \eqref{u0}, we can get that
    \begin{equation*}
        \begin{aligned}
            u_0(\Bx,t) &= e^{\lambda t} \cdot e^{\mu^{-\frac{1}{2}}(sd_1 + \Bi \sqrt{s^2 + \lambda}d_2)x_1 + \mu^{-\frac{1}{2}}(sd_2 - \Bi \sqrt{s^2 + \lambda}d_1)x_2 +\mu^{-\frac{1}{2}}sd_3 x_3},\\
            \nabla u_0(\Bx,t) &= \mu^{-\frac{1}{2}} (sd_1 + \Bi \sqrt{s^2 + \lambda}d_2,sd_2 - \Bi \sqrt{s^2 + \lambda}d_1,sd_3)^{\top} u_0(\Bx,t),
        \end{aligned}
    \end{equation*}
 On the boundary ${\Omega}_{\varepsilon}$, analogous to \eqref{eq:Gamma_2 w}--\eqref{eq:r2}, we obtain
    \begin{equation*}
        \begin{aligned}
            w(\Bx,t) =& w(\mathbf{0},t_0) + (x_1,0,x_3) \cdot (\partial_1 w(\mathbf{0},t_0),\partial_2 w(\mathbf{0},t_0),\partial_3 w(\mathbf{0},t_0)) \\ &+\partial_t w(\mathbf{0},t_0)(t-t_0)+ \mathcal{R}_{w|{\Omega}_{\varepsilon}} \\
            =& - h(\mathbf{0},u,v)_{(\Bx,t)=(\mathbf{0},t_0)}x_3 + \mathcal{R}_{w|{\Omega}_{\varepsilon}},\\
            \partial_{\nu} w(\Bx,t) =&(0,-1,0)\cdot (\partial_1 w(\Bx,t),\partial_2 w(\Bx,t),\partial_3 w(\Bx,t)) = - \partial_2 w(\Bx,t) \\
            =&- \partial_2 w(\mathbf{0},t_0) + \mathcal{R}_{\partial w|{\Omega}_{\varepsilon}} =\mathcal{R}_{\partial w|{\Omega}_{\varepsilon}}, \\
            \partial_{\nu} u_0 =&- \mu^{-\frac{1}{2}}(sd_2 - \Bi \sqrt{s^2 + \lambda}d_1) e^{\lambda t} e^{\mu^{-\frac{1}{2}}(sd_1 - \Bi \sqrt{s^2 + \lambda}d_2)x_1 +\mu^{-\frac{1}{2}}sd_3x_3 },
        \end{aligned}
    \end{equation*}
where
\begin{align*}
    |\mathcal{R}_{w|{\Omega}_{\varepsilon}}| &\leq C_4 ((|\Bx|+ |t-t_0|^{\frac{1}{2}}) ^{1+\alpha_4}) \leq C_4 \varepsilon ^{1+\alpha_4},\\
    |\mathcal{R}_{\partial w|{\Omega}_{\varepsilon}}| &\leq C_4 ((|\Bx|+ |t-t_0|^{\frac{1}{2}}) ^{\alpha_4}) \leq C_4 \varepsilon ^{\alpha_4},
\end{align*}
for some $t_0\in [T_1,T_2]$ and $T_2-T_1=\varepsilon^2$.

In a manner analogous to \eqref{eq:Gamma_4 w}--\eqref{eq:r4}, the following relation holds on the boundary ${\Omega}_{\varepsilon}^{\prime}$:
     \begin{equation*}
        \begin{aligned}
            w(\Bx,t) %=& w(\mathbf{0},t_0) + (x_1,\varepsilon^l,x_3,t_0) \cdot (\partial_1 w(\mathbf{0},t_0),\partial_2 w(\mathbf{0},t_0),\partial_1 w(\mathbf{0},t_0),\partial_t w(\mathbf{0},t_0)) + \mathcal{R}_{w|{\Omega}_{\varepsilon}^{\prime}} \\[0.5mm]
            =& - h(\mathbf{0},u,v)_{(\Bx,t)=(\mathbf{0},t_0)}x_3 + \mathcal{R}_{w|{\Omega}_{\varepsilon}^{\prime}},\\
            \partial_{\nu} w(\Bx,t) =&(0,1,0)\cdot (\partial_1 w(\Bx,t),\partial_2 w(\Bx,t),\partial_3 w(\Bx,t)) =  \partial_2 w(\Bx,t) \\[0.5mm]
            =& \partial_2 w(\mathbf{0},t_0) + \mathcal{R}_{\partial w|{\Omega}_{\varepsilon}^{\prime}} =\mathcal{R}_{\partial w|{\Omega}_{\varepsilon}^{\prime}}, \\[0.5mm]
            \partial_{\nu} u_0 =& \mu^{-\frac{1}{2}}(sd_2 - \Bi \sqrt{s^2 + \lambda}d_1) e^{\lambda t} e^{\mu^{-\frac{1}{2}}(sd_2-
            \Bi\sqrt{s^2+\lambda}d_1)\varepsilon^l} e^{\mu^{-\frac{1}{2}}(sd_1 - \Bi \sqrt{s^2 + \lambda}d_2)x_1 +\mu^{-\frac{1}{2}}sd_3x_3 },
        \end{aligned}
    \end{equation*}
where
\begin{align*}
    |\mathcal{R}_{w|{\Omega}_{\varepsilon}^{\prime}}| &\leq C_4 ((|\Bx|+ |t-t_0|^{\frac{1}{2}}) ^{1+\alpha_4}) \leq C_4 \varepsilon ^{(1+\alpha_4)l_0},\\
    |\mathcal{R}_{\partial w|{\Omega}_{\varepsilon}^{\prime}}| &\leq C_4 ((|\Bx|+ |t-t_0|^{\frac{1}{2}}) ^{\alpha_4}) \leq C_4 \varepsilon ^{\alpha_4 l_0}.
\end{align*}
where $l_0$ is denoted in \eqref{eq:l0}. Following the approach used in Lemma \ref{lemma:I1}, we obtain
    \begin{equation*}
        \begin{aligned}
        &\bigg|\int_{T_1}^{T_2}\int_{{\Omega}_{\varepsilon} \cup {\Omega}_{\varepsilon}^{\prime}} w \partial_{\nu} u_0  \mathrm{d} \sigma \mathrm{d} t\bigg| \\
        &\leq \bigg|   \mu^{-\frac{1}{2}}(sd_2 - \Bi \sqrt{s^2 + \lambda}d_1) (e^{\mu^{-\frac{1}{2}}(sd_2 - \Bi \sqrt{s^2 + \lambda}d_1)\varepsilon^l}-1) \frac{e^{\lambda T_1}}{\lambda} (e^{\lambda \varepsilon}-1) \int_{E}\int_0^{\varepsilon}\\
        &\quad \times ( -h(\mathbf{0},u,v)_{(\Bx,t)=(\mathbf{0},t)}x_3 + C_4 \varepsilon ^{(1+\alpha_4)l_0}) e^{\mu^{-\frac{1}{2}}(sd_1 - \Bi \sqrt{s^2 + \lambda}d_2)x_1 +\mu^{-\frac{1}{2}}sd_3x_3 } \mathrm{d} x_3 \mathrm{d} x_1\bigg|\\
        &\leq \frac{\mu^{-\frac{1}{2}} |h(\mathbf{0},u,v)_{(\Bx,t)=(\mathbf{0},t_0)}| e^{\lambda T_1}  s^2 \varepsilon^{l + 4}}{2} +C_4 \mu^{-\frac{1}{2}} e^{\lambda T_1}s^2 \varepsilon^{3+l+(1+\alpha_4)l_0} + \mathcal{O}(s^3\varepsilon^{4+2l}).
        \end{aligned}
    \end{equation*}
    Moreover, we have
    \begin{equation*}
      \left |\int_{T_1}^{T_2}\int_{{\Omega}_{\varepsilon} \cup {\Omega}_{\varepsilon}^{\prime}}  u_0 \partial_{\nu}w  \mathrm{d} \sigma \mathrm{d} t\right| \leq C_4 e^{\lambda T_1} \varepsilon^{3+\alpha_4 l_0} + \mathcal{O}(s \varepsilon^{4+\alpha_4 l_0}). 
    \end{equation*}
    
    From the above analysis, the definition of $I_1$ in \eqref{I1-6-3D}, and detailed calculations, we obtain \eqref{I1-te}.

\end{proof}

\begin{Lemma}\label{lemma:I2-3D}
    For term $I_2$, defined in \eqref{I1-6-3D}, can be decomposed as
\begin{equation}\label{est:I2-3D}
        I_2 = I_{21} + I_{22},
\end{equation}
with 
    \begin{align}
        I_{21}=
        &{(f_2^{\prime}-g_2^{\prime})_{(\Bx,t)=(\mathbf{0},t_0)}} \int_{T_1}^{T_2} \int_{{\Omega}_{\varepsilon}}  (e^{\mu^{-\frac{1}{2}}sd_3 \gamma(\varepsilon^l)}e^{\mu^{-\frac{1}{2}}(sd_2 - \Bi \sqrt{s^2 + \lambda}d_1)\varepsilon^l} -1) \notag\\
        &\times e^{\lambda t} \, e^{\mu^{-\frac{1}{2}}(sd_1 + \Bi \sqrt{s^2 + \lambda}d_2)x_1 +\mu^{-\frac{1}{2}}sd_3 x_3} \mathrm{d} x_3 \mathrm{d}x_1 \mathrm{d} t,\notag\\
        |I_{22}| \leq &  C(C_3,C_4) e^{\lambda T_1} \varepsilon^{3+\alpha_3 \alpha_4 l_0} + \mathcal{O}(s\varepsilon^{3+l+\alpha_3\alpha_4 l_0}),\label{eq:I22-TE}
    \end{align}
where $l_0$ is defined in \eqref{eq:l0}, and $C(\mu,C_3,C_4)$ is a positive constant depending only on $\mu$, $C_3$, and $C_4$.
\end{Lemma}

\begin{proof} 
 Parallel to the proof of Lemma \ref{lemma:I2}, we can show
  
        \begin{align}\label{eq:I2-3D}
         I_2 =& \int_{T_1}^{T_2} \int_{{\Omega}_{\varepsilon} \cup {\Omega}_{\varepsilon}^{\prime}}\nu \cdot (\BF^{\prime} -\mathbf{G}^{\prime}) u_0\mathrm{d} \sigma \mathrm{d} t \notag\\
         =& \int_{T_1}^{T_2} \int_{{\Omega}_{\varepsilon}} (0,-1,0) \cdot (\BF^{\prime} -\mathbf{G}^{\prime}) u_0\mathrm{d} \sigma \mathrm{d} t + \int_{T_1}^{T_2} \int_{{\Omega}_{\varepsilon}^{\prime}} (0,1,0) \cdot (\BF^{\prime} -\mathbf{G}^{\prime}) u_0\mathrm{d} \sigma \mathrm{d} t \notag \\
         =& \int_{T_1}^{T_2} \int_{{\Omega}_{\varepsilon}} -(f_2^{\prime}-g_2^{\prime}) u_0\mathrm{d} \sigma \mathrm{d} t + \int_{T_1}^{T_2} \int_{{\Omega}_{\varepsilon}^{\prime}} (f_2^{\prime}-g_2^{\prime}) u_0\mathrm{d} \sigma \mathrm{d} t \notag \\
         %=& \int_{T_1}^{T_2} \int_{{\Omega}_{\varepsilon}} -(f_2^{\prime}-g_2^{\prime})_{(\Bx,t)=(\mathbf{0},t_0)} u_0\mathrm{d} \Bx \mathrm{d} t - \int_{T_1}^{T_2} \int_{{\Omega}_{\varepsilon}} \delta(f_2^{\prime}-g_2^{\prime}) u_0\mathrm{d} \Bx \mathrm{d} t\\
        % &+ \int_{T_1}^{T_2} \int_{{\Omega}_{\varepsilon}^{\prime}} (f_2^{\prime}-g_2^{\prime})_{(\Bx,t)=(\mathbf{0},t_0)} u_0\mathrm{d} \Bx \mathrm{d} t + \int_{T_1}^{T_2} \int_{{\Omega}_{\varepsilon}^{\prime}} \delta(f_2^{\prime}-g_2^{\prime}) u_0\mathrm{d} \Bx \mathrm{d} t\\
         =&(f_2^{\prime}-g_2^{\prime})_{(\Bx,t)=(\mathbf{0},t_0)}\left(\int_{T_1}^{T_2}\int_{{\Omega}_{\varepsilon}^{\prime}} u_0\mathrm{d} \sigma \mathrm{d} t - \int_{T_1}^{T_2}\int_{{\Omega}_{\varepsilon}} u_0\mathrm{d} \sigma \mathrm{d} t\right) \notag\\
         &+\left( \int_{T_1}^{T_2} \int_{{\Omega}_{\varepsilon}^{\prime}} \delta_2 (f_2^{\prime}-g_2^{\prime}) u_0\mathrm{d} \sigma \mathrm{d} t - \int_{T_1}^{T_2} \int_{{\Omega}_{\varepsilon}} \delta_2 (f_2^{\prime}-g_2^{\prime}) u_0\mathrm{d} \sigma \mathrm{d} t\right) \notag\\
         :=& I_{21} +I_{22}.
        \end{align}
Furthermore, we have 
    \begin{align}\label{eq:I21-3D}
        &\int_{T_1}^{T_2}\int_{{\Omega}_{\varepsilon}^{\prime}} u_0\mathrm{d} \sigma \mathrm{d} t - \int_{T_1}^{T_2}\int_{{\Omega}_{\varepsilon}} u_0\mathrm{d} \sigma \mathrm{d} t \notag\\
        =& \int_{T_1}^{T_2} \int_{{\Omega}_{\varepsilon}} e^{\lambda t} e^{\mu^{-\frac{1}{2}}(sd_1 + \Bi \sqrt{s^2 + \lambda}d_2)x_1 + \mu^{-\frac{1}{2}}(sd_2 - \Bi \sqrt{s^2 + \lambda}d_1)\varepsilon^l +\mu^{-\frac{1}{2}}sd_3 (x_3 +\gamma(\varepsilon^l))} \mathrm{d} x_3 \mathrm{d}x_1 \mathrm{d} t\notag\\
        & - \int_{T_1}^{T_2} \int_{{\Omega}_{\varepsilon}} e^{\lambda t} e^{\mu^{-\frac{1}{2}}(sd_1 + \Bi \sqrt{s^2 + \lambda}d_2)x_1 +\mu^{-\frac{1}{2}}sd_3 x_3} \mathrm{d} x_3 \mathrm{d}x_1 \mathrm{d} t \notag\\
       =& \int_{T_1}^{T_2} \int_{{\Omega}_{\varepsilon}}  (e^{\mu^{-\frac{1}{2}}sd_3 \gamma(\varepsilon^l)}e^{\mu^{-\frac{1}{2}}(sd_2 - \Bi \sqrt{s^2 + \lambda}d_1)\varepsilon^l} -1) e^{\lambda t} \notag\\
       & \times e^{\mu^{-\frac{1}{2}}(sd_1 + \Bi \sqrt{s^2 + \lambda}d_2)x_1 +\mu^{-\frac{1}{2}}sd_3 x_3} \mathrm{d} x_3 \mathrm{d}x_1 \mathrm{d} t,
    \end{align}
and
\begin{equation}\label{eq:I22-3D}
    \begin{aligned}
        &\left|\int_{T_1}^{T_2} \int_{\Omega \cup {\Omega}_{\varepsilon}^{\prime}} \delta(f_1^{\prime}-g_1^{\prime}) u_0\mathrm{d} \sigma \mathrm{d} t\right|\\
        &\leq C(C_3,C_4) \varepsilon^{\alpha_3 \alpha_4 l_0} (e^{\mu^{-\frac{1}{2}}(sd_2 - \Bi \sqrt{s^2 + \lambda}d_1)\varepsilon^l}+1)  \Big|  \int_{T_1}^{T_2}\int_{0}^{\varepsilon} \int_{0}^{\varepsilon}e^{\lambda t} e^{\mu^{-\frac{1}{2}}s d_1 x_1  +\mu^{-\frac{1}{2}}sd_3 x_3} \mathrm{d} \sigma \mathrm{d} t\Big|\\
       & \leq  \frac{  C(C_3,C_4) \mu}{s^2 \lambda d_1 d_3} e^{\lambda T_1}  \big|(e^{\lambda \varepsilon -1}) (e^{\mu^{-\frac{1}{2}}sd_1\varepsilon}-1)(e^{\mu^{-\frac{1}{2}}sd_3\varepsilon} -1)\big| \varepsilon^{\alpha_3 \alpha_4 l_0}\\
        &\leq  C(C_3,C_4) e^{\lambda T_1} \varepsilon^{3+\alpha_3 \alpha_4 l_0} + \mathcal{O}(s\varepsilon^{3+l_0+\alpha_3 \alpha_4 l_0}).
    \end{aligned}
\end{equation}
Employing \eqref{eq:I2-3D}, \eqref{eq:I21-3D} and \eqref{eq:I22-3D}, we derive \eqref{est:I2-3D}.
\end{proof}

\begin{Lemma}%\label{lemma:I3-3D}
    For the $I_3$  defined in \eqref{I1-6-3D}, we have the following estimate
    \begin{equation}\label{eq:I3-TE}
       | I_3 | \leq  C(\mu,C_2,C_4)(\varepsilon^{3+l}+\varepsilon^{3+l+\alpha_2\alpha_4l_0})+ \mathcal{O}(s\varepsilon^{3+l+l_0}),
    \end{equation}
where $l_0$ is defined in \eqref{eq:l0}, and $C(\mu,C_2,C_4)$ is a positive constant depending only on $\mu$, $C_2$ and $C_4$.
\end{Lemma}

\begin{proof}
Analogous to the proof of Lemma \ref{lemma:I3}, we have
%By lemma \ref{lemma:exp}, we can get
        \begin{align*}
            |I_3| =& \left|\int_{T_1}^{T_2} \int _{D_{\varepsilon}} (f - g) u_0 \mathrm{d} \Bx \mathrm{d} t\right|\\
            \leq & |(f - g)_{(\Bx,t)=(\Bx_0,t_0)}\int_{T_1}^{T_2} \int _{D_{\varepsilon}}  u_0 \mathrm{d} \Bx \mathrm{d} t| \\
            &+ C(C_2 C_4) \left |\int_{T_1}^{T_2} \int _{D_{\varepsilon}}  u_0  \big(|\Bx-\Bx_0| + |t-t_0|^{\frac{1}{2}}\big)^{\alpha_2 \alpha_4} \mathrm{d} \Bx \mathrm{d} t\right|\\
            \leq & (|(f - g)_{(\Bx,t)=(\Bx_0,t_0)}|+ C(C_2 C_4)  \varepsilon^{\alpha_2\alpha_4 l_0} ) \int_{T_1}^{T_2} \int_{D_{\varepsilon}} |e^{\lambda t}e^{\mu^{-\frac{1}{2}}(sd_1x_1 + sd_2x_2+sd_3x_3)} |\mathrm{d} \Bx \mathrm{d} t \\
            \leq & \frac{|(f - g)_{(\Bx,t)=(\Bx_0,t_0)}|+ C(C_2 C_4)  \varepsilon^{\alpha_2\alpha_4 l_0}}{|\mu^{-\frac{3}{2}} \lambda s^3d_1d_2d_3|}e^{\lambda T_1+\mu^{-\frac{1}{2}}sd_3 \gamma(\xi_6)} (e^{\lambda \varepsilon}-1)\\
            &\times (e^{\mu^{-\frac{1}{2}}sd_1\varepsilon}-1)(e^{\mu^{-\frac{1}{2}}sd_2\varepsilon^l}-1)(e^{\mu^{-\frac{1}{2}}sd_3\varepsilon}-1)\\
            \leq &  C(\mu,C_2,C_4)(\varepsilon^{3+l}+\varepsilon^{3+l+\alpha_2\alpha_4l_0}) + \mathcal{O}(s\varepsilon^{3+l+l_0}),
            %\leq &|(f - g)_{(\Bx,t)=(\mathbf{0},t_0)}|e^{\lambda T_1}\varepsilon^{3+l}+ C(\mu,C_2,C_4)\varepsilon^{3+l+\alpha_2\alpha_4l_0} + \mathcal{O}(s\varepsilon^{3+l+l_0}),
        \end{align*}
where $\xi_6 \in (0,\varepsilon^l)$.
\end{proof}

\begin{Lemma}\label{lemma:I4-3D}
    For the $I_4$ defined in \eqref{I1-6-3D}, we have the following decomposition
    \begin{equation}\label{eq:I4-3}
            I_4 = I_{41} + I_{42} + I_{43} + I_{44} + I_{45} + I_{46},
    \end{equation}
with
    \begin{align*}
          |I_{43} - I_{21}|  \leq& C(\mu,C_3)s^2 \varepsilon^{4+l} + \mathcal{O}(s^3 \varepsilon^{4+l+l_0}),\\
          |I_{4j}|  \leq & {C(\mu,C_3,C_4)} e^{\lambda T_1} s \varepsilon^{3+l+\alpha_3 \alpha_4 l_0} + \mathcal{O}(s^2 \varepsilon^{3+l+l_0+\alpha_3 \alpha_4 l_0}), \, j=2,4,6,\\
          |I_{41} + I_{45}|  \geq & \big||(f_1^{\prime} -g_1^{\prime})|_{(\Bx,t) = (\Bx_0, t_0)} sd_1 -|(f_3^{\prime} -g_3^{\prime})|_{(\Bx,t) = (\Bx_0, t_0)}sd_3\big|(1-e^{\mu^{-1/2}d_2}) e^{\lambda T_1} \varepsilon^{3+l}\\ &\times e^{\mu^{-\frac{1}{2}}sd_1 x_{1,\xi}}  e^{\mu^{-\frac{1}{2}}sd_3\gamma(x_{2,\xi})} +\mathcal{O}(s^2\varepsilon^{3+l+l_0}),  
          %|I_{44}|  \leq & {C(\mu,C_3,C_4)} e^{\lambda T_1} s \varepsilon^{3+l+\alpha_3 \alpha_4 l_0} + \mathcal{O}(s^2 \varepsilon^{3+l+l_0+\alpha_3 \alpha_4 l_0}),
        \end{align*}
where $f_k^{\prime}$ ($k=1,3$) are defined in \eqref{eq:3FG}, $g_k^{\prime}$ are defined analogously to $f_k^{\prime}$, $x_{2,\xi} \in (0,\varepsilon^l)$, $x_{1,\xi} \in E$, $\Bx_0$ is an arbitrary point in $\overline{D}_{\varepsilon}$, and $t_0 \in [T_1,T_2]$, with $l_0$ and $I_{21}$ defined by \eqref{eq:l0} and \eqref{eq:I2-3D}, respectively.
\end{Lemma}

\begin{proof}
It follows from \eqref{Cond:FG} and \eqref{eq:2FG} that $\BF^{\prime},\mathbf{G}^{\prime} \in C^{\alpha_3}(\overline{D}_{\varepsilon})$. Combining Lemma \ref{lemma:exp} with \eqref{eq:cgo1} then gives the following H\"older expansion:
        \begin{align}\label{I-4-3D}
            I_4 = & \int_{T_1}^{T_2} \int _{D_{\varepsilon}} (\BF^{\prime} -\mathbf{G}^{\prime}) \cdot \nabla u_0 \mathrm{d} \Bx \mathrm{d} t \notag\\
            =&\int_{T_1}^{T_2} \int _{D_{\varepsilon}} \mu^{-\frac{1}{2}}(sd_1 + \Bi \sqrt{s^2 + \lambda}d_2) (f_1^{\prime} - g_1^{\prime}) u_0 \mathrm{d}\Bx \mathrm{d} t+\int_{T_1}^{T_2} \int _{D_{\varepsilon}} \mu^{-\frac{1}{2}}sd_3 (f_3^{\prime} - g_3^{\prime}) u_0 \mathrm{d}\Bx \mathrm{d} t \notag \\
            &+ \int_{T_1}^{T_2} \int _{D_{\varepsilon}} \mu^{-\frac{1}{2}}(sd_2 - \Bi \sqrt{s^2 + \lambda}d_1) (f_2^{\prime} - g_2^{\prime}) u_0 \mathrm{d} \Bx \mathrm{d} t\notag\\
            =& (f_1^{\prime} -g_1^{\prime})|_{(\Bx,t) = (\Bx_0, t_0)} \int_{T_1}^{T_2} \int _{D_{\varepsilon}} \mu^{-\frac{1}{2}}(sd_1 + \Bi \sqrt{s^2 + \lambda}d_2)  u_0 \mathrm{d} \Bx \mathrm{d} t \notag\\
            &+ \int_{T_1}^{T_2} \int _{D_{\varepsilon}} \delta_1 (f_1^{\prime} -g_1^{\prime}) \mu^{-\frac{1}{2}}(sd_1 + \Bi \sqrt{s^2 + \lambda}d_2)  u_0 \mathrm{d} \Bx \mathrm{d} t\notag \\
            & + (f_2^{\prime} -g_2^{\prime})|_{(\Bx,t) = (\mathbf{0}, t_0)} \int_{T_1}^{T_2} \int _{D_{\varepsilon}} \mu^{-\frac{1}{2}}(sd_2 - \Bi \sqrt{s^2 + \lambda}d_1)  u_0 \mathrm{d} \Bx \mathrm{d} t \notag\\
            &+ \int_{T_1}^{T_2} \int _{D_{\varepsilon}} \delta_2 (f_2^{\prime} -g_2^{\prime}) \mu^{-\frac{1}{2}}(sd_2 - \Bi \sqrt{s^2 + \lambda}d_1)  u_0 \mathrm{d} \Bx \mathrm{d} t\notag\\
            &+(f_3^{\prime} -g_3^{\prime})|_{(\Bx,t) = (\Bx_0, t_0)} \int_{T_1}^{T_2} \int _{D_{\varepsilon}} \mu^{-\frac{1}{2}}sd_3  u_0 \mathrm{d} \Bx \mathrm{d} t \notag\\
            &+ \int_{T_1}^{T_2} \int _{D_{\varepsilon}} \delta_3 (f_3^{\prime} -g_3^{\prime}) \mu^{-\frac{1}{2}}sd_3   u_0 \mathrm{d} \Bx \mathrm{d} t\notag\\
            :=& I_{41} + I_{42} + I_{43} + I_{44} + I_{45} + I_{46}.
        \end{align}
Combining \eqref{eq:I2-3D} and \eqref{eq:I21-3D}, we find
    \begin{align*}
        &|I_{43} -I_{21}| \\
       =&\bigg |(f_2^{\prime} -g_2^{\prime})|_{(\Bx,t) = (\mathbf{0}, t_0)} \Big( \int_{T_1}^{T_2} \int _{D_{\varepsilon}} \mu^{-\frac{1}{2}}(sd_2 - \Bi \sqrt{s^2 + \lambda}d_1)  u_0 \mathrm{d} \Bx \mathrm{d} t - \int_{T_1}^{T_2} \int_{{\Omega}_{\varepsilon}} e^{\lambda t} \\
       & \times (e^{\mu^{-\frac{1}{2}}sd_3 \gamma(\varepsilon^l)}e^{\mu^{-\frac{1}{2}}(sd_2 - \Bi \sqrt{s^2 + \lambda}d_1)\varepsilon^l} -1)  \, e^{\mu^{-\frac{1}{2}}(sd_1 + \Bi \sqrt{s^2 + \lambda}d_2)x_1 +\mu^{-\frac{1}{2}}sd_3 x_3} \mathrm{d} x_3 \mathrm{d}x_1 \mathrm{d} t \Big)\bigg|\\
       =&\left|(f_2^{\prime} -g_2^{\prime})|_{(\Bx,t) = (\mathbf{0}, t_0)}\right|\bigg|\int_{T_1}^{T_2} \int_{0}^{\varepsilon^l} \int_{\Omega_{\varepsilon}} \mu^{-\frac{1}{2}}(sd_2 - \Bi \sqrt{s^2 + \lambda}d_1) e^{\lambda t} \\ 
       &\times e^{\mu^{-\frac{1}{2}}(sd_1 + \Bi \sqrt{s^2 + \lambda}d_2)x_1}\, e^{\mu^{-\frac{1}{2}}(sd_2 - \Bi \sqrt{s^2 + \lambda}d_1)x_2} e^{\mu^{-\frac{1}{2}}sd_3(x_3+\gamma(x_2))} \mathrm{d} x_3 \mathrm{d} x_1 \mathrm{d}x_2 \mathrm{d} t\\
       &- \int_{T_1}^{T_2}\int_{0}^{\varepsilon^l} \int_{{\Omega}_{\varepsilon}} \mu^{-\frac{1}{2}}(sd_2 - \Bi \sqrt{s^2 + \lambda}d_1) e^{\mu^{-\frac{1}{2}}sd_3 \gamma(\varepsilon^l)}e^{\mu^{-\frac{1}{2}}(sd_2 - \Bi \sqrt{s^2 + \lambda}d_1)x_2}  e^{\lambda t} \\
       &\quad \times  e^{\mu^{-\frac{1}{2}}(sd_1 + \Bi \sqrt{s^2 + \lambda}d_2)x_1 +\mu^{-\frac{1}{2}}sd_3 x_3} \mathrm{d} x_3 \mathrm{d}x_1 \mathrm{d}x_2 \mathrm{d} t \bigg|\\
       =&\left|(f_2^{\prime} -g_2^{\prime})|_{(\Bx,t) = (\mathbf{0}, t_0)}\right| \bigg|\int_{T_1}^{T_2} \int_{0}^{\varepsilon^l} \int_{\Omega_{\varepsilon}} (e^{\mu^{-\frac{1}{2}}sd_3\gamma(x_2)} -e^{\mu^{-\frac{1}{2}}sd_3\gamma(\varepsilon^l)})\mu^{-\frac{1}{2}}e^{\lambda t} \mathrm{d} t\\
       & \times (sd_2 - \Bi \sqrt{s^2 + \lambda}d_1) e^{\mu^{-\frac{1}{2}}(sd_2 - \Bi \sqrt{s^2 + \lambda}d_1)x_2}  e^{\mu^{-\frac{1}{2}}(sd_1 + \Bi \sqrt{s^2 + \lambda}d_2)x_1 +\mu^{-\frac{1}{2}}sd_3 x_3} \mathrm{d} x_3 \mathrm{d}x_1 \mathrm{d}x_2 \bigg|.
    \end{align*}
Since $s \varepsilon \ll 1$, it follows from \eqref{eq:gamm2} that $s \gamma(x_2) \ll 1$ for any $x_2 \in [0, \varepsilon^l]$. Using this, we obtain
\[
|e^{\mu^{-\frac{1}{2}} s d_3 \gamma(x_2)} -e^{\mu^{-\frac{1}{2}} s d_3 \gamma(\varepsilon^l)}| \leq \mu^{-\frac{1}{2}} s \varepsilon + \mathcal{O}(s^2 \varepsilon^2).
\]
By direct calculation, we obtain
    \begin{align*}
        &\bigg |\int_{T_1}^{T_2} \int_{0}^{\varepsilon^l} \int_{\Omega_{\varepsilon}} \mu^{-\frac{1}{2}}(sd_2 - \Bi \sqrt{s^2 + \lambda}d_1) e^{\mu^{-\frac{1}{2}}(sd_2 - \Bi \sqrt{s^2 + \lambda}d_1)x_2}  e^{\lambda t} \\
        & \times e^{\mu^{-\frac{1}{2}}(sd_1 + \Bi \sqrt{s^2 + \lambda}d_2)x_1 +\mu^{-\frac{1}{2}}sd_3 x_3} \mathrm{d} x_3 \mathrm{d}x_1 \mathrm{d}x_2 \mathrm{d} t\bigg|\\
       \leq & \left |\frac{\mu}{s^2\lambda d_1d_3} e^{\lambda T_1}(e^{\lambda \varepsilon}-1)(e^{\mu^{-\frac{1}{2}}sd_1\varepsilon}-1)(e^{\mu^{-\frac{1}{2}}sd_2\varepsilon^l} -1)(e^{\mu^{-\frac{1}{2}}sd_3\varepsilon}-1)\right|\\
       \leq &C(\mu)s \varepsilon^{3+l} + \mathcal{O}(s^2 \varepsilon^{3+l+l_0}).
    \end{align*}
Therefore, we derive that $|I_{41}-I_{21}| \leq C(\mu,C_3)s^2 \varepsilon^{4+l} + \mathcal{O}(s^3 \varepsilon^{4+l+l_0})$.

For $I_{41}$ and $I_{45}$ defined in \eqref{I-4-3D}, the mean value theorem and H\"older expansion yield the following estimate:
\begin{equation*}
    \begin{aligned}
        |I_{41} + I_{45}| 
        &= \Bigl| \bigl((f_1' - g_1')|_{(\Bx,t)=(\Bx_0,t_0)} (sd_1 + \Bi \sqrt{s^2+\lambda}d_2) \\
        &\quad + (f_3' - g_3')|_{(\Bx,t)=(\Bx_0,t_0)} sd_3 \bigr) \Bigr|
     \left|\int_{T_1}^{T_2} \int_{D_{\varepsilon}} \mu^{-\frac{1}{2}} u_0  \mathrm{d}\Bx \mathrm{d}t\right| \\
        &\geq \Bigl| |(f_1' - g_1')|_{(\Bx,t)=(\Bx_0,t_0)} sd_1| \\
        &\quad - |(f_3' - g_3')|_{(\Bx,t)=(\Bx_0,t_0)} sd_3| \Bigr|
         \left|\int_{T_1}^{T_2} \int_{D_{\varepsilon}} \mu^{-\frac{1}{2}} u_0  \mathrm{d}\Bx \mathrm{d}t\right|,
    \end{aligned}
\end{equation*}
where
    \begin{align*}
        &\left|\int_{T_1}^{T_2} \int_{D_{\varepsilon}} \mu^{-\frac{1}{2}} u_0  \mathrm{d}\Bx \mathrm{d}t\right| \\
        =& \begin{multlined}[t]
            \left| \int_{T_1}^{T_2} \int_{0}^{\varepsilon^{l}} \int_{\Omega_{\varepsilon}} \mu^{-\frac{1}{2}} e^{\lambda t}  e^{\mu^{-\frac{1}{2}}(sd_1 + \Bi \sqrt{s^2 + \lambda}d_2)x_1} \right. \\
            \left. \times e^{\mu^{-\frac{1}{2}}(sd_2 - \Bi \sqrt{s^2 + \lambda}d_1)x_2 + \mu^{-\frac{1}{2}}sd_3(x_3+\gamma(x_2))}  \mathrm{d}x_3 \mathrm{d}x_1 \mathrm{d}x_2 \mathrm{d}t \right|
        \end{multlined} \\
        =& \begin{multlined}[t]
            \left| \int_{T_1}^{T_2} e^{\lambda t}  \mathrm{d}t \int_{0}^{\varepsilon^{l}} \mu^{-\frac{1}{2}} e^{\mu^{-\frac{1}{2}}(sd_2 - \Bi \sqrt{s^2 + \lambda}d_1)x_2} e^{\mu^{-\frac{1}{2}}sd_3\gamma(x_2)}  \mathrm{d}x_2 \right. \\
            \left. \times \int_{\Omega_{\varepsilon}} e^{\mu^{-\frac{1}{2}}(sd_1 + \Bi \sqrt{s^2 + \lambda}d_2)x_1 + \mu^{-\frac{1}{2}}sd_3x_3}  \mathrm{d}x_3 \mathrm{d}x_1 \right|
        \end{multlined} \\
        \geq& \Big|\frac{1}{\lambda sd_3} (e^{\mu^{-\frac{1}{2}}sd_3\varepsilon}-1) e^{\lambda T_1}(e^{\lambda \varepsilon}-1) \Big|
            \left| \Re \int_{0}^{\varepsilon^{l}} e^{\mu^{-\frac{1}{2}}(sd_2 - \Bi \sqrt{s^2 + \lambda}d_1)x_2} e^{\mu^{-\frac{1}{2}}sd_3\gamma(x_2)}  \mathrm{d}x_2 \right| \\
        &\times \left| \Re \int_{E} e^{\mu^{-\frac{1}{2}}(sd_1 + \Bi \sqrt{s^2 + \lambda}d_2)x_1}  \mathrm{d}x_1 \right| \\
        \geq& \begin{multlined}[t]
            \frac{\mu}{\lambda s^2 d_2 d_3} (e^{\mu^{-\frac{1}{2}}sd_3\varepsilon}-1) e^{\lambda T_1}(e^{\lambda \varepsilon}-1) 
            e^{\mu^{-\frac{1}{2}}sd_3\gamma(x_{2,\xi})} (e^{\mu^{-\frac{1}{2}}sd_2\varepsilon^l}-1) \\
            \times \cos((\sqrt{s^2 + \lambda}d_1)x_{2,\xi}) e^{\mu^{-\frac{1}{2}}sd_1 x_{1,\xi}} 
            \cos((\sqrt{s^2 + \lambda}d_2)x_{1,\xi}) |E|
        \end{multlined} \\
        \ge & (1 - e^{\mu^{-\frac{1}{2}}d_2}) e^{\lambda T_1}  e^{\mu^{-\frac{1}{2}}sd_1 x_{1,\xi}}  e^{\mu^{-\frac{1}{2}}sd_3\gamma(x_{2,\xi})} \varepsilon^{3+l} + \mathcal{O}(s\varepsilon^{3+l+l_0}),
    \end{align*}
where $x_{2,\xi} \in (0,\varepsilon^l)$ and $x_{1,\xi} \in E$. Similar to \eqref{eq:I42-2}, we can further obtain that
    \begin{align*}
        &|I_{42}|=\left|\int_{T_1}^{T_2} \int _{D_{\varepsilon}} \delta_1 (f_1^{\prime} -g_1^{\prime})\mu^{-\frac{1}{2}} (sd_1 + \Bi \sqrt{s^2 + \lambda}d_2)  u_0 \mathrm{d} \Bx \mathrm{d} t\right|\\
        \leq & C(C_3,C_4) \left|\int_{T_1}^{T_2} \int _{D_{\varepsilon}} \mu^{-\frac{1}{2}} (sd_1 + \Bi \sqrt{s^2 + \lambda}d_2)  u_0 (|\Bx -\Bx_0|+ |t-t_0|^{\frac{1}{2}})^{\alpha_3 \alpha_4} \mathrm{d} \Bx \mathrm{d} t\right|\\
        \leq & C(C_3,C_4) \varepsilon^{\alpha_3 \alpha_4 l_0}\left  |\int_{T_1}^{T_2} \int _{D_{\varepsilon}} \mu^{-\frac{1}{2}} (sd_1 + \Bi \sqrt{s^2 + \lambda}d_2)  u_0  \mathrm{d} \Bx \mathrm{d} t\right|\\
         \leq& \frac{ C(C_3,C_4) }{|\mu^{-1} \lambda s^2 d_2 d_3|} \varepsilon^{\alpha_3 \alpha_4 l_0} e^{\lambda T_1+\mu^{-1}sd_3 \gamma(\xi_6)} (e^{\lambda \varepsilon}-1)\\
         &\times (e^{\mu^{-\frac{1}{2}}sd_1\varepsilon}-1)(e^{\mu^{-\frac{1}{2}}sd_2\varepsilon^l}-1)(e^{\mu^{-\frac{1}{2}}sd_3\varepsilon}-1)\\  
       % \leq & |\frac{C(C_3,C_4)}{s^2\lambda d_1d_3} e^{\lambda T_1}(\sqrt{\varepsilon^2 + \varepsilon^{2l}} + \varepsilon)^{\alpha_3 \alpha_4}(e^{\lambda \varepsilon}-1)(e^{\mu^{-\frac{1}{2}}sd_1\varepsilon}-1)\\  
       % & \times (e^{\mu^{-\frac{1}{2}}sd_2\varepsilon^l} -1)(e^{\mu^{-\frac{1}{2}}sd_3\varepsilon}-1)| \\
         =& C(\mu, C_3,C_4) e^{\lambda T_1} s \varepsilon^{3+l+\alpha_3 \alpha_4 l_0} + \mathcal{O}(s^2 \varepsilon^{3+l+l_0+\alpha_3 \alpha_4 l_0}).
    \end{align*}
    Since the proofs for $I_{44}$ and $I_{46}$ proceed along similar lines to that of $I_{42}$, we omit the details.
\end{proof}

\begin{Lemma}\label{lemma:I5-3D}
    For the $I_5$  defined in \eqref{I1-6-3D}, we have the following estimate
    \begin{equation*}\label{eq:I5-TE}
       | I_5 |\leq  C(\mu,C_1,C_4)\varepsilon^{3+l+\alpha_1 \alpha_4 l_0}+  \mathcal{O}(s\varepsilon^{3+l+l_0+\alpha_1 \alpha_4 l_0}).
    \end{equation*}
\end{Lemma}

\begin{proof}
Similar to the lemma \ref{lemma:I5}, we deduce
    \begin{align*}
        |I_5| =& \left| \int_{D_{\varepsilon}}(H(u)-H(v))u_0|_{T_1}^{T_2} \mathrm{d} \Bx\right|\\
        \leq & C(C_1 C_4) \varepsilon^{\alpha_1 \alpha_4 l_0} \frac{e^{\lambda T_1}}{\lambda} \Big|(e^{\lambda \varepsilon}-1) \int_{D_{\varepsilon}} e^{\mu^{-\frac{1}{2}}(sd_1x_1+sd_2x_2 +sd_3x_3)} \mathrm{d} \Bx \Big| \\
        \leq & \bigg|\frac{C_1 C_4^{\alpha_1}\mu^{\frac{3}{2}}}{s^3d_1d_2d_3}  e^{\lambda T_1}e^{\mu^{-\frac{1}{2}}sd_3\gamma(x_{2,\xi})} \varepsilon^{\alpha_1 \alpha_4 l_0} (e^{\lambda \varepsilon}-1)(e^{\mu^{-\frac{1}{2}}sd_1\varepsilon}-1)\\
        &\,\,\times (e^{\mu^{-\frac{1}{2}}sd_2\varepsilon^{l}}-1)(e^{\mu^{-\frac{1}{2}}sd_3\varepsilon}-1)\bigg|\\
        \leq &C(\mu,C_1,C_4) \varepsilon^{3+l+\alpha_1 \alpha_4 l_0}+  \mathcal{O}(s\varepsilon^{3+l+l_0+\alpha_1 \alpha_4 l_0}).
    \end{align*}
\end{proof}
Combining the analysis of Lemma \ref{lemma:I5-3D} with Lemma \ref{lemma:I6}, we have the following lemma:
\begin{Lemma}\label{lemma:I6-3D}
    For the $I_6$ defined in \eqref{I1-6-3D}, we have the following estimate:
    \begin{equation}\label{eq:I6-TE}
        |I_6| \leq C(C_1,C_4)  \varepsilon^{3 + l +\alpha_1 \alpha_4 l_0} + \mathcal{O}(s\varepsilon^{3+l+l_0+\alpha_1 \alpha_4 l_0}),
    \end{equation}
    where $l_0$ is defined in \eqref{eq:l0}.
\end{Lemma}

% We now prove Theorem \ref{main:thm} for the three-dimensional case.  To ensure notational consistency, we set $u^1 := u$, $u^2 := v$, $f^1 := f$, $f^2 := g$, $f_k^1 := f_k$, and $f_k^2 := g_k$ for $k = 1, 2,3$ in Theorem \ref{main:thm}. Note that $\BF^{\prime} = (f_1^{\prime}, f_2^{\prime},f_3^{\prime})^{\top}$ is obtained by the coordinate transformation of $\BF = (f_1, f_2,f_3)^{\top}$ in \eqref{eq:3FG}. For convenience, we will use the original variables $u,\,v,\,f,\,g,\,f_k,\,g_k$ from Lemma \ref{lem:w} in proving Theorem \ref{main:thm}.

\begin{proof}[\textbf{Proof of Theorem \ref{main:thm} in nozzle domains}]
We first prove part (a) of Theorem \ref{main:thm}. Since \(\mathcal{M}_{\mathscr{F}^1, \mathfrak{f}^1} = \mathcal{M}_{\mathscr{F}^2, \mathfrak{f}^2}\) 
and \(u^j\) is the solution to \eqref{eq:main-mu} with the configuration \((H, \mathbf{F}^j, f^j)\) 
for \(j=1,2\), the pair \((u^1, u^2)\) solves the coupled system \eqref{eq:uv} with 
\(\mathbf{F}' = \mathbf{F}^{1\prime}\), \(\mathbf{G}' = \mathbf{F}^{2\prime}\), \(f = f^1\), \(g = f^2\), 
and \(h(\mathbf{x}, t) = h_{\mathbf{F}^1} - h_{\mathbf{F}^2}\) , where $\mathbf{F}^{1\prime}$ and $\mathbf{F}^{2\prime}$ are as defined in \eqref{eq:3FG}.

   From Lemma \ref{lemma:I1-6}, we obtain the identity
\begin{equation}\notag
I_1 + I_2 = I_3 + I_4 - I_5 - I_6
\end{equation}
for the quantities $I_j$ $(j = 1, \ldots, 6)$ defined in \eqref{I1-6-3D}. Based on the above analysis, we have
%\begin{equation}
%    I_{41} + I_{45} + I_{41} - I_{21} = I_1 + I_{22} - I_{3} - I_{42} - I_{44} - I_{46} + I_{5} + I_6.
%\end{equation}
%Therefore, we can get that
\begin{equation*}
    |I_{41} + I_{45}| \leq |I_1| + |I_{22}| + |I_{3}| + |I_{42}| + |I_{44}| + |I_{46}| + |I_{43} - I_{21}| + |I_{5}| +|I_6|.
\end{equation*}
Combining it with Lemma \ref{lemma:I1-3D}--Lemma \ref{lemma:I6-3D}, we obtain that
    \begin{align}\label{thm:pt3}
        &\big||{(f_1^{1\prime}-f_1^{2\prime})}|_{(\Bx,t) = (\Bx_0, t_0)} sd_1 -|{(f_3^{1\prime}-f_3^{2\prime})}|_{(\Bx,t) = (\Bx_0, t_0)}sd_3\big| \varepsilon^{3+l}\notag\\
        &\times(1 - e^{\mu^{-\frac{1}{2}}d_2}) e^{\lambda T_1}  e^{\mu^{-\frac{1}{2}}sd_1 x_{1,\xi}}  e^{\mu^{-\frac{1}{2}}sd_3\gamma(x_{2,\xi})} \notag\\
        \leq & C(\mu,\|{h_{\BF^1}-h_{\BF^2}}\|_{L^{\infty}},C_1,C_2,C_3,C_4) (s^2 \varepsilon^{4+l}+s^2 \varepsilon^{3+l+(1+\alpha_4)l_0} + \varepsilon^{3+\alpha_4 l_0} + \varepsilon^{3+\alpha_3 \alpha_4l_0}\notag\\
       & \quad  +\varepsilon^{3+ l} + s \varepsilon^{3+l+\alpha_3 \alpha_4l_0} + \varepsilon^{3+l+\alpha_1 \alpha_4l_0} + \varepsilon^{3+l+\alpha_2 \alpha_4l_0})+ \Oh(s^3 \varepsilon^{4+2l})+\Oh(s^3 \varepsilon^{4+l+l_0})\notag\\
        & \quad  + \Oh(s \varepsilon^{4+\alpha_4 l_0})+\Oh(s \varepsilon^{3+l+\alpha_3 \alpha_4 l_0}) + \Oh(s^2 \varepsilon^{3+l+l_0+\alpha_3 \alpha_4 l_0})+\Oh(s \varepsilon^{3+l+l_0+\alpha_1 \alpha_4 l_0}) ,
    \end{align}
where $C(\mu,\|{h}\|_{L^{\infty}},C_1,C_2,C_3,C_4)$ is related to $\mu$, $\|{h}\|_{L^{\infty}} $, $C_1$, $C_2$, $C_3$ and $C_4$.

From the definition of $d_1$ and $d_3$ in \eqref{cond:d-3D}, and taking the same $s$ in \eqref{eq:beta} as in the two-dimensional case, we obtain:
\begin{equation*}
    \begin{aligned}
        |{f_1^{1\prime}(\Bx_0,t_0,u^1(\Bx_0,t_0)) -f_1^{2\prime}(\Bx_0,t_0,u^2(\Bx_0,t_0))}| \leq & C(\mu,\|{h}\|_{L^{\infty}}\|_{L^{\infty}},C_1,C_2,C_3,C_4) \varepsilon^{\tau }, \\
        |{f_3^{1\prime}(\Bx_0,t_0,u^1(\Bx_0,t_0)) -f_3^{2\prime}(\Bx_0,t_0,u^2(\Bx_0,t_0))}| \leq & C(\mu,\|{h}\|_{L^{\infty}},C_1,C_2,C_3,C_4)\varepsilon^{\tau },
    \end{aligned}
\end{equation*}
where the value of $\tau$ in \eqref{eq:tau} is consistent with the two-dimensional case. Similar to the analysis in the two-dimensional case, we can obtain
\begin{align*}
   |{\mathscr{F}_1^{1'}(u^j)(\Bx_0,t_0) - \mathscr{F}_1^{2'}(u^j)(\Bx_0,t_0)}| \leq & C(\mu,\|h\|_{L^{\infty}},C_1,C_2,C_3,C_4) \varepsilon^{\tau },\\
   | \mathscr{F}_3^{1'}(u^j)(\Bx_0,t_0) - \mathscr{F}_3^{2'}(u^j)(\Bx_0,t_0)| \leq & C(\mu,\|h\|_{L^{\infty}},C_1,C_2,C_3,C_4) \varepsilon^{\tau },
\end{align*}
where $j=1, 2$.

As for the proof of Theorem \ref{main:thm}(b), since the proof process is similar to that of the two-dimensional case, we omit it here for the sake of brevity.
\end{proof}

\subsection{Slab~ domains}

Recall that $\Omega_{\varepsilon}$ is a square with a side length of $\varepsilon$, as mentioned in subsection \ref{sub:geometric} . Thus, the boundary of ${D}_{\varepsilon}$ is given by

\begin{equation*}
    \partial {D}_{\varepsilon}={\Gamma}_{V} \cup {\Omega}_{\varepsilon} \cup {\Omega}_{\varepsilon}^{\prime} \cup {\Gamma}_f \cup {\Gamma}_b,
\end{equation*}
where $\Gamma_V$ denotes the upper and lower surfaces of $D_{\varepsilon}$, and 
\begin{equation*}
    \begin{aligned}
        & {\Omega}_{\varepsilon}=\left\{(x_1, x_2, x_3) \mid x_2=0, x_1 \in(0, \varepsilon), x_3 \in(\gamma(0), \gamma(0  )+ \varepsilon)\right\}, \\
        & {\Omega}_{\varepsilon}^{\prime}=\left\{(x_1, x_2, x_3) \mid x_2=\varepsilon^{l}, x_1 \in(0, \varepsilon), x_3 \in(\gamma(\varepsilon^l ), \gamma(\varepsilon^l)+ \varepsilon)\right\}, \\
        & {\Gamma}_f=\left\{(x_1, x_2, x_3) \mid x_1=\varepsilon, x_2 \in(0, \varepsilon^{l},) x_3 \in(\gamma(x_2), \gamma(x_2)+\varepsilon)\right\}, \\
        & {\Gamma}_b=\left\{(x_1, x_2, x_3) \mid x_1=0, x_2 \in(0, \varepsilon^{l},) x_3 \in(\gamma(x_2), \gamma(x_2)+\varepsilon)\right\}.\\
        \end{aligned}
\end{equation*}

\begin{figure}[htbp]
    \centering
    \begin{subfigure}[b]{0.45\textwidth}
        \centering
        \includegraphics[width=\linewidth]{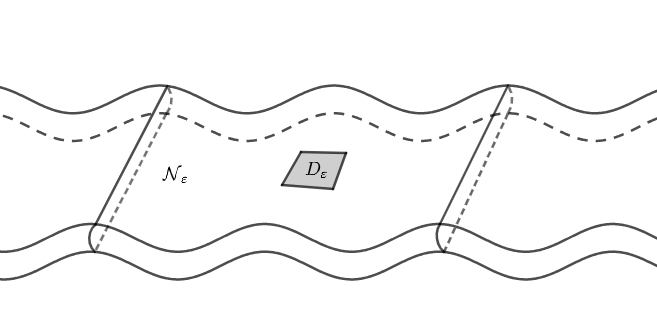}
        \caption{slab~domain}
        %\label{fig:image1}
    \end{subfigure}
    \hfill
    \begin{subfigure}[b]{0.45\textwidth}
        \centering
        \includegraphics[width=\linewidth]{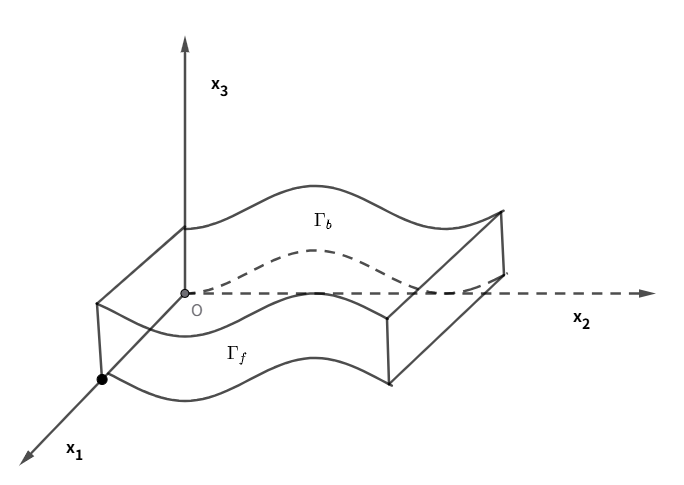}
        \caption{$D_\varepsilon $}
        %\label{fig:image2}
    \end{subfigure}
    \caption{illustration of the geometry in 3D}
    \label{fg:03}
\end{figure}

Compared to the analysis for nozzle domains in Subsection \ref{sub:nozzle}, we need to specifically analyze the volume integral $I_4$ defined in Lemma \ref{lemma:I4-3D} and the boundary integrals on $\Gamma_f$ and $\Gamma_b$. The estimates for all other terms remain the same as in the nozzle-domain case. Define 
\begin{equation}\label{I7-8}
       I_7 =\int_{T_1}^{T_2} \int_{\Gamma_f \cup \Gamma_b} \mu (w \partial_{\nu} u_0 - u_0 \partial_{\nu} w )\mathrm{d} \sigma \mathrm{d} t,\,\,\,
       I_8 =\int_{T_1}^{T_2} \int_{\Gamma_f \cup \Gamma_b}\nu \cdot (\BF^{\prime} -\mathbf{G}^{\prime}) u_0\mathrm{d} \sigma \mathrm{d} t.
\end{equation}

\begin{Lemma}\label{lemma:I7-3D}
    For boundary term $I_7$ defined in \eqref{I7-8}, we have the following estimate: 
    \begin{equation}\label{eq:I7-3D}
     |I_7|\leq C\left(\mu, \|h\|_{L^{\infty},} C_4\right)\left(s^2 \varepsilon^{4+l}+s^2 \varepsilon^{3+l+\left(1+\alpha_4\right) l_0}+\varepsilon^{3+\alpha_4 l_0}\right)+\mathcal{O}\left(s^3 \varepsilon^{4+2 l}\right)+\mathcal{O}\left(s \varepsilon^{4+\alpha_4 l_0}\right).
    \end{equation}
 \end{Lemma}
 \begin{proof}
On $\Gamma_b$, analogous to Lemma \ref{lemma:I1-3D}, we have
     \begin{equation*}
         \begin{aligned}
             w(\Bx,t) =& w(\mathbf{0},t_0) + (0,x_2,x_3) \cdot (\partial_1 w(\mathbf{0},t_0),\partial_2 w(\mathbf{0},t_0),\partial_3 w(\mathbf{0},t_0))\\ &+\partial_t w(\mathbf{0},t_0)) (t-t_0)+ \mathcal{R}_{w|{\Gamma}_{b}} \\
             =& - h(\mathbf{0},u,v)_{(\Bx,t)=(\mathbf{0},t_0)}x_3 + \mathcal{R}_{w|{\Gamma}_{b}},\\
             \partial_{\nu} w(\Bx,t) =&(-1,0,0)\cdot (\partial_1 w(\Bx,t),\partial_2 w(\Bx,t),\partial_3 w(\Bx,t)) = - \partial_1 w(\Bx,t) \\
             =&- \partial_1 w(\mathbf{0},t_0) + \mathcal{R}_{\partial w|{\Gamma}_{b}} =\mathcal{R}_{\partial w|{\Gamma}_{b}}, \\
             \partial_{\nu} u_0 =&- \mu^{-\frac{1}{2}} (sd_1 + \Bi \sqrt{s^2 + \lambda}d_2) e^{\lambda t} \cdot e^{\mu^{-\frac{1}{2}}(sd_2 - \Bi \sqrt{s^2 + \lambda}d_1)x_2 + \mu^{-\frac{1}{2}} s d_3x_3 },
         \end{aligned}
     \end{equation*}
where
\begin{align}
    |\mathcal{R}_{w|{\Gamma}_{b}}| &\leq C_4 (|(0,x_2,x_3)|+ |t-t_0|^{\frac{1}{2}})^{1+\alpha_4}\leq C_4 \varepsilon ^{1+\alpha_4},\label{eq:nb}\\
    |\mathcal{R}_{\partial w|{\Gamma}_{b}}| &\leq C_4 (|(0,x_2,x_3)|+ |t-t_0|^{\frac{1}{2}}) ^{\alpha_4} \leq C_4 \varepsilon ^{\alpha_4}.\label{eq:nb1}
\end{align}
Analogously, on $\Gamma_f$, the following expansion holds:
     \begin{equation*}
         \begin{aligned}
            w(\Bx,t) =& w(\mathbf{0},t_0) + (0,x_2,x_3) \cdot (\partial_1 w(\mathbf{0},t_0),\partial_2 w(\mathbf{0},t_0),\partial_3 w(\mathbf{0},t_0))\\ & + \partial_t w(\mathbf{0},t_0)) (t-t_0)+ \mathcal{R}_{w|{\Gamma}_{f}} \\
             =&  -h(\mathbf{0},u,v)_{(\Bx,t)=(\mathbf{0},t_0)}x_3 + \mathcal{R}_{w|{\Gamma}_{f}},\\
             \partial_{\nu} w(\Bx,t) =&(1,0,0)\cdot (\partial_1 w(\Bx,t),\partial_2 w(\Bx,t),\partial_3 w(\Bx,t)) =  \partial_1 w(\Bx,t) \\
             =& \partial_1 w(\mathbf{0},t_0) + \mathcal{R}_{\partial w|{\Gamma}_{f}} =\mathcal{R}_{\partial w|{\Gamma}_{f}}, \\
             \partial_{\nu} u_0 =&   \mu^{-\frac{1}{2}} (sd_1 + \Bi \sqrt{s^2 + \lambda}d_2) e^{\lambda t} e^{(sd_1 + \Bi \sqrt{s^2 + \lambda})\varepsilon} e^{ \mu^{-\frac{1}{2}} (sd_2 - \Bi \sqrt{s^2 + \lambda}d_1)x_2 + \mu^{-\frac{1}{2}} s d_3 x_3 },
         \end{aligned}
     \end{equation*}
where $\mathcal{R}_{w|{\Gamma}_{f}}$ and $\mathcal{R}_{\partial w|{\Gamma}_{f}}$ satisfy \eqref{eq:nb} and \eqref{eq:nb1}, respectively.
%\begin{align*}
%    |\mathcal{R}_{w|{\Gamma}_{f}}| &\leq C_4 ((x_2^2 +x_3^2+ |t-t_0|) ^\frac{1+\alpha_4}{2}) \leq C_4 \varepsilon ^{1+\alpha_4},\\
%    |\mathcal{R}_{\partial w|{\Gamma}_{f}}| &\leq C_4 ((x_2^2 +x_3^2+ |t-t_0|) ^\frac{\alpha_4}{2}) \leq C_4 \varepsilon ^{\alpha_4}.
%\end{align*}
Therefore, we deduce that 
\begin{align*}
    &\left|\int_{T_1}^{T_2}\int_{\Gamma_f \cup \Gamma_b} w \partial_{\nu} u_0  \mathrm{d} \sigma \mathrm{d} t\right|\\
    &\leq \frac{\mu^{-\frac{1}{2}}\left|h(\mathbf{0}, u, v)_{(\Bx, t)=\left(\mathbf{0}, t_0\right)}\right| e^{\lambda T_1} s^2 \varepsilon^{l+4}}{2}+C_4 \mu^{-\frac{1}{2}} e^{\lambda T_1} s^2 \varepsilon^{3+l+\left(1+\alpha_4\right) l_0}+\mathcal{O}\left(s^3 \varepsilon^{4+2 l}\right),
\end{align*}
     and
     \begin{equation*}
             \left |\int_{T_1}^{T_2}\int_{\Gamma_f \cup \Gamma_b}  u_0 \partial_{\nu}w  \mathrm{d} \sigma \mathrm{d} t\right| 
             \leq C_4 e^{\lambda T_1} \varepsilon^{3+\alpha_4 l_0}+\mathcal{O}\left(s \varepsilon^{4+\alpha_4 l_0}\right).     
     \end{equation*}
     Thus, we establish \eqref{eq:I7-3D} from the preceding analysis.
 \end{proof}

Based on the analysis in Lemmas \ref{lemma:I2-3D} and \ref{lemma:I4-3D}, we establish the following lemma.
 \begin{Lemma}\label{lemma:I8-3D}
    The boundary term $I_8$, defined in \eqref{I7-8}, can be decomposed as:
\begin{equation*}
        I_8 = I_{81} + I_{82},
\end{equation*}
where
\begin{align*}
    I_{81}&:= (f_1^{\prime}-g_1^{\prime})_{(\Bx,t)=(\mathbf{0},t_0)}\left(\int_{T_1}^{T_2}\int_{\Gamma_f} u_0\mathrm{d} \sigma \mathrm{d} t - \int_{T_1}^{T_2}\int_{\Gamma_b} u_0\mathrm{d} \sigma \mathrm{d} t\right), \notag\\
     I_{82}&:= \left( \int_{T_1}^{T_2} \int_{\Gamma_f} \delta_2 (f_1^{\prime}-g_1^{\prime}) u_0\mathrm{d} \sigma \mathrm{d} t - \int_{T_1}^{T_2} \int_{\Gamma_b} \delta_2 (f_1^{\prime}-g_1^{\prime}) u_0\mathrm{d} \sigma \mathrm{d} t\right),
\end{align*}
with
\begin{equation*}
    \begin{aligned}
        I_{81}=&  \int_{T_1}^{T_2} \int_{\Gamma_b}  (e^{\mu^{-\frac{1}{2}}(sd_1 + \Bi \sqrt{s^2 + \lambda}d_2)\varepsilon} -1) e^{\lambda t} e^{\mu^{-\frac{1}{2}}(sd_2 - \Bi \sqrt{s^2 + \lambda}d_1)x_2 +\mu^{-\frac{1}{2}}sd_3 x_3} \mathrm{d} x_3 \mathrm{d}x_1 \mathrm{d} t\\
        &\times {(f_1^{\prime}-g_1^{\prime})_{(\Bx,t)=(\mathbf{0},t_0)}},\\
        |I_{82}| \leq & C(C_3,C_4) e^{\lambda T_1} \varepsilon^{3+\alpha_3 \alpha_4 l_0}+\mathcal{O}\left(s \varepsilon^{3+l+\alpha_3 \alpha_4 l_0}\right).
    \end{aligned}
\end{equation*}
For the term $I_4$ in \eqref{eq:I4-3}, we have
\begin{align*}
     |I_{45}|  \geq & \big||(f_3^{\prime} -g_3^{\prime})|_{(\Bx,t) = (\Bx_0, t_0)}sd_3\big|(1-e^{\mu^{-1/2}d_2}) e^{\lambda T_1} \varepsilon^{3+l}\\ &\times e^{\mu^{-\frac{1}{2}}sd_1 x_{1,\xi}}  e^{\mu^{-\frac{1}{2}}sd_3\gamma(x_{2,\xi})} +\mathcal{O}(s^2\varepsilon^{3+l+l_0}),\\
    |I_{41} - I_{81}| &\leq C(\mu,C_3) e^{\lambda T_1} s^2 \varepsilon^{4+l}+\mathcal{O}\left(s^3 \varepsilon^{4+l+l_0}\right).
\end{align*}
\end{Lemma}

\begin{proof}[\textbf{Proof of Theorem \ref{main:thm} in slab domains}]
We begin by proving part (a) of Theorem \ref{main:thm}. Following an argument analogous to that of Lemma \ref{lemma:I1-6}, we establish the identity
\begin{equation*}
I_1 + I_2 +I_7 +I_8 = I_3 + I_4 - I_5 - I_6,
\end{equation*}
where the terms $I_j$ $(j = 1, \ldots, 8)$ are defined in \eqref{I1-6-3D} and \eqref{I7-8}. Incorporating the analysis for nozzle domains with the results of Lemmas \ref{lemma:I7-3D} and \ref{lemma:I8-3D}, we derive the estimate
\begin{equation*}
|I_{45}| \leq |I_1| + |I_{22}| + |I_{3}| + |I_{42}| + |I_{44}| + |I_{46}| + |I_{43} - I_{21}| + |I_{5}| +|I_6| + |I_{7}| + |I_{41} - I_{81}|+|I_{82}|,
\end{equation*}
where the estimates $I_1$, $I_{22}$, $I_3$, $I_5$, $I_6$, $I_7$ are defined in \eqref{I1-te}, \eqref{eq:I22-TE}, \eqref{eq:I3-TE}, \eqref{eq:I5-TE}, \eqref{eq:I6-TE}, \eqref{eq:I7-3D} respectively, and the remaining estimates are provided in Lemmas \ref{lemma:I4-3D} and \ref{lemma:I8-3D}. Building upon this analysis and following an approach similar to that in \eqref{thm:pt3}, we deduce
\begin{equation*}
|{f_3^{1\prime}(\Bx_0,t_0,u^1(\Bx_0,t_0)) -f_3^{2\prime}(\Bx_0,t_0,u^2(\Bx_0,t_0))}| \leq C(\mu,C_1,C_2,C_3,C_4) \varepsilon^{\tau },
\end{equation*}
where $\tau$ in \eqref{eq:tau} agrees with the two-dimensional case. Following an analogous procedure to the two-dimensional analysis in \eqref{eq:op FG} yields the corresponding estimate for the flux terms:
\begin{equation*}
|{\mathscr{F}_3^{1\prime}(u^j)(\Bx_0,t_0) -\mathscr{F}_3^{2\prime}(u^j)(\Bx_0,t_0)}| \leq C(\mu,C_1,C_2,C_3,C_4) \varepsilon^{\tau },
\end{equation*}
where $j=1, 2$.

The argument for part (b) with slab domains in Theorem \ref{main:thm} parallels that of the two-dimensional case in \eqref{eq:fg 2D} and is therefore omitted.
\end{proof}

\section{Application to Reaction-Diffusion-Convection Systems}\label{sec:application}
To illustrate the broad applicability of our framework, as discussed in Section~\ref{sec:intro}, we turn to a central model in mathematical physics: reaction-diffusion-convection systems. These systems describe the spatiotemporal evolution of scalar fields under coupled transport mechanisms, with applications ranging from microfluidic chemical reactors to ecological population dynamics. The fundamental model takes the form:
\begin{equation}\label{eq:rdc}
    \partial_t u + \nabla \cdot (\boldsymbol{c} u) = \mu \Delta u + R(u, \nabla u) \quad \text{in } \Omega \times [0,T],
\end{equation}
where $u(\Bx,t)$ represents concentration fields such as chemical species or biological populations, $\boldsymbol{c}=(c_k)_{k=1}^3$ denotes an incompressible velocity field, $\mu$ is the diffusion coefficient, and $R(u,\nabla u)$ encodes nonlinear reactions or gradient-dependent processes like chemotaxis. Here, $\Omega$ is a product-type main defined in \eqref{eq:NNN}.

 Within our operator-theoretic framework, the system \eqref{eq:rdc} corresponds to a specific realization of the general balance law \eqref{eq:bl}. This correspondence is established through the identifications:
 \begin{align*}
     H(u) &= u, &
     \BF(\Bx, t, u) &= \boldsymbol{c} u, &
     f(\Bx, t, u, \nabla u, \Delta u) &= R(u, \nabla u)+ \mu \Delta u.
 \end{align*}
This reformulation allows us to interpret the flux $\BF$ and the reaction term $R$ as realizations of the abstract operators $\mathscr{F}$ and $\mathfrak{f}$ introduced in Section \ref{sub:math}. Note that the results of Theorem \ref{main:thm} imply that the physical parameters in system \eqref{eq:rdc} are uniquely identifiable from the boundary data $\mathcal{M}|_\Sigma$, achieving an accuracy on the order of $\mathcal{O}(\varepsilon^{\tau})$ and $\mathcal{O}(\varepsilon^{\tau_1})$ as specified in the following theorem. The proof follows directly from Theorem~\ref{main:thm} and is omitted for brevity.

\begin{thm}%\label{thm:rdc-app}
Let $(u^1,\boldsymbol{c}^1, R^1)$ and $(u^2,\boldsymbol{c}^2, R^2)$ in the admissible class $\mathcal{A}$ be two sets of parameters for the reaction-diffusion-convection system \eqref{eq:rdc} that yield solutions with identical passive boundary measurements on $\Sigma$, i.e., $\mathcal{M}_{\boldsymbol{c}^1,R^1}|_\Sigma = \mathcal{M}_{\boldsymbol{c}^2,R^2}|_\Sigma$. Assume $ R^j, \boldsymbol{c}^j, {u}^j$ satisfying \eqref{Cond:fg0}, \eqref{Cond:FG} and \eqref{Cond:uv}, respectively. Then, we obtain the following estimate for the fluxes:
 \begin{equation}\label{eq:flux_est}
\|\widetilde{c}^1 {u}^1  - \widetilde{c}^2{u}^2  \|_{L^\infty(\Omega)} \leq C(\varepsilon_0, C_2, C_4) \varepsilon^\tau,
 \end{equation}
 where the specific forms of $\widetilde{c}^j$ $(j=1,2)$ are given in Remark \ref{rmk:tilde-fg-forms}, the regularity parameter $\tau > 0$ is as in \eqref{eq:ta}, and $C(\varepsilon_0, C_2, C_4) > 0$ depends only on the parameters $\varepsilon_0$, $C_2$, $C_4$ from Theorem \ref{main:thm}.

 Furthermore, in the scenario where the fluxes are \textit{a priori} known to be identical, i.e., $ \boldsymbol{c}^1 {u}^1 = \boldsymbol{c}^2 {u}^2$,  we obtain the following estimate for the  reaction term:
 \begin{equation}\label{eq:react_est}
 \| R^1(u^1, \nabla u^1) - R^2(u^2, \nabla u^2) \|_{L^\infty(\Omega)} \leq \widetilde{C}(\varepsilon_0, C_2,  C_4) \varepsilon^{\tau_1},
 \end{equation}
 where the regularity parameter $\tau_1 $  and positive constant $\widetilde{C}(\varepsilon_0, C_2, C_4)$ are introduced in Theorem \ref{main:thm}.
 \end{thm}

%\begin{remark}
    
 The regularity parameters $\tau$ and $\tau_1$ in our analysis are not merely mathematical abstractions; they embody the interplay between the system's geometry and the underlying physical laws, quantifying how geometric confinement at the scale $\varepsilon \ll 1$ governs the information transfer from internal dynamics to boundary measurements. This framework reveals that even systems producing identical boundary data can only sustain internal flux and source variations of order $\mathcal{O}(\varepsilon^\tau)$ and $\mathcal{O}(\varepsilon^{\tau_1})$ as given in \eqref{eq:flux_est} and \eqref{eq:react_est}, respectively. This theoretical scaling, conceptually aligned with experimental observations in microfluidics \cite{GHB2018,ISKW}, provides a quantitative principle for experimental design-guiding, for instance, the choice of system size to achieve a desired sensing resolution or the optimal placement of sensors in regions of high dynamic activity to maximize information content.

\vspace{0.4cm}
\noindent\textbf{Acknowledgment.} 
 The work of H. Liu is supported by the Hong Kong RGC General Research Funds (projects 11311122, 11300821, and 11303125), the NSFC/RGC Joint Research Fund (project  N\_CityU101/21), the France-Hong Kong ANR/RGC Joint Research Grant, A-CityU203/19. The work of Q. Meng is fully supported by a fellowship award from the Research Grants Council of the Hong Kong Special Administrative Region, China (Project No. CityU PDFS2324-1S09).

\renewcommand\refname{References}

\end{document}